\documentclass[10pt,aps,prb, preprint,reqno]{amsart}
\usepackage{amsmath}
\usepackage{amssymb}
\usepackage{yfonts}
\usepackage{hyperref}
\usepackage{mathrsfs}
\usepackage {latexsym}
\usepackage{graphics}
\usepackage{txfonts}
\usepackage{txfonts}	
\usepackage{breqn}
\usepackage{cite}
\usepackage{xcolor}
\usepackage{color}
\usepackage[shortlabels]{enumitem}

\newtheorem{theorem}{Theorem}[section]
\newtheorem{remark}{Remark}[section]

\newtheorem{proposition}[theorem]{Proposition}

\newtheorem{define}{Definition}[section]

\newcommand{\joinR}{\hspace{-.1em}}
\newcommand{\RomanI}{I}
\newcommand{\RomanII}{\mbox{\RomanI\joinR\RomanI}}
\newcommand{\RomanIII}{\mbox{\RomanI\joinR\RomanII}}
\newcommand{\RomanIV}{\mbox{\RomanI\joinR\RomanV}}
\newcommand{\RomanV}{V}
\newcommand{\RomanVI}{\mbox{\RomanV\joinR\RomanI}}
\newcommand{\RomanVII}{\mbox{\RomanV\joinR\RomanII}}
\newcommand{\RomanVIII}{\mbox{\RomanV\joinR\RomanIII}}
\newcommand{\RomanIX}{\mbox{\RomanI\joinR\RomanX}}
\newcommand{\RomanX}{X}
\newcommand{\RomanXI}{\mbox{\RomanX\joinR\RomanI}}

\allowdisplaybreaks

\begin{document}
\title[Hall-magnetohydrodynamics system]{Another remark on the global regularity issue of the Hall-magnetohydrodynamics system}

\subjclass[2010]{35B65; 76W05}
 
\author[Rahman]{Mohammad Mahabubur Rahman}
\address{Department of Mathematics and Statistics, Texas Tech University, Lubbock, Texas, 79409-1042, U.S.A.}
\email{mohammad-mahabu.rahman@ttu.edu}
\author[Yamazaki]{Kazuo Yamazaki}
\address{Department of Mathematics and Statistics, Texas Tech University, Lubbock, Texas, 79409-1042, U.S.A.}
\email{kyamazak@ttu.edu} 
\date{}
\keywords{Electron magnetohydrodynamics; Hall-magnetohydrodynamics; regularity criteria.}

\begin{abstract} 
 We discover cancellations upon $H^{2}(\mathbb{R}^{n})$-estimate of the Hall term for $n \in \{2,3\}$. As its consequence, first, we derive a regularity criterion for the 3-dimensional Hall-magnetohydrodynamics system in terms of only horizontal components of velocity and magnetic fields. Second, we prove the global regularity of the $2\frac{1}{2}$-dimensional electron magnetohydrodynamics system with magnetic diffusion $(-\Delta)^{\frac{3}{2}} (b_{1}, b_{2}, 0) + (-\Delta)^{\alpha} (0, 0, b_{3})$ for $\alpha > \frac{1}{2}$. Lastly, we extend this result to the $2\frac{1}{2}$-dimensional Hall-magnetohydrodynamics system with $-\Delta u$ replaced by $(-\Delta)^{\alpha} (u_{1}, u_{2}, 0) -\Delta (0, 0, u_{3})$ for $\alpha > \frac{1}{2}$. The sum of the derivatives in diffusion that our global regularity result requires is $11+ \epsilon$ for any $\epsilon > 0$ while the analogous sum for the classical $2\frac{1}{2}$-dimensional Hall-magnetohydrodynamics system is 12 considering $-\Delta u$ and $-\Delta b$. 
\end{abstract}

\maketitle

\section{Introduction} 
\subsection{Motivation from physics and real-world applications}
Ever since the pioneering work of Alfv$\acute{\mathrm{e}}$n \cite{A42a} 80 years ago, the magnetohydrodynamics (MHD) system concerning electrically conducting fluids has attracted many interest from researchers in a wide array of applied sciences. For example, while the Navier-Stokes (NS) equations is often utilized to study fluid turbulence, the MHD system is the conventional choice to study MHD turbulence that occurs in laboratory settings such as fusion confinement devices (e.g., reversed field pinch), as well as astrophysical systems (e.g., solar corona). The Hall term arises upon writing the current density as the sum of the ohmic current and a Hall current that is perpendicular to the magnetic field (see \cite[Equation (94)]{L60}) and the Hall-MHD system, that consists of the MHD system with an addition of the Hall term, was formally introduced by Lighthill \cite{L60} in 1960. Thereafter, the Hall-MHD system has received much attention from physicists and engineers due to its applicability: the study of the sun \cite{C98}, star formation \cite{W04}, magnetic reconnection \cite{HG05}, and turbulence \cite{MH09}. The Hall-MHD system with zero velocity field informally reduces to the electron MHD system which governs the electron's self-induced magnetic field (see \cite{WH09}). 

Nevertheless, the singularity of the Hall term has disallowed mathematicians to prove some results which are well-known for the NS equations and can be extended in a standard way to the MHD system, two examples of such being the following. 
\begin{enumerate}
\item While the solution to the $2\frac{1}{2}$-dimensional ($2\frac{1}{2}$-D) MHD system starting from a sufficiently smooth initial data preserves its regularity for all time, an analogous problem is open for the Hall-MHD system (e.g., ``Contrary to the usual MHD the global well-posedness in the $2\frac{1}{2}$-dimensional Hall-MHD is wide open'' from \cite[Abstract]{CL14}).
\item While the solution to the MHD system with zero viscous diffusion and zero magnetic diffusion in any dimension has a unique solution locally in time, an analogous problem is open for the Hall-MHD system (e.g., \cite{CWW15a}). 
\end{enumerate} 
The purpose of this manuscript is to present new cancellations within the Hall term upon $H^{2}(\mathbb{R}^{n})$-estimate for both $n \in \{2,3\}$ (see Proposition \ref{Proposition 3.1}). The following is a summary of our findings due to such cancellations, with details of notations to be given subsequently.
\begin{enumerate}[(a)]
\item We obtain a regularity criterion for the 3-dimensional (3-D) Hall-MHD system that relies only on the horizontal components of its solution (see Theorem \ref{Theorem 2.1}). 
\item We prove that global regularity holds for $2\frac{1}{2}$-D electron MHD system as long as the horizontal components of the magnetic vector field have sufficiently strong diffusion of $(-\Delta)^{\frac{3}{2}}$, even if the diffusion on the vertical component of the magnetic vector field is as weak as $(-\Delta)^{\alpha}$ for $\alpha > \frac{1}{2}$ (see Theorem \ref{Theorem 2.2}). We point out that 
\begin{enumerate}[i]
\item the electron MHD system has scaling-invariance property (see \eqref{est 148}), 
\item considering its best-conserved quantity clearly indicates the appropriate exponent $\beta$ of $(-\Delta)^{\beta}$ in its diffusion that makes the equation critical (see \eqref{est 149}), 
\item  and yet we are able to prove global regularity for the equation when one of the components has a significantly weaker diffusion than the critical level  $\beta$ (see Theorem \ref{Theorem 2.2}). 
\end{enumerate}
\item We extend the aforementioned global regularity result of the electron MHD system to the $2\frac{1}{2}$-D Hall-MHD system with $-\Delta u$ and $-\Delta b$ respectively replaced by 
\begin{equation*}
(-\Delta)^{\alpha} (u_{1}, u_{2}, 0) - \Delta(0, 0, u_{3}) \hspace{1mm} \text{ and } \hspace{1mm} (-\Delta)^{\frac{3}{2}} (b_{1}, b_{2}, 0) + (-\Delta)^{\alpha} (0, 0, b_{3}) \hspace{1mm} \text{ for } \alpha > \frac{1}{2} 
\end{equation*}
(see Theorem \ref{Theorem 2.3}). The sum of such derivatives we require is $11 + \epsilon$ for any $\epsilon > 0$ while the analogous sum for the $2\frac{1}{2}$-D Hall-MHD system with $-\Delta u$ and $-\Delta b$ is 12 (see Remark \ref{Remark 2.2} (3)).
\end{enumerate} 

\subsection{Previous works} 
We will work with a spatial domain of $\mathbb{R}^{n}, n \in \{2,3\}$, although much of our discussions can be transferred to $\mathbb{T}^{n}$ via straight-forward modifications. We write $\partial_{t} \triangleq \frac{\partial}{\partial t}, \partial_{j} \triangleq \frac{\partial}{\partial x_{j}}$ for $j \in \{1, \hdots, n\}$, and $A \overset{(\cdot)}{\lesssim} B$ to imply the existence of a constant $C \geq 0$ of no dependence on any important parameter such that $A \leq CB$ due to the equation $(\cdot)$. Let us define $\Lambda^{\alpha} \triangleq (-\Delta)^{\frac{\alpha}{2}}$ for any $\alpha \in \mathbb{R}$ as a Fourier operator with a Fourier symbol of $\lvert \xi \rvert^{\alpha}$ so that $\mathcal{F}(\Lambda^{\alpha} f)(\xi) = \lvert \xi \rvert^{\alpha} \mathcal{F}(f)(\xi)$ where  $\mathcal{F}$ is the Fourier transform. We let $b: \mathbb{R}_{\geq 0} \times \mathbb{R}^{3} \mapsto \mathbb{R}^{3}$ represent the magnetic field,  
\begin{equation}\label{est 4} 
j \triangleq (j_{1}, j_{2}, j_{3}) \triangleq \nabla \times b = (\partial_{2}b_{3} - \partial_{3} b_{2}, -\partial_{1} b_{3} + \partial_{3} b_{1}, \partial_{1}b_{2} - \partial_{2} b_{1})
\end{equation} 
the current density field, $\epsilon \geq 0$ the Hall parameter, and $\eta \geq 0$ the magnetic diffusivity. The electron MHD system consists of 
\begin{equation}\label{est 48}  
\partial_{t} b + \epsilon \nabla\times (j\times b) = \eta \Delta b \hspace{4mm}\text{ for }  t > 0, 
\end{equation} 
starting from the initial data $b^{\text{in}} \triangleq b \rvert_{t=0}$ that is divergence-free so that the divergence-free property is propagated (see \cite[Equation (1)]{WH09}). 

Additionally, with $u: \mathbb{R}_{\geq 0} \times \mathbb{R}^{3} \mapsto \mathbb{R}^{3}$ and $\pi: \mathbb{R}_{\geq 0} \times \mathbb{R}^{3} \mapsto \mathbb{R}$ representing respectively the velocity field and pressure field, as well as $\nu \geq 0$ the viscosity, the 3-D Hall-MHD system reads  
\begin{subequations}\label{est 1}
\begin{align}
& \partial_{t} u + (u\cdot\nabla) u + \nabla \pi = \nu \Delta u + (b\cdot\nabla) b  \hspace{17mm} \text{ for } t > 0,\label{est 1a}\\
& \partial_{t} b + (u\cdot\nabla) b + \epsilon \nabla \times ( j \times b)  = \eta \Delta b + (b\cdot\nabla) u \hspace{4mm} \text{ for }  t > 0, \label{est 1b}\\
& \nabla\cdot u = 0  \hspace{56mm} \text{ for }  t > 0, \label{est 1c} 
\end{align}
\end{subequations} 
starting from initial data $(u^{\text{in}}, b^{\text{in}}) \triangleq (u,b) \rvert_{t=0}$ that are both divergence-free so that $\nabla\cdot b=  0$ is again  propagated through \eqref{est 1b} (e.g., \cite{CL14}). We refer to the Hall-MHD system with $\epsilon  = 0$ as the MHD system and in turn the MHD system with $b \equiv 0$ the NS equations if $\nu > 0$ and the Euler equations if $\nu = 0$. Let us clarify that the $2\frac{1}{2}$-D case of \eqref{est 1} refers to 
\begin{equation*}
u(t,x) = (u_{1}, u_{2}, u_{3})(t, x_{1}, x_{2}) \text{ and } b(t,x) = (b_{1}, b_{2}, b_{3})(t, x_{1}, x_{2})
\end{equation*} 
(e.g., \cite[Section 2.3.1]{MB02} for the $2\frac{1}{2}$-D NS and Euler equations). Physicists such as \cite{DSDCSCM12} relied on such $2\frac{1}{2}$-D Hall-MHD system because in the 2-D case when $b(t,x) = (b_{1}, b_{2})(t,x_{1}, x_{2})$, the Hall term decouples from the rest.

Concerning the mathematical analysis of the Hall-MHD system, Acheritogaray, Degond, Frouvelle, and Liu in \cite{ADFL11} proved the global existence of a weak solution to the 3-D Hall-MHD system \eqref{est 1} in $\mathbb{T}^{3}$ making use of the key identity
\begin{equation}\label{est 3} 
(\Theta \times \Psi) \cdot \Theta = 0 \hspace{5mm} \forall \hspace{1mm} \Theta, \Psi \in \mathbb{R}^{3}
\end{equation} 
to handle the Hall term so that the Hall term makes zero contribution to the energy identity. More fundamental well-posedness results were obtained in \cite{CDL14}. In particular, following the classical regularity criteria of the NS equations (e.g., \cite{S62, P59, ESS03}) and the MHD system (e.g., \cite{HX05, Z05}), Chae and Lee \cite{CL14} obtained various blow-up criteria for the Hall-MHD system with one of them being that for $m \in \mathbb{N}$ such that $m > 1 + \frac{n}{2}$, 
\begin{equation}\label{est 57} 
\limsup_{t\nearrow T^{\ast}} (\lVert u(t) \rVert_{H^{m}}^{2} + \lVert b(t) \rVert_{H^{m}}^{2}) = \infty \text{ if and only if } \int_{0}^{T^{\ast}} (\lVert u \rVert_{BMO}^{2} + \lVert \nabla b \rVert_{BMO}^{2}) dt = \infty 
\end{equation} 
where $T^{\ast} < \infty$ denotes the first blow-up time in the 3-D case; in the $2\frac{1}{2}$-D case, this is relaxed to 
\begin{equation}\label{est 64} 
\limsup_{t\nearrow T^{\ast}} (\lVert u(t) \rVert_{H^{m}}^{2} + \lVert b(t) \rVert_{H^{m}}^{2}) = \infty \text{ if and only if } \int_{0}^{T^{\ast}}   \lVert j \rVert_{BMO}^{2} dt = \infty. 
\end{equation} 
As these criteria indicate, or directly by comparing $(b\cdot\nabla) b$ in \eqref{est 1a} and the Hall term $\nabla \times (j\times b)$ in \eqref{est 1b} that can be written as 
\begin{equation}\label{est 150} 
\nabla \times (j\times b) = \nabla \times \left[ -\nabla (\frac{ \lvert  b \rvert^{2}}{2} ) + (b\cdot\nabla) b \right] = \nabla \times ((b\cdot\nabla) b), 
\end{equation}
the Hall term is informally one derivative more singular than the non-linear terms in the MHD system (see Remark \ref{Remark 1.1}). We also refer to \cite{CS13} for temporal decay,  \cite{CW15, CW16b} for partial regularity results, and \cite{CW16} for singularity formation of the 3-D Hall-MHD system with zero magnetic diffusion. 

In relevance to the regularity criteria of the Hall-MHD system such as \eqref{est 57}-\eqref{est 64}, we recall that the research direction on various component reduction for the NS equations flourished in the past few decades. E.g, Chae and Choe \cite{CC99} in 1999 reduced the well-known Beale-Kato-Majda criterion \cite{BKM84} to two components of the vorticity $\omega \triangleq \nabla \times u$ for the 3-D NS equations. Cao and Titi \cite{CT08}, Kukavica and Ziane \cite{KZ06, KZ07}, Zhou and Pokorn$\acute{\mathrm{y}}$ \cite{ZP10}, and many other works were devoted to reducing the regularity criteria from \cite{S62, P59} to a few components of the velocity $u$ or a few entries of $\nabla u$, all in norms which are not scaling-invariant except \cite{KZ07}. Because the non-linear terms for the MHD system, namely $(u\cdot\nabla)u, (b\cdot\nabla) b, (u\cdot\nabla) b$, and $(b\cdot\nabla) u$, have a similar structure to $(u\cdot\nabla) u$ of the NS equations, many component reduction results were extended from the NS equations to the MHD system (e.g., \cite{CW10, Y14a}). Using anisotropic Littlewood-Paley theory, component reduction results at the scaling-invariant level were obtained by Chemin and Zhang \cite{CZ16} for the NS equations (see \cite{Y16c} in the case of the MHD system). All these component reduction results, one way or another, relied on some cancellations using divergence-free property; e.g., one can find 
\begin{equation}\label{est 98}
\int_{\mathbb{R}^{3}} (u\cdot\nabla) u \cdot \Delta u dx \lesssim \int_{\mathbb{R}^{3}} \lvert \nabla_{h} u \rvert \lvert \nabla u \rvert^{2} dx
\end{equation} 
in \cite[p. 1102]{ZP10} or 
\begin{equation}\label{est 99} 
\int_{\mathbb{R}^{3}} (u\cdot\nabla) u \cdot \sum_{k=1}^{2} \partial_{k}^{2} u dx \lesssim \int_{\mathbb{R}^{3}} \lvert u_{3} \rvert \lvert \nabla  u \rvert \lvert \nabla \nabla_{h} u \rvert dx \hspace{1mm} \text{ where } \hspace{1mm} \nabla_{h} \triangleq (\partial_{1},\partial_{2}, 0)
\end{equation} 
as a consequence \cite[Lemma 2.3]{KZ06} (cf. also horizontal Biot-Savart law identity utilized in \cite{CZ16, Y18a}). As we pointed out already, such component reduction results for the NS equations were successfully extended to the MHD system because their non-linear terms had identical structures; with this in mind, due to the completely distinct singular structure of the Hall term $\nabla \times (j\times b)$, to the best of our knowledge, no significant attempt was made to discover any non-trivial cancellations in the Hall term and study its global well-posedness from anisotropic perspective until very recently; we review these new developments next. 

\subsection{Motivation from previous works}\label{Subsection 1.3}  
Let us denote 
\begin{align*}
f_{h} \triangleq 
\begin{pmatrix}
f_{1} & f_{2} & 0 
\end{pmatrix}^{T} \text{ and } f_{v} \triangleq 
\begin{pmatrix}
0 & 0 & f_{3} 
\end{pmatrix}^{T}
\text{ for any } f= 
\begin{pmatrix}
f_{1} & f_{2} & f_{3}  
\end{pmatrix}^{T}. 
\end{align*}  
In \cite{RY22} we discovered multiple cancellations in the $H^{1}(\mathbb{R}^{2})$-estimate on the $2\frac{1}{2}$-D Hall-MHD system that ultimately resulted in the following inequality: any smooth function $b$ that is divergence-free satisfies  
\begin{equation}\label{est 38} 
\int_{\mathbb{R}^{2}} \nabla\times (j \times b) \cdot \Delta b dx  \lesssim \int_{\mathbb{R}^{2}} \lvert \nabla b \rvert  \lvert \nabla b_{h} \rvert \lvert \nabla^{2} b_{h} \rvert dx  
\end{equation} 
(see \cite[Equation (76)]{RY22}). Consequently, we obtained various component reduction results of regularity criteria for the $2\frac{1}{2}$-D Hall-MHD system, e.g., in terms of $b_{3}$, $j_{3}$, $u_{3}$, as well as $u_{1}$ and $u_{2}$ in \cite[Theorems 2.1, 2.2, 2.3 (1), and 2.3 (2)]{RY22}, respectively. We mention that in the $2\frac{1}{2}$-D case, completely independently some cancellations and component reduction of criteria were also discovered very recently in \cite{BK23, DW22}. 
 
Remarkably, \eqref{est 38} can be extended to the 3D case; in \cite[Equations (95)-(100)]{RY22}, \eqref{est 38} was extended to  
\begin{equation}\label{est 53} 
\int_{\mathbb{R}^{3}} \nabla \times (j\times b) \cdot \Delta b dx \lesssim \int_{\mathbb{R}^{3}} \lvert \nabla b \rvert \lvert \nabla b_{h} \rvert \lvert \nabla^{2} b_{h} \rvert dx. 
\end{equation} 
However, we were unable to obtain any component reduction result of regularity criteria for the 3-D Hall-MHD system using \eqref{est 53} because $H^{1}(\mathbb{R}^{3})$-bound cannot lead to higher regularity in the 3-D case; indeed, due to the embedding 
\begin{equation}\label{est 130}
L_{\text{loc}}^{1} (\mathbb{R}^{n}) \cap \dot{H}^{\frac{n}{2}}(\mathbb{R}^{n}) \hookrightarrow BMO(\mathbb{R}^{n}) \hspace{1mm} \forall \hspace{1mm} n \in\mathbb{N}
\end{equation}
(e.g., \cite[Theorem 1.48]{BCD11}), we see that it suffices to prove $\int_{0}^{T} ( \lVert u \rVert_{\dot{H}^{\frac{3}{2}}}^{2} + \lVert b \rVert_{\dot{H}^{\frac{5}{2}}}^{2}) dt < \infty$ to prevent finite-time blow-up according to \eqref{est 57} while $H^{1}(\mathbb{R}^{3})$-bound only gives $\int_{0}^{T} \lVert \Delta u \rVert_{L^{2}}^{2} + \lVert \Delta b \rVert_{L^{2}}^{2} dt < \infty$ from diffusive terms. Therefore, we desperately need the $H^{2}(\mathbb{R}^{3})$-bound to bootstrap to higher regularity in the 3-D case. In \cite{RY22} we did not think such cancellations can be found in the $H^{2}(\mathbb{R}^{3})$-estimate due to its complexity. We are also not aware of any work that found cancellations in the $H^{2}(\mathbb{R}^{3})$-estimate even for the 3-D NS equations which is much simpler than the Hall-MHD system; e.g., \eqref{est 98}-\eqref{est 99} are for estimates on $\lVert \nabla u \rVert_{L^{2}}^{2}$ and $\lVert \nabla_{h} u \rVert_{L^{2}}^{2}$, respectively. Remarkably, \eqref{est 58}-\eqref{est 100} will show that \eqref{est 38}-\eqref{est 53} can be extended to the $H^{2}(\mathbb{R}^{n})$-estimate for both $n \in \{2,3\}$.  

Second, by relying on the inequality \eqref{est 38}, in \cite[Theorem 2.4]{RY22} we proved the global well-posedness of the 2$\frac{1}{2}$-D Hall-MHD system with magnetic diffusion $-\Delta b$ replaced by $(-\Delta)^{\frac{3}{2}} b_{h} -\Delta b_{v}$ in $H^{3} (\mathbb{R}^{2})$: 
\begin{subequations}\label{est 39}
\begin{align}
& \partial_{t} u + (u\cdot\nabla) u + \nabla \pi = \nu \Delta u + (b\cdot\nabla) b  \hspace{33mm} \text{ for } t > 0,\label{est 39a}\\
& \partial_{t} b + (u\cdot\nabla) b + \epsilon \nabla \times ( j \times b)  + (-\Delta)^{\frac{3}{2}} b_{h} - \Delta b_{v} = (b\cdot\nabla) u \hspace{4mm} \text{ for }  t > 0, \label{est 39b}\\
& \nabla\cdot u = 0  \hspace{71mm} \text{ for }  t > 0, \label{est 39c} 
\end{align}
\end{subequations} 
where we point out that the property of $\nabla\cdot b= 0$ is propagated from $\nabla\cdot b^{\text{in}} = 0$ due to the nature of the $2\frac{1}{2}$-D flow. This improved the previous result \cite[Theorem 2.3]{Y19a} in which the global regularity required $(-\Delta)^{\frac{3}{2}} b$ which was in accordance with the general belief that the Hall term is one more derivative more singular than the MHD system (recall \eqref{est 150}) and thus the global regularity requires the magnetic diffusion $(-\Delta)^{\frac{3}{2}}b$ rather than $-\Delta b$. 
 
\begin{remark}\label{Remark 1.1} 
In relevance, let us point out that the generalized electron MHD system
\begin{equation}\label{est 147} 
\partial_{t} b + \epsilon \nabla \times (j\times b) + (-\Delta)^{\beta} b = 0, 
\end{equation} 
possesses a scaling-invariance property although the Hall-MHD system does not; i.e., if $b(t,x)$ solves \eqref{est 147}, then so does 
\begin{equation}\label{est 148} 
b_{\lambda}(t,x) \triangleq \lambda^{2\beta -2} b(\lambda^{2\beta} t, \lambda x)  \hspace{1mm} \forall\hspace{1mm} \lambda \in \mathbb{R}_{+}. 
\end{equation} 
Considering the fact that the best conserved quantity for the solution to \eqref{est 147} is $\lVert \cdot \rVert_{L_{T}^{\infty} L_{x}^{2} \cap L_{T}^{2} \dot{H}_{x}^{\beta}}$ and e.g., 
\begin{equation}\label{est 149}
\lVert b_{\lambda} \rVert_{L_{T}^{\infty} L_{x}^{2}}^{2}  =  \lambda^{4\beta - 4 - n} \lVert b \rVert_{L_{\lambda^{2\beta} T}^{\infty} L_{x}^{2}}^{2}, 
\end{equation}
we clearly see that the critical threshold is $\beta = 1 + \frac{n}{4}$, in particular, $\beta \geq \frac{3}{2}$ in case $n = 2$.
\end{remark} 

\begin{remark}
At this point, we point out another surprising feature of the inequalities \eqref{est 38}-\eqref{est 53}. We observe that both bounds in \eqref{est 38}-\eqref{est 53} separated $b_{h}$ twice rather than just once, specifically $\lvert \nabla b_{h} \rvert \lvert \nabla^{2} b_{h} \rvert$. In contrast, \eqref{est 98} did not separate $u_{3}$ or $u_{h}$;  \eqref{est 99} separated $u_{3}$ but only once and that bound on $\int_{\mathbb{R}^{3}} (u\cdot\nabla) u \cdot \sum_{k=1}^{2} \partial_{k}^{2} u dx$ is not for an $\lVert \nabla u \rVert_{L^{2}}^{2}$-estimate but only for $\lVert \nabla_{h} u \rVert_{L^{2}}^{2}$-estimate.   

Typically when one wishes to obtain a component reduction result in terms of ``$X$,'' it suffices to separate it only once; e.g., in order to deduce a criterion in terms of $\nabla b_{h}$ for the $2\frac{1}{2}$-D Hall-MHD system, one can immediately estimate via H$\ddot{o}$lder's, Gagliardo-Nirenberg, and Young's inequalities 
\begin{align*}
\int_{\mathbb{R}^{2}} \nabla\times (j \times b) \cdot \Delta b dx  \overset{\eqref{est 38}}{\lesssim}& \int_{\mathbb{R}^{2}} \lvert \nabla b \rvert  \lvert \nabla b_{h} \rvert \lvert \nabla^{2} b_{h} \rvert dx \lesssim \lVert \nabla b \rVert_{L^{\frac{2p}{p-2}}} \lVert \nabla b_{h} \rVert_{L^{p}} \lVert \nabla^{2} b_{h} \rVert_{L^{2}}\\
\lesssim& \lVert \nabla b \rVert_{L^{2}}^{\frac{p-2}{p}} \lVert \Delta b \rVert_{L^{2}}^{\frac{2}{p} + 1} \lVert \nabla b_{h} \rVert_{L^{p}}  \leq \frac{\eta}{2} \lVert \Delta b \rVert_{L^{2}}^{2}+ C \lVert \nabla b \rVert_{L^{2}}^{2} \lVert \nabla b_{h} \rVert_{L^{p}}^{\frac{2p}{p-2}} 
\end{align*}
for any $p\in (2,\infty]$ in which we simply bounded $\lVert \nabla^{2} b_{h} \rVert_{L^{2}} \lesssim \lVert \Delta b \rVert_{L^{2}}$. However, in pursuit of the global well-posedness of the system \eqref{est 39}, we realize that taking full advantage of \eqref{est 38} shows that one does not need $(-\Delta)^{\frac{3}{2}} b_{h} - \Delta b_{v}$ in  \eqref{est 39} but only  
\begin{equation}\label{est 129} 
(-\Delta)^{\frac{3}{2}} b_{h}
\end{equation} 
to close its $H^{1}(\mathbb{R}^{2})$-estimate; i.e., vertical diffusion is not necessary at all because we can estimate from \eqref{est 38} via H$\ddot{o}$lder's inequality, the Sobolev embedding of $\dot{H}^{1}(\mathbb{R}^{2}) \hookrightarrow L^{4}(\mathbb{R}^{2})$ and Young's inequality, 
\begin{align}
\int_{\mathbb{R}^{2}} \nabla\times (j \times b) \cdot \Delta b dx & \overset{\eqref{est 38}}{\lesssim} \int_{\mathbb{R}^{2}} \lvert \nabla b_{h} \rvert \lvert \nabla^{2} b_{h} \rvert \lvert \nabla b \rvert dx  \lesssim \lVert \nabla b_{h} \rVert_{L^{4}} \lVert \nabla^{2} b_{h} \rVert_{L^{4}} \lVert \nabla b \rVert_{L^{2}}\nonumber\\
\lesssim& \lVert \Lambda^{\frac{3}{2}} b_{h} \rVert_{L^{2}} \lVert \Lambda^{\frac{5}{2}} b_{h} \rVert_{L^{2}} \lVert \nabla b\rVert_{L^{2}} \leq \frac{\eta}{2} \lVert \Lambda^{\frac{5}{2}} b_{h} \rVert_{L^{2}}^{2} + C \lVert \Lambda^{\frac{3}{2}} b_{h} \rVert_{L^{2}}^{2} \lVert \nabla b \rVert_{L^{2}}^{2} \label{est 139}
\end{align} 
and close this estimate using $\int_{0}^{T} \lVert \Lambda^{\frac{3}{2}} b_{h} \rVert_{L^{2}}^{2} dt \lesssim 1$ from energy inequality. Nevertheless, in \cite{RY22} we did not pursue the global well-posedness of the $2\frac{1}{2}$-D Hall-MHD system with zero diffusion in the vertical component of the magnetic field in \eqref{est 129} due to the following two reasons. 
\begin{enumerate}
\item First, to the best of our knowledge, even local well-posedness of the Hall-MHD system in any dimension  requires magnetic diffusion of the form 
\begin{equation}\label{est 62} 
(-\Delta)^{\alpha} b \text{ with } \alpha > \frac{1}{2} 
\end{equation}
according to \cite{CWW15a}.
\item Second, as we described in \eqref{est 130}, $H^{1}(\mathbb{R}^{2})$-bound of the solution to the Hall-MHD system bootstraps to higher regularity in the $2\frac{1}{2}$-D case because an $H^{1}(\mathbb{R}^{2})$-bound implies from its diffusive terms
\begin{align}
\int_{0}^{T} \lVert j \rVert_{BMO}^{2} dt \lesssim \int_{0}^{T} \lVert \Delta b \rVert_{L^{2}}^{2} dt \lesssim 1;  \label{est 132}
\end{align} 
however, an $H^{1}(\mathbb{R}^{2})$-bound for $b$ with magnetic diffusion of the form \eqref{est 129} only gives $\int_{0}^{T} \lVert \Lambda^{\frac{5}{2}} b_{h} \rVert_{L^{2}}^{2} dt \lesssim 1$ and that does not bound $\int_{0}^{T} \lVert j \rVert_{BMO}^{2} dt $ in general. 
\end{enumerate} 
\end{remark}  
 
Considering the restriction \eqref{est 62}, as a second result in this manuscript, we aim to prove the global well-posedness of the following $2\frac{1}{2}$-D generalized electron MHD system: 
\begin{equation}\label{est 63} 
\partial_{t} b + \epsilon \nabla \times (j\times b) + (-\Delta)^{\frac{3}{2}}  b_{h} + (-\Delta)^{\alpha} b_{v} = 0
\end{equation} 
starting from divergence-free initial data $b^{\text{in}}$ where we note that the divergence-free property is propagated. Local well-posedness of \eqref{est 63} in $H^{m}(\mathbb{R}^{n})$ for $m \in\mathbb{N}$ such that $m > 2$ can be shown following \cite{CWW15a}; for completeness, we leave a sketch in the Appendix. Given $(-\Delta)^{\frac{3}{2}} b_{h}$ in \eqref{est 63}, we know from \eqref{est 139} that we can obtain an $H^{1}(\mathbb{R}^{2})$-bound for the solution to \eqref{est 63}. However, as we discussed in \eqref{est 132}, $H^{1}(\mathbb{R}^{2})$-bound suffices to bootstrap to higher regularity only if we have the  diffusion of $-\Delta b$, but not with $(-\Delta)^{\alpha} b_{v}$ for $\alpha > \frac{1}{2}$ in \eqref{est 63}. In Proposition \ref{Proposition 3.1} we present new cancellations in the $H^{2}(\mathbb{R}^{2})$-estimate and overcome these difficulties. Finally, we will elaborate on our last result Theorem \ref{Theorem 2.3} concerning the $2\frac{1}{2}$-D Hall-MHD system after the statement of Theorem \ref{Theorem 2.2}. 

\section{Statement of main results} 
In this section we present our main results, all of which rely crucially on Proposition \ref{Proposition 3.1}. For simplicity, hereafter we assume that $\nu = \eta = \epsilon = 1$. 

\begin{theorem}\label{Theorem 2.1} 
Suppose that $(u^{\text{in}},b^{\text{in}}) \in H^{m} (\mathbb{R}^{3}) \times H^{m} (\mathbb{R}^{3})$ where $m > \frac{5}{2}$ is an integer and $\nabla\cdot u^{\text{in}} = \nabla\cdot b^{\text{in}} = 0$. If $(u,b)$ is a corresponding local smooth solution to the 3-D Hall-MHD system \eqref{est 1} over $[0,T)$   emanating from $(u^{\text{in}}, b^{\text{in}})$ and 
\begin{subequations}\label{est 49} 
\begin{align}
& u_{h} \in L_{T}^{r_{1}} L_{x}^{p_{1}} \hspace{4mm} \text{ where } \frac{3}{p_{1}} + \frac{2}{r_{1}} \leq 1, 3 < p_{1} \leq \infty, \label{est 49a} \\
& \nabla^{2} b_{h} \in L_{T}^{r_{2}} L_{x}^{p_{2}} \text{ where } \frac{3}{p_{2}} + \frac{2}{r_{2}} \leq 2, 2 \leq p_{2} \leq 3, \label{est 49b} 
\end{align}
\end{subequations} 
then for all $t \in [0,T]$, 
\begin{equation*}
\lVert u(t) \rVert_{H^{m}} + \lVert b(t) \rVert_{H^{m}} < \infty. 
\end{equation*} 
\end{theorem}
An immediate corollary of Theorem \ref{Theorem 2.1} is a regularity criterion in terms of only \eqref{est 49b} for the 3-D electron MHD system \eqref{est 48}.  

\begin{remark}
Ji and Lee in \cite[Theorem 2]{JL10} obtained a regularity criteria for the 3-D MHD system of the form 
\begin{align*} 
& u_{h} \in L_{T}^{r_{1}} L_{x}^{p_{1}} \text{ where } \frac{3}{p_{1}} + \frac{2}{r_{1}} \leq 1, 3 < p_{1} \leq \infty,  \\
& b_{h} \in L_{T}^{r_{2}} L_{x}^{p_{2}} \text{ where } \frac{3}{p_{2}} + \frac{2}{r_{2}} \leq 1, 3 < p_{2} \leq \infty,
\end{align*}  
and hence Theorem \ref{Theorem 2.1} can be seen as a successful extension of \cite[Theorem 2]{JL10} on the 3-D MHD system to that of the 3-D Hall-MHD system (see Remark \ref{Remark 2.3} (1)). Although the Hall-MHD system does not have a scaling-invariance property (recall Remark \ref{Remark 1.1}), \eqref{est 49} is considered to be the Hall-MHD system analogue of the scaling-invariant level because the Hall term is informally one derivative more singular than the non-linear terms of the NS equations and the MHD system (recall \eqref{est 150} and Remark \ref{Remark 1.1}). 
\end{remark} 

\begin{theorem}\label{Theorem 2.2} 
Let $\alpha \in (\frac{1}{2}, 1)$. Suppose that $b^{\text{in}} \in H^{3}(\mathbb{R}^{2})$ and $\nabla\cdot b^{\text{in}} = 0$. Then, there exists a unique solution $b$ such that 
\begin{align*}
b \in L^{\infty} ((0,\infty); H^{3}(\mathbb{R}^{2})), \hspace{2mm} b_{h} \in L^{2} ((0,\infty); H^{\frac{9}{2}}(\mathbb{R}^{2})), \hspace{2mm} b_{v} \in L^{2} ((0,\infty); H^{3+ \alpha} (\mathbb{R}^{2}))
\end{align*}
to \eqref{est 63} and $b\rvert_{t=0} = b^{\text{in}}$. 
\end{theorem}
\noindent The upper bound of $\alpha$ in the hypothesis of Theorem \ref{Theorem 2.2} is only for convenience in proof. Theorem \ref{Theorem 2.2} improves \cite[Theorem 2.4]{RY22} which required $(-\Delta)^{\frac{3}{2}} b_{h} - \Delta b_{v}$ (recall \eqref{est 39b}), which in turn improved \cite[Theorem 2.3]{Y19a}.  Considering Remark \ref{Remark 1.1}, Theorem \ref{Theorem 2.2} allows us to give horizontal components the strength of critical diffusion while give the remaining component significantly weaker diffusion, by almost as much as a full Laplacian, and still obtain global regularity results (cf. \cite{YJW19}). 

We can extend Theorem \ref{Theorem 2.2} to the $2\frac{1}{2}$-D Hall-MHD system in the following manner: 
\begin{theorem}\label{Theorem 2.3}
Let $\alpha \in (\frac{1}{2}, 1)$. Suppose that $u^{\text{in}}, b^{\text{in}} \in H^{3}(\mathbb{R}^{2})$ and $\nabla\cdot u^{\text{in}} = \nabla\cdot b^{\text{in}} = 0$. Then, there exists a unique solution $(u,b)$ such that 
\begin{align*}
&u \in L^{\infty} ((0,\infty); H^{3}(\mathbb{R}^{2})), \hspace{3mm} u_{h} \in L^{2} ((0,\infty); H^{3+ \alpha}(\mathbb{R}^{2})), \hspace{3mm} u_{v} \in L^{2} ((0,\infty); H^{4} (\mathbb{R}^{2})), \\
&b \in L^{\infty} ((0,\infty); H^{3}(\mathbb{R}^{2})), \hspace{3mm} b_{h} \in L^{2} ((0,\infty); H^{\frac{9}{2}}(\mathbb{R}^{2})), \hspace{5mm} b_{v} \in L^{2} ((0,\infty); H^{3+ \alpha} (\mathbb{R}^{2})) 
\end{align*}
to 
\begin{subequations}\label{est 151}
\begin{align}
& \partial_{t} u + (u\cdot\nabla) u + \nabla \pi + (-\Delta)^{\alpha} u_{h} -\Delta u_{v}=  (b\cdot\nabla) b  \hspace{24mm} \text{ for } t > 0, \label{est 151a}\\
& \partial_{t} b + (u\cdot\nabla) b + \nabla \times ( j \times b) +  (-\Delta)^{\frac{3}{2}} b_{h} + (-\Delta)^{\alpha} b_{v} =  (b\cdot\nabla) u  \hspace{6mm} \text{ for } t > 0,  \label{est 151b}\\
& \nabla\cdot u = 0  \hspace{77mm} \text{ for } t > 0, 
\end{align}
\end{subequations} 
and $(u, b)\rvert_{t=0} = (u^{\text{in}}, b^{\text{in}})$. 
\end{theorem}  
\noindent Again, the upper bound of $\alpha$ in the hypothesis of Theorem \ref{Theorem 2.3} is only for convenience in proof. Theorem \ref{Theorem 2.3} improves \cite[Theorem 2.4]{RY22} not only in the magnetic but also viscous diffusion as the viscous diffusion in \eqref{est 39a} was $-\Delta u$. Let us make multiple comments. 

\begin{remark}\label{Remark 2.2} 
\indent
\begin{enumerate}
\item In the proof of Theorem \ref{Theorem 2.3}, our strategy is to apply a curl operator on \eqref{est 151a} and study the equation of $\omega_{3}$, the third component of vorticity $\omega$. The convenience of this equation is that the difficult non-linear terms vanish, namely
\begin{subequations}\label{est 170} 
\begin{align}
&(\omega\cdot \nabla) u_{3} = \omega_{1} \partial_{1} u_{3} + \omega_{2} \partial_{2} u_{3} = \partial_{2} u_{3} \partial_{1} u_{3} + (-\partial_{1} u_{3}) \partial_{2} u_{3} = 0, \label{est 170a}\\
&(j\cdot\nabla) b_{3} = j_{1}\partial_{1} b_{3} + j_{2} \partial_{2} b_{3} = \partial_{2} b_{3} \partial_{1} b_{3} + (-\partial_{1} b_{3}) \partial_{2} b_{3} = 0.\label{est 170b} 
\end{align}
\end{subequations}
Such cancellations are automatic in the 2-D case in which $u_{3}$ and $b_{3}$ vanish; however, the third component is non-trivial in the $2\frac{1}{2}$-D case, making \eqref{est 170} less obvious. Indeed, in sharp contrast, analogous difficult terms in the equation of the third component of the current density $j_{3}$ do not vanish because
\begin{align*}
& (j\cdot \nabla) u_{3} = j_{1} \partial_{1} u_{3} + j_{2} \partial_{2} u_{3} = \partial_{2} b_{3} \partial_{1} u_{3} + (-\partial_{1} b_{3}) \partial_{2} u_{3}, \\
& (\omega \cdot\nabla) b_{3} = \omega_{1} \partial_{1} b_{3} + \omega_{2} \partial_{2} b_{3} = \partial_{2} u_{3} \partial_{1} b_{3} + (-\partial_{1} u_{3}) \partial_{2} b_{3}, 
\end{align*}
and they appear with opposite signs and thus they do not cancel out even in sum (e.g., \cite[Equations (18)-(19)]{Y16c}). We came to the realization of \eqref{est 170} upon considering the proof of the global regularity of the $2\frac{1}{2}$-D Euler equations. Upon considering the equation of $\omega_{3}$, we see that the viscous diffusion $(-\Delta)^{\alpha} u_{h}$ in \eqref{est 151a} becomes $(-\Delta)^{\alpha} \omega_{v}$ that matches the structure of the magnetic diffusion $(-\Delta)^{\alpha} b_{v}$ in \eqref{est 151b} so that the sum $\omega + b$ will have a favorable structure for us (see \eqref{est 156}-\eqref{est 154}).  
\item We recall that Cao and Wu \cite{CW11} proved the global regularity of the following 2-D MHD system with partial dissipation and magnetic diffusion:
\begin{subequations}\label{est 152} 
\begin{align}
& \partial_{t} u + (u\cdot\nabla) u + \nabla \pi + \partial_{1}^{2} u=  (b\cdot\nabla) b  \hspace{16mm} \text{ for } t > 0, \label{est 152a}\\
& \partial_{t} b + (u\cdot\nabla) b + \partial_{2}^{2} b =  (b\cdot\nabla) u  \hspace{24mm} \text{ for } t > 0,  \label{est 152b}\\
& \nabla\cdot u = 0  \hspace{54mm} \text{ for } t > 0, 
\end{align}
\end{subequations}
and also in the case of viscous diffusion $\partial_{2}^{2} u$ and magnetic diffusion $\partial_{1}^{2} b$. On one hand, \eqref{est 152} is the 2-D MHD system while \eqref{est 151} is the $2\frac{1}{2}$-D Hall-MHD system. On the other hand, the diffusion strength somehow complement one another; i.e., $\partial_{1}^{2} u$ and $\partial_{2}^{2} b$ in \eqref{est 152} while the need for strong magnetic diffusion $(-\Delta)^{\frac{3}{2}} b_{h} + (-\Delta)^{\alpha} b_{v}$ for $\alpha > \frac{1}{2}$ in Theorem \ref{Theorem 2.3} can be offset by relatively weak viscous diffusion $(-\Delta)^{\alpha} u_{h} - \Delta u_{v}$. 
\item The $2\frac{1}{2}$-D Hall-MHD system, for which the global regularity issue remains open, can be written as follows:
\begin{subequations}\label{est 153}
\begin{align}
& \partial_{t} u + (u\cdot\nabla) u + \nabla \pi -\Delta u_{h} -\Delta u_{v}=  (b\cdot\nabla) b  \hspace{17mm} \text{ for } t > 0, \\
& \partial_{t} b + (u\cdot\nabla) b +  \nabla \times ( j \times b) -\Delta b_{h} -\Delta b_{v} =  (b\cdot\nabla) u  \hspace{6mm} \text{ for } t > 0,  \\
& \nabla\cdot u = 0  \hspace{65mm} \text{ for } t > 0. 
\end{align}
\end{subequations} 
We observe that \eqref{est 153} has a total of six derivatives in viscous diffusion and six derivatives in magnetic diffusion
\begin{align*}
-\Delta u_{1}, -\Delta u_{2}, -\Delta u_{3}, \hspace{3mm} -\Delta b_{1}, -\Delta b_{2}, -\Delta b_{3},
\end{align*}
summing to 12. Analogous sums for \eqref{est 151} in Theorem \ref{Theorem 2.3} is $4\alpha + 2$ in viscous diffusion and $6 + 2 \alpha$ in magnetic diffusion
\begin{align*}
(-\Delta)^{\alpha} u_{1}, (-\Delta)^{\alpha} u_{2}, -\Delta u_{3}, \hspace{3mm} (-\Delta)^{\frac{3}{2}} b_{1}, (-\Delta)^{\frac{3}{2}} b_{2}, (-\Delta)^{\alpha} b_{3}, 
\end{align*}
summing to $ 8 + 6 \alpha$ for $\alpha > \frac{1}{2}$ and thus larger than, but arbitrarily close to, 11. The fact that we need $(6 + 2 \alpha)$-many derivatives in the magnetic diffusion rather than 6 can be understood as the effect from the Hall term (recall our discussion at \eqref{est 150} and Remark \ref{Remark 1.1}).
\end{enumerate} 
\end{remark} 
 
\begin{remark}\label{Remark 2.3}
We end with a few open questions for future work. 
\begin{enumerate}
\item It will be of great interest if we can improve \eqref{est 49b} of Theorem \ref{Theorem 2.1} from $\nabla^{2} b_{h}$ to $\nabla b_{h}$, of course with different conditions on $p_{2}$ and $r_{2}$.
\item It is of great interest if we can improve Theorems \ref{Theorem 2.2}-\ref{Theorem 2.3} by reducing the required strength of diffusion. In relevance, we recall how the global regularity issue of the 2-D MHD system with zero viscous diffusion has caught much attention and made remarkable progress in the past decade (e.g., \cite{CWY14, FMMNZ14, JZ14a, Y18} and references therein).
\end{enumerate} 
\end{remark} 

In Section \ref{Section 3} we prove Proposition \ref{Proposition 3.1} about the new cancellations within the Hall term upon $H^{2}(\mathbb{R}^{n})$-estimates, $n \in \{2,3\}$. In Section \ref{Section 4}, we prove Theorem \ref{Theorem 2.1}. Taking $u \equiv 0$ in \eqref{est 151} does not deduce exactly the electron MHD system \eqref{est 63} as it additionally requires $\nabla \pi = (b\cdot\nabla)b$; moreover, the extension of Theorem \ref{Theorem 2.2} on the electron MHD system \eqref{est 63} to Theorem \ref{Theorem 2.3} on the Hall-MHD system \eqref{est 151} does not seem trivial; thus, we will prove Theorem \ref{Theorem 2.2} in Section \ref{Section 5} and then Theorem \ref{Theorem 2.3} in Section \ref{Section 6}.  

\section{The cancellation within the Hall term}\label{Section 3} 
The following is the crux of the proofs of all of Theorems \ref{Theorem 2.1}-\ref{Theorem 2.3}. 

\begin{proposition}\label{Proposition 3.1} 
\indent
\begin{enumerate}
\item Suppose that $b (x) = (b_{1}, b_{2}, b_{3})(x_{1}, x_{2},x_{3})$ is smooth. Then it satisfies 
\begin{equation}\label{est 58} 
\int_{\mathbb{R}^{3}}  \Delta \nabla \times (j\times b) \cdot \Delta b dx  \lesssim \int_{\mathbb{R}^{3}}  \lvert \nabla^{2} b_{h} \rvert ( \lvert \nabla b \rvert \lvert \nabla^{3} b \rvert + \lvert \nabla^{2} b_{v} \rvert \lvert \nabla^{2} b \rvert)dx. 
\end{equation} 
\item Suppose that $b (x) = (b_{1}, b_{2}, b_{3})(x_{1}, x_{2})$ is smooth and $\nabla\cdot b = 0$. Then it satisfies 
\begin{equation}\label{est 86}
\int_{\mathbb{R}^{2}} \Delta \nabla \times (j\times b) \cdot \Delta b dx \lesssim \int_{\mathbb{R}^{2}} ( \lvert \nabla b_{h} \rvert \lvert \nabla^{3} b_{h} \rvert + \lvert \nabla^{2} b_{h} \rvert^{2}) \lvert \nabla^{2} b_{v} \rvert dx. 
\end{equation} 
\end{enumerate}
\end{proposition}

The inequality \eqref{est 58} will be used to prove Theorem \ref{Theorem 2.1} concerning a regularity criterion in terms of $\nabla^{2} b_{h}$. To prove Theorems \ref{Theorem 2.2}-\ref{Theorem 2.3}, we need to separate $b_{h}$ twice in the upper bound; moreover, it is well-known that the derivatives must be relatively balanced to be able to close the necessary estimates. This is the context of the inequality \eqref{est 86}; indeed, there is no ``$\lvert b_{h} \rvert$'' in the upper bound of \eqref{est 86}.  

\begin{remark}\label{Remark 3.1}  
We mention that we can extend \eqref{est 86} to the 3-D case as follows. If $b (x) = (b_{1}, b_{2}, b_{3})(x_{1}, x_{2},x_{3})$ is smooth and $\nabla\cdot b= 0$, then 
\begin{equation}\label{est 100} 
\int_{\mathbb{R}^{3}}  \Delta \nabla \times (j\times b) \cdot \Delta b dx  \lesssim \int_{\mathbb{R}^{3}} \lvert \nabla b \rvert \lvert \nabla^{2} b_{h} \rvert \lvert \nabla^{3} b_{h} \rvert + \lvert \nabla^{2} b_{v} \rvert (\lvert \nabla^{2} b_{h} \rvert^{2} + \lvert \nabla b_{h} \rvert \lvert \nabla^{3} b_{h} \rvert) dx; 
\end{equation} 
we emphasize that it separates $b_{h}$ twice and derivatives are relatively balanced. For the purpose of proving Theorem \ref{Theorem 2.1}, \eqref{est 58} suffices, and the proof of \eqref{est 100} is more difficult than those of \eqref{est 58} and \eqref{est 86}. Nonetheless, because it may be useful for future works, we leave its proof in the Appendix. 
\end{remark} 

\begin{proof}[Proof of Proposition \ref{Proposition 3.1}]
As we will see, all these inequalities \eqref{est 58}-\eqref{est 100} will be deduced from the identity \eqref{est 61}. Up to \eqref{est 61} we will write $\int_{\mathbb{R}^{n}} \hdots dx$ and $\sum_{k,l=1}^{n}$ and we will specify $n \in \{2,3\}$ after \eqref{est 61} is derived; of course, all the terms involving $\partial_{3}$ are considered to be zero in the $2\frac{1}{2}$-D case. We compute  the Hall term as follows:
\begin{equation}\label{est 28}
\int_{\mathbb{R}^{n}}  \Delta \nabla \times (j\times b) \cdot \Delta b dx = \int_{\mathbb{R}^{n}}  \sum_{k=1}^{n} \partial_{k}^{2} (j\times b) \cdot \sum_{l=1}^{n} \partial_{l}^{2} j dx = \RomanI + \RomanII 
\end{equation} 
where due to \eqref{est 3}, 
\begin{equation}\label{est 11} 
\RomanI \triangleq 2 \sum_{k,l=1}^{n} \int_{\mathbb{R}^{n}}  (\partial_{k} j \times \partial_{k} b) \cdot \partial_{l}^{2} j dx\hspace{2mm} \text{ and } \hspace{2mm} \RomanII \triangleq \sum_{k,l=1}^{n} \int_{\mathbb{R}^{n}}  (j \times \partial_{k}^{2} b) \cdot \partial_{l}^{2} j dx.
\end{equation} 
We first decompose $\RomanI$ from \eqref{est 11} as 
\begin{equation}\label{est 29}
\RomanI = \sum_{i=1}^{6} \RomanI_{i}
\end{equation} 
where 
\begin{subequations}\label{est 5a} 
\begin{align}
\RomanI_{1} \triangleq & 2 \sum_{k,l=1}^{n} \int_{\mathbb{R}^{n}}  \partial_{k} j_{2} \partial_{k} b_{3} \partial_{l}^{2} j_{1}dx,  \hspace{14mm} \RomanI_{2} \triangleq  -2 \sum_{k,l=1}^{n} \int_{\mathbb{R}^{n}} \partial_{k} j_{3}\partial_{k} b_{2}\partial_{l}^{2} j_{1}dx, \\
\RomanI_{3} \triangleq& -2 \sum_{k,l=1}^{n} \int_{\mathbb{R}^{n}}  \partial_{k} j_{1}\partial_{k} b_{3}\partial_{l}^{2} j_{2}dx, \hspace{10mm}  \RomanI_{4} \triangleq 2 \sum_{k,l=1}^{n} \int_{\mathbb{R}^{n}}  \partial_{k} j_{3} \partial_{k} b_{1}\partial_{l}^{2} j_{2}dx, \\
\RomanI_{5} \triangleq& 2 \sum_{k,l=1}^{n} \int_{\mathbb{R}^{n}}  \partial_{k}j_{1}\partial_{k} b_{2} \partial_{l}^{2} j_{3}dx, \hspace{14mm}  \RomanI_{6} \triangleq -2 \sum_{k,l=1}^{n}\int_{\mathbb{R}^{n}}  \partial_{k} j_{2}\partial_{k}b_{1}\partial_{l}^{2} j_{3}dx.
\end{align}
\end{subequations}
We strategically pair up from \eqref{est 5a} 
\begin{equation}\label{est 5} 
\RomanI_{1} + \RomanI_{3} \overset{\eqref{est 5a}}{=} 2 \sum_{k,l=1}^{n} \int_{\mathbb{R}^{n}}  \partial_{k} b_{3} (\partial_{k} j_{2} \partial_{l}^{2} j_{1} - \partial_{k} j_{1} \partial_{l}^{2} j_{2})dx 
\overset{\eqref{est 4}}{=} \sum_{l=1}^{8} \RomanI_{1,3,l}
\end{equation}
where 
\begin{subequations}\label{est 6} 
\begin{align}
\RomanI_{1,3,1} \triangleq& -2\sum_{k,l=1}^{n}\int_{\mathbb{R}^{n}}  \partial_{k} b_{3}\partial_{k}\partial_{1}b_{3}\partial_{l}^{2}\partial_{2}b_{3}dx, \hspace{1mm}  \RomanI_{1,3,2} \triangleq 2\sum_{k,l=1}^{n}\int_{\mathbb{R}^{n}}  \partial_{k} b_{3}\partial_{k}\partial_{1}b_{3}\partial_{l}^{2}\partial_{3}b_{2}dx, \\
\RomanI_{1,3,3} \triangleq& 2\sum_{k,l=1}^{n}\int_{\mathbb{R}^{n}}  \partial_{k} b_{3}\partial_{k}\partial_{3}b_{1}\partial_{l}^{2}\partial_{2}b_{3}dx, \hspace{5mm}  \RomanI_{1,3,4} \triangleq -2\sum_{k,l=1}^{n}\int_{\mathbb{R}^{n}}  \partial_{k} b_{3}\partial_{k}\partial_{3}b_{1}\partial_{l}^{2}\partial_{3}b_{2}dx, \\
\RomanI_{1,3,5} \triangleq& 2\sum_{k,l=1}^{n}\int_{\mathbb{R}^{n}}  \partial_{k} b_{3}\partial_{k}\partial_{2}b_{3}\partial_{l}^{2}\partial_{1}b_{3}dx, \hspace{5mm}  \RomanI_{1,3,6} \triangleq -2\sum_{k,l=1}^{n}\int_{\mathbb{R}^{n}}  \partial_{k} b_{3}\partial_{k}\partial_{2}b_{3}\partial_{l}^{2}\partial_{3}b_{1}dx, \\
\RomanI_{1,3,7} \triangleq& -2\sum_{k,l=1}^{n}\int_{\mathbb{R}^{n}}  \partial_{k} b_{3}\partial_{k}\partial_{3}b_{2}\partial_{l}^{2}\partial_{1}b_{3}dx, \hspace{1mm} \RomanI_{1,3,8} \triangleq 2\sum_{k,l=1}^{n}\int_{\mathbb{R}^{n}}  \partial_{k} b_{3}\partial_{k}\partial_{3}b_{2}\partial_{l}^{2}\partial_{3}b_{1}dx.
\end{align}
\end{subequations} 
Within \eqref{est 6}, we see a cancellation from $\RomanI_{1,3,1} + \RomanI_{1,3,5}$: 
\begin{align}
\RomanI_{1,3,1} + \RomanI_{1,3,5} \overset{\eqref{est 6}}{=}& 2 \sum_{k,l=1}^{n} \int_{\mathbb{R}^{n}}  \partial_{k} b_{3} (-\partial_{k}\partial_{1} b_{3}\partial_{l}^{2} \partial_{2} b_{3} + \partial_{k}\partial_{2}b_{3}\partial_{l}^{2}\partial_{1}b_{3})dx \nonumber \\
=& -2 \sum_{k,l=1}^{n}\int_{\mathbb{R}^{n}}   \frac{1}{2}\partial_{1} (\partial_{k}b_{3})^{2}\partial_{l}^{2}\partial_{2} b_{3} - \frac{1}{2}\partial_{2} (\partial_{k} b_{3})^{2} \partial_{l}^{2}\partial_{1}b_{3}dx \nonumber \\
=& \sum_{k,l=1}^{n}\int_{\mathbb{R}^{n}}   (\partial_{k} b_{3})^{2} \partial_{l}^{2}\partial_{1}\partial_{2} b_{3} - (\partial_{k}b_{3})^{2} \partial_{l}^{2} \partial_{2}\partial_{1} b_{3}dx = 0. \label{est 34}
\end{align}
Similarly, we pair up from \eqref{est 5a}
\begin{equation}\label{est 7} 
\RomanI_{2} + \RomanI_{5} \overset{\eqref{est 5a}}{=} 2 \sum_{k,l=1}^{n} \int_{\mathbb{R}^{n}}  \partial_{k} b_{2} (-\partial_{k} j_{3} \partial_{l}^{2} j_{1} + \partial_{k} j_{1} \partial_{l}^{2} j_{3}) dx
\overset{\eqref{est 4}}{=} \sum_{l=1}^{8} I_{2,5,l}
\end{equation}
where 
\begin{subequations}\label{est 8} 
\begin{align}
\RomanI_{2,5,1} \triangleq& -2\sum_{k,l=1}^{n}\int_{\mathbb{R}^{n}}  \partial_{k} b_{2}\partial_{k}\partial_{1}b_{2}\partial_{l}^{2}\partial_{2}b_{3}dx , \hspace{1mm}  \RomanI_{2,5,2} \triangleq 2\sum_{k,l=1}^{n}\int_{\mathbb{R}^{n}}  \partial_{k} b_{2}\partial_{k}\partial_{1}b_{2}\partial_{l}^{2}\partial_{3}b_{2}dx, \\
\RomanI_{2,5,3} \triangleq& 2\sum_{k,l=1}^{n}\int_{\mathbb{R}^{n}}  \partial_{k} b_{2}\partial_{k}\partial_{2}b_{1}\partial_{l}^{2}\partial_{2}b_{3}dx, \hspace{5mm}  \RomanI_{2,5,4} \triangleq -2\sum_{k,l=1}^{n}\int_{\mathbb{R}^{n}}  \partial_{k} b_{2}\partial_{k}\partial_{2}b_{1}\partial_{l}^{2}\partial_{3}b_{2}dx, \\
\RomanI_{2,5,5} \triangleq& 2\sum_{k,l=1}^{n}\int_{\mathbb{R}^{n}}  \partial_{k} b_{2}\partial_{k}\partial_{2}b_{3}\partial_{l}^{2}\partial_{1}b_{2}dx, \hspace{5mm}  \RomanI_{2,5,6} \triangleq -2\sum_{k,l=1}^{n}\int_{\mathbb{R}^{n}}  \partial_{k} b_{2}\partial_{k}\partial_{2}b_{3}\partial_{l}^{2}\partial_{2}b_{1}dx, \\
\RomanI_{2,5,7} \triangleq& -2\sum_{k,l=1}^{n}\int_{\mathbb{R}^{n}}  \partial_{k} b_{2}\partial_{k}\partial_{3}b_{2}\partial_{l}^{2}\partial_{1}b_{2}dx, \hspace{1mm} \RomanI_{2,5,8} \triangleq 2\sum_{k,l=1}^{n}\int_{\mathbb{R}^{n}}  \partial_{k} b_{2}\partial_{k}\partial_{3}b_{2}\partial_{l}^{2}\partial_{2}b_{1}dx.
\end{align}
\end{subequations} 
The cancellation within \eqref{est 8} is $\RomanI_{2,5,2} + \RomanI_{2,5,7}$: 
\begin{align}
\RomanI_{2,5,2} + \RomanI_{2,5,7} \overset{\eqref{est 8}}{=}& 2 \sum_{k,l=1}^{n} \int_{\mathbb{R}^{n}}  \partial_{k} b_{2} (\partial_{k}\partial_{1} b_{2}\partial_{l}^{2} \partial_{3} b_{2} - \partial_{k}\partial_{3}b_{2}\partial_{l}^{2}\partial_{1}b_{2})dx \nonumber \\
=& 2 \sum_{k,l=1}^{n}\int_{\mathbb{R}^{n}}   \frac{1}{2}\partial_{1} (\partial_{k}b_{2})^{2}\partial_{l}^{2}\partial_{3} b_{2} - \frac{1}{2}\partial_{3} (\partial_{k} b_{2})^{2} \partial_{l}^{2}\partial_{1}b_{2} dx \nonumber \\
=&- \sum_{k,l=1}^{n}\int_{\mathbb{R}^{n}}   (\partial_{k} b_{2})^{2} \partial_{l}^{2}\partial_{1}\partial_{3} b_{2} - (\partial_{k}b_{2})^{2} \partial_{l}^{2} \partial_{3}\partial_{1} b_{2} dx = 0. \label{est 35}
\end{align}
Finally, we pair up from \eqref{est 5a} 
\begin{equation}\label{est 9} 
\RomanI_{4} + \RomanI_{6} \overset{\eqref{est 5a}}{=} 2 \sum_{k,l=1}^{n} \int_{\mathbb{R}^{n}}  \partial_{k} b_{1} (\partial_{k} j_{3} \partial_{l}^{2} j_{2} - \partial_{k} j_{2} \partial_{l}^{2} j_{3}) dx 
\overset{\eqref{est 4}}{=} \sum_{l=1}^{8} I_{4,6,l}
\end{equation}
where 
\begin{subequations}\label{est 10} 
\begin{align}
\RomanI_{4,6,1} \triangleq& -2\sum_{k,l=1}^{n}\int_{\mathbb{R}^{n}}  \partial_{k} b_{1}\partial_{k}\partial_{1}b_{2}\partial_{l}^{2}\partial_{1}b_{3} dx, \hspace{1mm} \RomanI_{4,6,2} \triangleq 2\sum_{k,l=1}^{n}\int_{\mathbb{R}^{n}}  \partial_{k} b_{1}\partial_{k}\partial_{1}b_{2}\partial_{l}^{2}\partial_{3}b_{1} dx, \\
\RomanI_{4,6,3} \triangleq& 2\sum_{k,l=1}^{n}\int_{\mathbb{R}^{n}}  \partial_{k} b_{1}\partial_{k}\partial_{2}b_{1}\partial_{l}^{2}\partial_{1}b_{3} dx, \hspace{5mm} \RomanI_{4,6,4} \triangleq -2\sum_{k,l=1}^{n}\int_{\mathbb{R}^{n}}  \partial_{k} b_{1}\partial_{k}\partial_{2}b_{1}\partial_{l}^{2}\partial_{3}b_{1}dx, \\
\RomanI_{4,6,5} \triangleq& 2\sum_{k,l=1}^{n}\int_{\mathbb{R}^{n}}  \partial_{k} b_{1}\partial_{k}\partial_{1}b_{3}\partial_{l}^{2}\partial_{1}b_{2}dx, \hspace{5mm}  \RomanI_{4,6,6} \triangleq -2\sum_{k,l=1}^{n}\int_{\mathbb{R}^{n}}  \partial_{k} b_{1}\partial_{k}\partial_{1}b_{3}\partial_{l}^{2}\partial_{2}b_{1}dx, \\
\RomanI_{4,6,7} \triangleq& -2\sum_{k,l=1}^{n}\int_{\mathbb{R}^{n}}  \partial_{k} b_{1}\partial_{k}\partial_{3}b_{1}\partial_{l}^{2}\partial_{1}b_{2}dx, \hspace{1mm} \RomanI_{4,6,8} \triangleq 2\sum_{k,l=1}^{n}\int_{\mathbb{R}^{n}}  \partial_{k} b_{1}\partial_{k}\partial_{3}b_{1}\partial_{l}^{2}\partial_{2}b_{1}dx.
\end{align}
\end{subequations} 
The cancellation within \eqref{est 10} is $\RomanI_{4,6,4} + \RomanI_{4,6,8}$: 
\begin{align}
\RomanI_{4,6,4} + \RomanI_{4,6,8} \overset{\eqref{est 10}}{=}& 2 \sum_{k,l=1}^{n} \int_{\mathbb{R}^{n}}  \partial_{k} b_{1} (-\partial_{k}\partial_{2} b_{1}\partial_{l}^{2} \partial_{3} b_{1} + \partial_{k}\partial_{3}b_{1}\partial_{l}^{2}\partial_{2}b_{1}) dx \nonumber \\
=& -2 \sum_{k,l=1}^{n}\int_{\mathbb{R}^{n}}   \frac{1}{2}\partial_{2} (\partial_{k}b_{1})^{2}\partial_{l}^{2}\partial_{3} b_{1} - \frac{1}{2}\partial_{3} (\partial_{k} b_{1})^{2} \partial_{l}^{2}\partial_{2}b_{1} dx \nonumber \\
=& \sum_{k,l=1}^{n}\int_{\mathbb{R}^{n}}   (\partial_{k} b_{1})^{2} \partial_{l}^{2}\partial_{2}\partial_{3} b_{1} - (\partial_{k}b_{1})^{2} \partial_{l}^{2} \partial_{2}\partial_{3} b_{1}dx = 0. \label{est 36}
\end{align}
Next, we work on $\RomanII$ from \eqref{est 11} in which the cancellations are less obvious. First, we compute 
\begin{equation}\label{est 30} 
\RomanII \overset{\eqref{est 11}}{=} \sum_{k,l=1}^{n} \int_{\mathbb{R}^{n}}  (j\times \partial_{k}^{2} b) \cdot \partial_{l}^{2} j  dx= \sum_{i=1}^{6} \RomanII_{i} 
\end{equation} 
where 
\begin{subequations}\label{est 12} 
\begin{align}
\RomanII_{1} \triangleq& \sum_{k,l=1}^{n} \int_{\mathbb{R}^{n}}  j_{2}\partial_{k}^{2} b_{3}\partial_{l}^{2} j_{1}dx, \hspace{12mm}\RomanII_{2} \triangleq - \sum_{k,l=1}^{n} \int_{\mathbb{R}^{n}}  j_{3}\partial_{k}^{2} b_{2}\partial_{l}^{2} j_{1}dx, \\
\RomanII_{3} \triangleq& - \sum_{k,l=1}^{n} \int_{\mathbb{R}^{n}}  j_{1} \partial_{k}^{2} b_{3}\partial_{l}^{2} j_{2}dx, \hspace{9mm} \RomanII_{4} \triangleq \sum_{k,l=1}^{n}\int_{\mathbb{R}^{n}}  j_{3} \partial_{k}^{2} b_{1} \partial_{l}^{2} j_{2}dx, \\
\RomanII_{5} \triangleq& \sum_{k,l=1}^{n} \int_{\mathbb{R}^{n}}  j_{1} \partial_{k}^{2} b_{2} \partial_{l}^{2} j_{3}dx, \hspace{12mm}\RomanII_{6} \triangleq - \sum_{k,l=1}^{n}\int_{\mathbb{R}^{n}}  j_{2} \partial_{k}^{2} b_{1}\partial_{l}^{2} j_{3}dx. 
\end{align}
\end{subequations}
Now we strategically couple from \eqref{est 12}
\begin{equation}\label{est 13} 
\RomanII_{1} + \RomanII_{3} \overset{\eqref{est 12}}{=} \sum_{k,l=1}^{n} \int_{\mathbb{R}^{n}}  j_{2}\partial_{k}^{2} b_{3} \partial_{l}^{2} j_{1} - j_{1} \partial_{k}^{2} b_{3}\partial_{l}^{2} j_{2} dx
\end{equation}
where we integrate by parts separately to obtain 
\begin{align*}
& \int_{\mathbb{R}^{n}}  j_{2}\partial_{k}^{2} b_{3}\partial_{l}^{2} j_{1} dx= - \int_{\mathbb{R}^{n}}  \partial_{l} j_{2}\partial_{k}^{2} b_{3}\partial_{l} j_{1} + j_{2}\partial_{k}^{2}\partial_{l} b_{3}\partial_{l} j_{1}dx, \\
& -\int_{\mathbb{R}^{n}}  j_{1} \partial_{k}^{2} b_{3}\partial_{l}^{2} j_{2} dx= \int_{\mathbb{R}^{n}}  \partial_{l} j_{1}\partial_{k}^{2} b_{3} \partial_{l} j_{2} + j_{1}\partial_{k}^{2}\partial_{l} b_{3}\partial_{l} j_{2}dx, 
\end{align*}
so that we see a cancellation in sum, specifically $-\int_{\mathbb{R}^{n}}  \partial_{l} j_{2} \partial_{k}^{2} b_{3}\partial_{l} j_{1}dx$ and $\int_{\mathbb{R}^{n}}  \partial_{l} j_{1}\partial_{k}^{2} b_{3}\partial_{l} j_{2}dx$, leading us to 
\begin{equation}\label{est 31}
\RomanII_{1} + \RomanII_{3} \overset{\eqref{est 13}}{=}  -\sum_{k,l=1}^{n}\int_{\mathbb{R}^{n}}  \partial_{k}^{2}\partial_{l} b_{3} (j_{2}\partial_{l} j_{1} - j_{1} \partial_{l} j_{2}) dx
\overset{\eqref{est 4}}{=}\sum_{l=1}^{8} \RomanII_{1,3,l}
\end{equation}
where 
\begin{subequations}\label{est 14} 
\begin{align}
\RomanII_{1,3,1} \triangleq& \sum_{k,l=1}^{n} \int_{\mathbb{R}^{n}}  \partial_{k}^{2}\partial_{l} b_{3}\partial_{1}b_{3}\partial_{l}\partial_{2} b_{3}dx, \hspace{7mm}\RomanII_{1,3,2} \triangleq -\sum_{k,l=1}^{n} \int_{\mathbb{R}^{n}}  \partial_{k}^{2}\partial_{l} b_{3}\partial_{1}b_{3}\partial_{l}\partial_{3} b_{2}dx, \\
\RomanII_{1,3,3} \triangleq& -\sum_{k,l=1}^{n} \int_{\mathbb{R}^{n}}  \partial_{k}^{2}\partial_{l} b_{3}\partial_{3}b_{1}\partial_{l}\partial_{2} b_{3}dx, \hspace{4mm}\RomanII_{1,3,4} \triangleq \sum_{k,l=1}^{n} \int_{\mathbb{R}^{n}}  \partial_{k}^{2}\partial_{l} b_{3}\partial_{3}b_{1}\partial_{l}\partial_{3} b_{2}dx, \\
\RomanII_{1,3,5} \triangleq& -\sum_{k,l=1}^{n} \int_{\mathbb{R}^{n}}  \partial_{k}^{2}\partial_{l} b_{3}\partial_{2}b_{3}\partial_{l}\partial_{1} b_{3}dx, \hspace{4mm}\RomanII_{1,3,6} \triangleq \sum_{k,l=1}^{n} \int_{\mathbb{R}^{n}}  \partial_{k}^{2}\partial_{l} b_{3}\partial_{2}b_{3}\partial_{l}\partial_{3} b_{1}dx, \\
\RomanII_{1,3,7} \triangleq& \sum_{k,l=1}^{n} \int_{\mathbb{R}^{n}}  \partial_{k}^{2}\partial_{l} b_{3}\partial_{3}b_{2}\partial_{l}\partial_{1} b_{3}dx, \hspace{7mm}\RomanII_{1,3,8} \triangleq -\sum_{k,l=1}^{n} \int_{\mathbb{R}^{n}}  \partial_{k}^{2}\partial_{l} b_{3}\partial_{3}b_{2}\partial_{l}\partial_{3} b_{1}dx.
\end{align}
\end{subequations}
We couple $\RomanII_{1,3,1}$ and $\RomanII_{1,3,5}$ from \eqref{est 14} to obtain 
\begin{align}
\RomanII_{1,3,1} + \RomanII_{1,3,5} \overset{\eqref{est 14}}{=}& \sum_{k,l=1}^{n} \int_{\mathbb{R}^{n}}  \partial_{k}^{2} \partial_{l} b_{3} \partial_{1} b_{3} \partial_{l} \partial_{2} b_{3} - \partial_{k}^{2}\partial_{l} b_{3} \partial_{2} b_{3}\partial_{l} \partial_{1} b_{3}dx \label{est 16} 
\end{align}
and integrate by parts separately to obtain 
\begin{subequations}\label{est 15}
\begin{align}
&\sum_{k,l=1}^{n} \int_{\mathbb{R}^{n}}  \partial_{k}^{2} \partial_{l} b_{3} \partial_{1} b_{3}\partial_{l} \partial_{2} b_{3} dx\nonumber\\
=& -\sum_{k,l=1}^{n}  \int_{\mathbb{R}^{n}}  \partial_{k} \partial_{l} b_{3}\partial_{k}\partial_{1} b_{3}\partial_{l} \partial_{2} b_{3} + \partial_{k}\partial_{l} b_{3}\partial_{1} b_{3} \partial_{k}\partial_{l}\partial_{2} b_{3}dx, \label{est 15a} \\
& -\sum_{k,l=1}^{n} \int_{\mathbb{R}^{n}}  \partial_{k}^{2}\partial_{l} b_{3}\partial_{2}b_{3}\partial_{l}\partial_{1} b_{3} dx\nonumber \\
=& \sum_{k,l=1}^{n} \int_{\mathbb{R}^{n}}  \partial_{k}\partial_{l} b_{3}\partial_{k}\partial_{2} b_{3}\partial_{l}\partial_{1} b_{3} + \partial_{k}\partial_{l} b_{3} \partial_{2} b_{3} \partial_{k} \partial_{l} \partial_{1} b_{3}dx.\label{est 15b}
\end{align}
\end{subequations} 
Then the first terms in \eqref{est 15a}-\eqref{est 15b} together cancel out as follows:
\begin{equation*}
- \sum_{k,l=1}^{n} \int_{\mathbb{R}^{n}}  \partial_{k} \partial_{l} b_{3}\partial_{k}\partial_{1}b_{3}\partial_{l}\partial_{2}b_{3} - \partial_{k}\partial_{l}b_{3}\partial_{k}\partial_{2}b_{3}\partial_{l}\partial_{1}b_{3} dx= 0
\end{equation*}
which can be seen by just swapping $k\leftrightarrow l$ in the second integrand. On the other hand, the second terms in \eqref{est 15a}-\eqref{est 15b} also cancel out as 
\begin{align*}
&- \sum_{k,l=1}^{n} \int_{\mathbb{R}^{n}}  \partial_{k} \partial_{l} b_{3}\partial_{1}b_{3}\partial_{k}\partial_{l} \partial_{2} b_{3} - \partial_{k}\partial_{l} b_{3}\partial_{2}b_{3}\partial_{k}\partial_{l}\partial_{1} b_{3} dx \nonumber \\
=&- \sum_{k,l=1}^{n}\int_{\mathbb{R}^{n}}  \frac{1}{2}\partial_{2} (\partial_{k}\partial_{l} b_{3})^{2}\partial_{1}b_{3} - \frac{1}{2} \partial_{1} (\partial_{k}\partial_{l} b_{3})^{2}\partial_{2} b_{3}dx \nonumber \\
=&  \frac{1}{2}\sum_{k,l=1}^{n}\int_{\mathbb{R}^{n}}  (\partial_{k}\partial_{l} b_{3})^{2} \partial_{1}\partial_{2} b_{3} - (\partial_{k}\partial_{l} b_{3})^{2}\partial_{1}\partial_{2} b_{3} dx= 0. 
\end{align*}
Therefore, we conclude from \eqref{est 16} that 
\begin{equation}\label{est 25}
\RomanII_{1,3,1} + \RomanII_{1,3,5} = 0. 
\end{equation}
Similarly, we couple from \eqref{est 12}  
\begin{equation}\label{est 17} 
\RomanII_{2} + \RomanII_{5} \overset{\eqref{est 12}}{=} -\sum_{k,l=1}^{n} \int_{\mathbb{R}^{n}}  j_{3}\partial_{k}^{2} b_{2} \partial_{l}^{2} j_{1} - j_{1} \partial_{k}^{2} b_{2}\partial_{l}^{2} j_{3} dx
\end{equation}
where we integrate by parts separately to obtain 
\begin{align*}
&- \int_{\mathbb{R}^{n}}  j_{3}\partial_{k}^{2} b_{2}\partial_{l}^{2} j_{1} dx=  \int_{\mathbb{R}^{n}}  \partial_{l} j_{3}\partial_{k}^{2} b_{2}\partial_{l} j_{1} + j_{3}\partial_{k}^{2}\partial_{l} b_{2}\partial_{l} j_{1}dx, \\
& \int_{\mathbb{R}^{n}}  j_{1} \partial_{k}^{2} b_{2}\partial_{l}^{2} j_{3} dx= -\int_{\mathbb{R}^{n}}  \partial_{l} j_{1}\partial_{k}^{2} b_{2} \partial_{l} j_{3} + j_{1}\partial_{k}^{2}\partial_{l} b_{2}\partial_{l} j_{3}dx, 
\end{align*}
so that we see a cancellation in sum, specifically $\int_{\mathbb{R}^{n}}  \partial_{l} j_{3}\partial_{k}^{2} b_{2}\partial_{l} j_{1}dx$ and $- \int_{\mathbb{R}^{n}}  \partial_{l} j_{1}\partial_{k}^{2} b_{2}\partial_{l} j_{3}dx$, leading us to 
\begin{equation}\label{est 32}
\RomanII_{2} + \RomanII_{5} \overset{\eqref{est 17}}{=}  \sum_{k,l=1}^{n}\int_{\mathbb{R}^{n}}  \partial_{k}^{2}\partial_{l} b_{2} (j_{3}\partial_{l} j_{1} - j_{1} \partial_{l} j_{3}) dx
\overset{\eqref{est 4}}{=}\sum_{l=1}^{8} \RomanII_{2,5,l}
\end{equation}
where 
\begin{subequations}\label{est 18} 
\begin{align}
\RomanII_{2,5,1} \triangleq& \sum_{k,l=1}^{n} \int_{\mathbb{R}^{n}}  \partial_{k}^{2}\partial_{l} b_{2}\partial_{1}b_{2}\partial_{l}\partial_{2} b_{3}dx, \hspace{8mm} \RomanII_{2,5,2} \triangleq -\sum_{k,l=1}^{n} \int_{\mathbb{R}^{n}}  \partial_{k}^{2}\partial_{l} b_{2}\partial_{1}b_{2}\partial_{l}\partial_{3} b_{2}dx, \\
\RomanII_{2,5,3} \triangleq& -\sum_{k,l=1}^{n} \int_{\mathbb{R}^{n}}  \partial_{k}^{2}\partial_{l} b_{2}\partial_{2}b_{1}\partial_{l}\partial_{2} b_{3}dx, \hspace{5mm}\RomanII_{2,5,4} \triangleq \sum_{k,l=1}^{n} \int_{\mathbb{R}^{n}}  \partial_{k}^{2}\partial_{l} b_{2}\partial_{2}b_{1}\partial_{l}\partial_{3} b_{2}dx, \\
\RomanII_{2,5,5} \triangleq& -\sum_{k,l=1}^{n} \int_{\mathbb{R}^{n}}  \partial_{k}^{2}\partial_{l} b_{2}\partial_{2}b_{3}\partial_{l}\partial_{1} b_{2}dx, \hspace{5mm}\RomanII_{2,5,6} \triangleq \sum_{k,l=1}^{n} \int_{\mathbb{R}^{n}}  \partial_{k}^{2}\partial_{l} b_{2}\partial_{2}b_{3}\partial_{l}\partial_{2} b_{1}dx, \\
\RomanII_{2,5,7} \triangleq& \sum_{k,l=1}^{n} \int_{\mathbb{R}^{n}}  \partial_{k}^{2}\partial_{l} b_{2}\partial_{3}b_{2}\partial_{l}\partial_{1} b_{2}dx, \hspace{8mm} \RomanII_{2,5,8} \triangleq -\sum_{k,l=1}^{n} \int_{\mathbb{R}^{n}}  \partial_{k}^{2}\partial_{l} b_{2}\partial_{3}b_{2}\partial_{l}\partial_{2} b_{1}dx.
\end{align}
\end{subequations}
We couple $\RomanII_{2,5,2}$ and $\RomanII_{2,5,7}$ from \eqref{est 18} to obtain 
\begin{align}
\RomanII_{2,5,2} + \RomanII_{2,5,7} \overset{\eqref{est 18}}{=}& -\sum_{k,l=1}^{n} \int_{\mathbb{R}^{n}}  \partial_{k}^{2} \partial_{l} b_{2} \partial_{1} b_{2} \partial_{l} \partial_{3} b_{2} - \partial_{k}^{2}\partial_{l} b_{2} \partial_{3} b_{2}\partial_{l} \partial_{1} b_{2} dx \label{est 20} 
\end{align}
and integrate by parts separately to obtain 
\begin{subequations}\label{est 19}
\begin{align}
&-\sum_{k,l=1}^{n} \int_{\mathbb{R}^{n}}  \partial_{k}^{2} \partial_{l} b_{2} \partial_{1} b_{2}\partial_{l} \partial_{3} b_{2} dx= \sum_{k,l=1}^{n}  \int_{\mathbb{R}^{n}}  \partial_{k} \partial_{l} b_{2}\partial_{k}\partial_{1} b_{2}\partial_{l} \partial_{3} b_{2} + \partial_{k}\partial_{l} b_{2}\partial_{1} b_{2} \partial_{k}\partial_{l}\partial_{3} b_{2}dx, \label{est 19a} \\
& \sum_{k,l=1}^{n} \int_{\mathbb{R}^{n}}  \partial_{k}^{2}\partial_{l} b_{2}\partial_{3}b_{2}\partial_{l}\partial_{1} b_{2} dx= -\sum_{k,l=1}^{n} \int_{\mathbb{R}^{n}}  \partial_{k}\partial_{l} b_{2}\partial_{k}\partial_{3} b_{2}\partial_{l}\partial_{1} b_{2} + \partial_{k}\partial_{l} b_{2} \partial_{3} b_{2} \partial_{k} \partial_{l} \partial_{1} b_{2}dx.\label{est 19b}
\end{align}
\end{subequations} 
Then the first terms in \eqref{est 19a}-\eqref{est 19b} together cancel out as follows:
\begin{equation*}
\sum_{k,l=1}^{n} \int_{\mathbb{R}^{n}}  \partial_{k} \partial_{l} b_{2}\partial_{k}\partial_{1}b_{2}\partial_{l}\partial_{3}b_{2} - \partial_{k}\partial_{l}b_{2}\partial_{k}\partial_{3}b_{2}\partial_{l}\partial_{1}b_{2} dx= 0
\end{equation*}
which can be seen by just swapping $k\leftrightarrow l$ in the second integrand. On the other hand, the second terms in \eqref{est 19a}-\eqref{est 19b} also cancel out as 
\begin{align*}
& \sum_{k,l=1}^{n} \int_{\mathbb{R}^{n}}  \partial_{k} \partial_{l} b_{2}\partial_{1}b_{2}\partial_{k}\partial_{l} \partial_{3} b_{2} - \partial_{k}\partial_{l} b_{2}\partial_{3}b_{2}\partial_{k}\partial_{l}\partial_{1} b_{2} dx \nonumber \\
=& \sum_{k,l=1}^{n}\int_{\mathbb{R}^{n}}  \frac{1}{2}\partial_{3} (\partial_{k}\partial_{l} b_{2})^{2}\partial_{1}b_{2} - \frac{1}{2} \partial_{1} (\partial_{k}\partial_{l} b_{2})^{2}\partial_{3} b_{2}dx \nonumber \\
=& - \frac{1}{2}\sum_{k,l=1}^{n}\int_{\mathbb{R}^{n}}  (\partial_{k}\partial_{l} b_{2})^{2} \partial_{1}\partial_{3} b_{2} - (\partial_{k}\partial_{l} b_{2})^{2}\partial_{1}\partial_{3} b_{2}dx = 0. 
\end{align*}
Therefore, we conclude from \eqref{est 20} that 
\begin{equation}\label{est 26}
\RomanII_{2,5,2} + \RomanII_{2,5,7} = 0. 
\end{equation}
Finally, we couple from \eqref{est 12} 
\begin{equation}\label{est 21} 
\RomanII_{4} + \RomanII_{6} \overset{\eqref{est 12}}{=} \sum_{k,l=1}^{n} \int_{\mathbb{R}^{n}}  j_{3}\partial_{k}^{2} b_{1} \partial_{l}^{2} j_{2} - j_{2} \partial_{k}^{2} b_{1}\partial_{l}^{2} j_{3} dx
\end{equation}
where we integrate by parts separately to obtain 
\begin{align*}
& \int_{\mathbb{R}^{n}}  j_{3}\partial_{k}^{2} b_{1}\partial_{l}^{2} j_{2} dx=  -\int_{\mathbb{R}^{n}}  \partial_{l} j_{3}\partial_{k}^{2} b_{1}\partial_{l} j_{2} + j_{3}\partial_{k}^{2}\partial_{l} b_{1}\partial_{l} j_{2}dx, \\
& -\int_{\mathbb{R}^{n}}  j_{2} \partial_{k}^{2} b_{1}\partial_{l}^{2} j_{3}dx = \int_{\mathbb{R}^{n}}  \partial_{l} j_{2}\partial_{k}^{2} b_{1} \partial_{l} j_{3} + j_{2}\partial_{k}^{2}\partial_{l} b_{1}\partial_{l} j_{3}dx, 
\end{align*}
so that we see a cancellation in sum, specifically $- \int_{\mathbb{R}^{n}}  \partial_{l} j_{3}\partial_{k}^{2} b_{1}\partial_{l} j_{2}dx$ and $\int_{\mathbb{R}^{n}}  \partial_{l} j_{2}\partial_{k}^{2} b_{1}\partial_{l} j_{3}dx$, leading us to 
\begin{equation}\label{est 33}
\RomanII_{4} + \RomanII_{6} \overset{\eqref{est 21}}{=}  \sum_{k,l=1}^{n}\int_{\mathbb{R}^{n}}  \partial_{k}^{2}\partial_{l} b_{1} (-j_{3}\partial_{l} j_{2} +  j_{2} \partial_{l} j_{3}) dx
\overset{\eqref{est 4}}{=}\sum_{l=1}^{8} \RomanII_{4,6,l}
\end{equation}
where 
\begin{subequations}\label{est 22} 
\begin{align}
\RomanII_{4,6,1} \triangleq& \sum_{k,l=1}^{n} \int_{\mathbb{R}^{n}}  \partial_{k}^{2}\partial_{l} b_{1}\partial_{1}b_{2}\partial_{l}\partial_{1} b_{3}dx, \hspace{7mm} \RomanII_{4,6,2} \triangleq -\sum_{k,l=1}^{n} \int_{\mathbb{R}^{n}}  \partial_{k}^{2}\partial_{l} b_{1}\partial_{1}b_{2}\partial_{l}\partial_{3} b_{1}dx, \\
\RomanII_{4,6,3} \triangleq& -\sum_{k,l=1}^{n} \int_{\mathbb{R}^{n}}  \partial_{k}^{2}\partial_{l} b_{1}\partial_{2}b_{1}\partial_{l}\partial_{1} b_{3}dx, \hspace{4mm} \RomanII_{4,6,4} \triangleq \sum_{k,l=1}^{n} \int_{\mathbb{R}^{n}}  \partial_{k}^{2}\partial_{l} b_{1}\partial_{2}b_{1}\partial_{l}\partial_{3} b_{1}dx, \\
\RomanII_{4,6,5} \triangleq& -\sum_{k,l=1}^{n} \int_{\mathbb{R}^{n}}  \partial_{k}^{2}\partial_{l} b_{1}\partial_{1}b_{3}\partial_{l}\partial_{1} b_{2}dx, \hspace{4mm}\RomanII_{4,6,6} \triangleq \sum_{k,l=1}^{n} \int_{\mathbb{R}^{n}}  \partial_{k}^{2}\partial_{l} b_{1}\partial_{1}b_{3}\partial_{l}\partial_{2} b_{1}dx, \\
\RomanII_{4,6,7} \triangleq& \sum_{k,l=1}^{n} \int_{\mathbb{R}^{n}}  \partial_{k}^{2}\partial_{l} b_{1}\partial_{3}b_{1}\partial_{l}\partial_{1} b_{2}dx, \hspace{7mm} \RomanII_{4,6,8} \triangleq -\sum_{k,l=1}^{n} \int_{\mathbb{R}^{n}}  \partial_{k}^{2}\partial_{l} b_{1}\partial_{3}b_{1}\partial_{l}\partial_{2} b_{1}dx.
\end{align}
\end{subequations} 
We couple $\RomanII_{4,6,4}$ and $\RomanII_{4,6,8}$ from \eqref{est 22} to obtain 
\begin{align}  
\RomanII_{4,6,4} + \RomanII_{4,6,8} \overset{\eqref{est 22}}{=}& \sum_{k,l=1}^{n} \int_{\mathbb{R}^{n}}  \partial_{k}^{2} \partial_{l} b_{1} \partial_{2} b_{1} \partial_{l} \partial_{3} b_{1} - \partial_{k}^{2}\partial_{l} b_{1} \partial_{3} b_{1}\partial_{l} \partial_{2} b_{1}dx \label{est 24} 
\end{align}
and integrate by parts separately to obtain 
\begin{subequations}\label{est 23}
\begin{align}
&\sum_{k,l=1}^{n} \int_{\mathbb{R}^{n}}  \partial_{k}^{2} \partial_{l} b_{1} \partial_{2} b_{1}\partial_{l} \partial_{3} b_{1}dx \nonumber\\
=& -\sum_{k,l=1}^{n}  \int_{\mathbb{R}^{n}}  \partial_{k} \partial_{l} b_{1}\partial_{k}\partial_{2} b_{1}\partial_{l} \partial_{3} b_{1} + \partial_{k}\partial_{l} b_{1}\partial_{2} b_{1} \partial_{k}\partial_{l}\partial_{3} b_{1}dx, \label{est 23a} \\
&- \sum_{k,l=1}^{n} \int_{\mathbb{R}^{n}}  \partial_{k}^{2}\partial_{l} b_{1}\partial_{3}b_{1}\partial_{l}\partial_{2} b_{1}dx \nonumber\\
=& \sum_{k,l=1}^{n} \int_{\mathbb{R}^{n}}  \partial_{k}\partial_{l} b_{1}\partial_{k}\partial_{3} b_{1}\partial_{l}\partial_{2} b_{1} + \partial_{k}\partial_{l} b_{1} \partial_{3} b_{1} \partial_{k} \partial_{l} \partial_{2} b_{1}dx.\label{est 23b}
\end{align}
\end{subequations} 
Then the first terms in \eqref{est 23a}-\eqref{est 23b} together cancel out as follows:
\begin{equation*}
-\sum_{k,l=1}^{n} \int_{\mathbb{R}^{n}}  \partial_{k} \partial_{l} b_{1}\partial_{k}\partial_{2}b_{1}\partial_{l}\partial_{3}b_{1} - \partial_{k}\partial_{l}b_{1}\partial_{k}\partial_{3}b_{1}\partial_{l}\partial_{2}b_{1} dx= 0
\end{equation*}
which can be seen by just swapping $k\leftrightarrow l$ in the second integrand. On the other hand, the second terms in \eqref{est 23a}-\eqref{est 23b} also cancel out as 
\begin{align*}
& -\sum_{k,l=1}^{n} \int_{\mathbb{R}^{n}}  \partial_{k} \partial_{l} b_{1}\partial_{2}b_{1}\partial_{k}\partial_{l} \partial_{3} b_{1} - \partial_{k}\partial_{l} b_{1}\partial_{3}b_{1}\partial_{k}\partial_{l}\partial_{2} b_{1}dx  \nonumber \\
=& -\sum_{k,l=1}^{n}\int_{\mathbb{R}^{n}}  \frac{1}{2}\partial_{3} (\partial_{k}\partial_{l} b_{1})^{2}\partial_{2}b_{1} - \frac{1}{2} \partial_{2} (\partial_{k}\partial_{l} b_{1})^{2}\partial_{3} b_{1}dx \nonumber \\
=& \frac{1}{2}\sum_{k,l=1}^{n}\int_{\mathbb{R}^{n}}  (\partial_{k}\partial_{l} b_{1})^{2} \partial_{2}\partial_{3} b_{1} - (\partial_{k}\partial_{l} b_{1})^{2}\partial_{2}\partial_{3} b_{1}dx = 0. 
\end{align*}
Therefore, we conclude from \eqref{est 24} that 
\begin{equation}\label{est 27}
\RomanII_{4,6,4} + \RomanI_{4,6,8} = 0. 
\end{equation}
In conclusion we have shown 
\begin{align}
&\int_{\mathbb{R}^{n}}  \Delta \nabla \times (j\times b) \cdot \Delta b dx \overset{\eqref{est 28} }{=} \RomanI + \RomanII \nonumber \\
\overset{\eqref{est 29} \eqref{est 30}}{=}& \sum_{i=1}^{6} \RomanI_{i} + \sum_{i=1}^{6} \RomanII_{i} \nonumber \\
=& (\RomanI_{1} + \RomanI_{3}) + (\RomanI_{2} + \RomanI_{5}) + (\RomanI_{4} + \RomanI_{6}) + (\RomanII_{1} + \RomanII_{3}) + (\RomanII_{2} + \RomanII_{5}) + (\RomanII_{4} + \RomanII_{6}) \nonumber \\
\overset{\eqref{est 5} \eqref{est 7} \eqref{est 9} \eqref{est 31} \eqref{est 32} \eqref{est 33}}{=}& \sum_{l=1}^{8} \RomanI_{1,3,l} + \RomanI_{2,5,l} + \RomanI_{4,6,l} + \RomanII_{1,3,l} + \RomanII_{2,5,l} + \RomanII_{4,6,l} \nonumber \\
\overset{\eqref{est 34} \eqref{est 35} \eqref{est 36} \eqref{est 25}\eqref{est 26}\eqref{est 27} }{=}& \sum_{l\in \{2,3,4,6,7,8\}} \RomanI_{1,3,l} + \sum_{l\in \{1,3,4,5,6,8 \}} \RomanI_{2,5,l} + \sum_{l \in \{1,2,3, 5,6,7 \}} \RomanI_{4,6,l} \nonumber \\
&+ \sum_{l\in \{2,3,4,6,7,8\}} \RomanII_{1,3,l} +  \sum_{l\in \{1,3,4,5,6,8 \}} \RomanII_{2,5,l} + \sum_{l \in \{1,2,3, 5,6,7 \}} \RomanII_{4,6,l}.  \label{est 61} 
\end{align}

We are now ready to conclude \eqref{est 58} in part (1) in the 3-D case. We compute 
\begin{align}
\sum_{l \in \{3,4,7,8\}} \RomanI_{1,3,l} &\overset{\eqref{est 6}}{=} 2 \sum_{k,l=1}^{3} \int_{\mathbb{R}^{3}}  \partial_{k} b_{3}\partial_{k}\partial_{3} b_{1}\partial_{l}^{2} \partial_{2} b_{3} - \partial_{k} b_{3} \partial_{k}\partial_{3} b_{1}\partial_{l}^{2}\partial_{3}b_{2} \nonumber\\
&- \partial_{k}b_{3}\partial_{k}\partial_{3}b_{2}\partial_{l}^{2}\partial_{1}b_{3}+ \partial_{k}b_{3}\partial_{k}\partial_{3}b_{2}\partial_{l}^{2}\partial_{3}b_{1} dx\lesssim \int_{\mathbb{R}^{3}}  \lvert \nabla b_{v} \rvert \lvert \nabla^{2} b_{h}\rvert \lvert \nabla^{3} b \rvert dx, \label{est 135} 
\end{align}
\begin{align}
\sum_{l\in \{2,6\}} \RomanI_{1,3,l} \overset{\eqref{est 6}}{=}& 2 \sum_{k,l=1}^{3}\int_{\mathbb{R}^{3}}  \partial_{k}b_{3}\partial_{k}\partial_{1} b_{3}\partial_{l}^{2}\partial_{3}b_{2} - \partial_{k} b_{3}\partial_{k}\partial_{2}b_{3}\partial_{l}^{2}\partial_{3}b_{1} dx\nonumber \\
=& -2 \sum_{k,l=1}^{3}\int_{\mathbb{R}^{3}}  \partial_{l} (\partial_{k} b_{3}\partial_{k}\partial_{1}b_{3}) \partial_{l}\partial_{3}b_{2} - \partial_{l} (\partial_{k}b_{3}\partial_{k}\partial_{2}b_{3}) \partial_{l}\partial_{3} b_{1} dx \nonumber\\
&\lesssim \int_{\mathbb{R}^{3}}  (\lvert \nabla^{2} b_{v} \rvert^{2} + \lvert \nabla b_{v} \rvert \lvert \nabla^{3} b_{v} \rvert ) \lvert \nabla^{2} b_{h} \rvert dx, 
\end{align}
\begin{align}
\sum_{l\in \{1,3,4,8\}} \RomanI_{2,5,l} &\overset{\eqref{est 8}}{=} -2 \sum_{k,l=1}^{3} \int_{\mathbb{R}^{3}}  \partial_{k} b_{2} \partial_{k}\partial_{1} b_{2}\partial_{l}^{2}\partial_{2} b_{3} - \partial_{k} b_{2}\partial_{k}\partial_{2} b_{1}\partial_{l}^{2}\partial_{2}b_{3} \nonumber \\
& + \partial_{k}b_{2}\partial_{k}\partial_{2}b_{1}\partial_{l}^{2}\partial_{3}b_{2} - \partial_{k}b_{2}\partial_{k}\partial_{3}b_{2}\partial_{l}^{2}\partial_{2}b_{1} dx\lesssim \int_{\mathbb{R}^{3}}  \lvert \nabla b_{h} \rvert \lvert \nabla^{2} b_{h} \rvert \lvert \nabla^{3} b \rvert dx, 
\end{align}
\begin{align}
\sum_{l\in \{5,6\}} \RomanI_{2,5,l} \overset{\eqref{est 8}}{=}&2\sum_{k,l=1}^{3} \int_{\mathbb{R}^{3}}  \partial_{k} b_{2}\partial_{k}\partial_{2} b_{3}\partial_{l}^{2}\partial_{1} b_{2} - \partial_{k}b_{2}\partial_{k}\partial_{2}b_{3}\partial_{l}^{2}\partial_{2}b_{1} dx\nonumber \\
=& - 2 \sum_{k,l=1}^{3}\int_{\mathbb{R}^{3}}  \partial_{l} (\partial_{k}b_{2}\partial_{k}\partial_{2} b_{3}) \partial_{l}\partial_{1} b_{2} - \partial_{l} (\partial_{k} b_{2} \partial_{k}\partial_{2} b_{3}) \partial_{l}\partial_{2} b_{1}dx \nonumber \\
\lesssim& \int_{\mathbb{R}^{3}}  ( \lvert \nabla^{2} b_{h} \rvert \lvert \nabla^{2} b_{v} \rvert  + \lvert \nabla b_{h}  \rvert \lvert \nabla^{3} b_{v} \rvert) \lvert \nabla^{2} b_{h} \rvert dx, 
\end{align}
\begin{align}
\sum_{l\in \{1,2,3,7\}} \RomanI_{4,6,l} &\overset{\eqref{est 10}}{=} -2 \sum_{k,l=1}^{3} \int_{\mathbb{R}^{3}}  \partial_{k} b_{1} \partial_{k} \partial_{1} b_{2} \partial_{l}^{2}\partial_{1} b_{3} - \partial_{k} b_{1}\partial_{k}\partial_{1} b_{2}\partial_{l}^{2}\partial_{3} b_{1}  \nonumber\\
& - \partial_{k} b_{1}\partial_{k}\partial_{2} b_{1}\partial_{l}^{2}\partial_{1} b_{3} + \partial_{k} b_{1}\partial_{k}\partial_{3} b_{1}\partial_{l}^{2} \partial_{1} b_{2} dx \lesssim\int_{\mathbb{R}^{3}}  \lvert \nabla b_{h} \rvert \lvert \nabla^{2} b_{h} \rvert \lvert \nabla^{3} b \rvert dx, 
\end{align}
and
\begin{align}
\sum_{l\in \{5,6\}} \RomanI_{4,6,l} \overset{\eqref{est 10}}{=}& 2 \sum_{k,l=1}^{3} \int_{\mathbb{R}^{3}}  \partial_{k} b_{1}\partial_{k}\partial_{1} b_{3}\partial_{l}^{2}\partial_{1}b_{2} - \partial_{k}b_{1}\partial_{k}\partial_{1} b_{3}\partial_{l}^{2}\partial_{2} b_{1} dx \nonumber \\
=& -2 \sum_{k,l=1}^{3}\int_{\mathbb{R}^{3}}  \partial_{l} (\partial_{k} b_{1}\partial_{k}\partial_{1}b_{3}) \partial_{l} \partial_{1} b_{2} - \partial_{l} (\partial_{k}b_{1}\partial_{k}\partial_{1}b_{3}) \partial_{l}\partial_{2} b_{1} dx \nonumber \\
\lesssim& \int_{\mathbb{R}^{3}}  (\lvert \nabla^{2} b_{h} \rvert \lvert \nabla^{2} b_{v} \rvert  + \lvert \nabla b_{h} \rvert \lvert \nabla^{3} b_{v} \rvert) \lvert \nabla^{2} b_{h} \rvert dx. 
\end{align}
The terms within $\RomanII$ are more difficult but can be manipulated as follows: 
\begin{align}
\sum_{l\in \{2,4,6,8\}} \RomanII_{1,3,l} &\overset{\eqref{est 14}}{=} -\sum_{k,l=1}^{3} \int_{\mathbb{R}^{3}}  \partial_{k}^{2}\partial_{l} b_{3}\partial_{1} b_{3}\partial_{l} \partial_{3}b_{2} - \partial_{k}^{2}\partial_{l}b_{3} \partial_{3}b_{1}\partial_{l}\partial_{3}b_{2} \nonumber\\
& - \partial_{k}^{2}\partial_{l} b_{3}\partial_{2}b_{3}\partial_{l}\partial_{3}b_{1} + \partial_{k}^{2}\partial_{l} b_{3} \partial_{3}b_{2}\partial_{l}\partial_{3} b_{1} dx \lesssim \int_{\mathbb{R}^{3}}  \lvert \nabla^{3} b_{v} \rvert \lvert \nabla b \rvert \lvert \nabla^{2} b_{h} \rvert dx, 
\end{align}
\begin{align}
\sum_{l\in \{3,7 \}} \RomanII_{1,3,l} \overset{\eqref{est 14}}{=}& - \sum_{k,l=1}^{3}\int_{\mathbb{R}^{3}} \partial_{k}^{2}\partial_{l}b_{3}\partial_{3}b_{1}\partial_{l}\partial_{2}b_{3} - \partial_{k}^{2}\partial_{l}b_{3}\partial_{3}b_{2}\partial_{l}\partial_{1}b_{3} dx \nonumber \\
=& \sum_{k,l=1}^{3}\int_{\mathbb{R}^{3}}  \partial_{k}\partial_{l}b_{3}\partial_{k}\partial_{3}b_{1}\partial_{l}\partial_{2} b_{3} + \partial_{k}\partial_{l} b_{3} \partial_{3} b_{1}\partial_{k}\partial_{l}\partial_{2} b_{3} \nonumber \\
& - \partial_{k}\partial_{l} b_{3} \partial_{k} \partial_{3} b_{2} \partial_{l}\partial_{1} b_{3} - \partial_{k}\partial_{l} b_{3} \partial_{3} b_{2}\partial_{k}\partial_{l}\partial_{1} b_{3}  dx \nonumber \\
=& \sum_{k,l=1}^{3} \int_{\mathbb{R}^{3}}  \partial_{k}\partial_{l} b_{3}\partial_{k}\partial_{3} b_{1} \partial_{l}\partial_{2} b_{3} + \partial_{3} b_{1} \frac{1}{2}\partial_{2} (\partial_{k} \partial_{l} b_{3})^{2} \nonumber \\
& - \partial_{k} \partial_{l} b_{3} \partial_{k}\partial_{3} b_{2} \partial_{l}\partial_{1} b_{3} - \partial_{3} b_{2} \frac{1}{2}\partial_{1} (\partial_{k} \partial_{l} b_{3})^{2} dx \nonumber \\
=& \sum_{k,l=1}^{3} \int_{\mathbb{R}^{3}}  \partial_{k} \partial_{l} b_{3} \partial_{k} \partial_{3} b_{1} \partial_{l} \partial_{2} b_{3} - \frac{1}{2}\partial_{3}\partial_{2} b_{1} (\partial_{k}\partial_{l} b_{3})^{2} \nonumber \\
& - \partial_{k}\partial_{l} b_{3}\partial_{k}\partial_{3} b_{2} \partial_{l}\partial_{1} b_{3} + \frac{1}{2} \partial_{1}\partial_{3} b_{2} (\partial_{k} \partial_{l} b_{3})^{2} dx \lesssim \int_{\mathbb{R}^{3}}  \lvert \nabla^{2} b_{v} \rvert^{2} \lvert \nabla^{2} b_{h} \rvert dx, 
\end{align}
\begin{align}
\sum_{l\in \{4,5, 6, 8 \}} \RomanII_{2,5,l} &\overset{\eqref{est 18}}{=} \sum_{k,l=1}^{3} \int_{\mathbb{R}^{3}}  \partial_{k}^{2} \partial_{l} b_{2}\partial_{2} b_{1}\partial_{l}\partial_{3} b_{2}  - \partial_{k}^{2}\partial_{l} b_{2}\partial_{2} b_{3}\partial_{l}\partial_{1} b_{2} \nonumber \\
& + \partial_{k}^{2}\partial_{l} b_{2}\partial_{2} b_{3}\partial_{l} \partial_{2} b_{1} - \partial_{k}^{2}\partial_{l} b_{2}\partial_{3} b_{2}\partial_{l}\partial_{2} b_{1} dx \lesssim \int_{\mathbb{R}^{3}}  \lvert \nabla^{3} b_{h} \rvert \lvert \nabla b \rvert \lvert \nabla^{2} b_{h} \rvert dx, 
\end{align}
\begin{align}
\sum_{l\in \{1,3\}} \RomanII_{2,5,l} \overset{\eqref{est 18}}{=}&  \sum_{k,l=1}^{3} \int_{\mathbb{R}^{3}}  \partial_{k}^{2}\partial_{l} b_{2}\partial_{1} b_{2}\partial_{l}\partial_{2} b_{3} - \partial_{k}^{2}\partial_{l} b_{2}\partial_{2} b_{1}\partial_{l}\partial_{2} b_{3} dx \nonumber \\
=& - \sum_{k,l=1}^{3}\int_{\mathbb{R}^{3}}  \partial_{k}\partial_{l} b_{2}\partial_{k} (\partial_{1} b_{2}\partial_{l}\partial_{2} b_{3}) - \partial_{k}\partial_{l} b_{2}\partial_{k} (\partial_{2}b_{1}\partial_{l}\partial_{2} b_{3}) dx\nonumber\\
\lesssim& \int_{\mathbb{R}^{3}}  \lvert \nabla^{2} b_{h} \rvert ( \lvert \nabla^{2} b_{h} \rvert \lvert \nabla^{2} b_{v} \rvert  + \lvert \nabla b_{h} \rvert \lvert \nabla^{3} b_{v} \rvert) dx, 
\end{align}
\begin{align}
\sum_{l\in \{2,5,6,7 \}} \RomanII_{4,6,l} &\overset{\eqref{est 22}}{=} - \sum_{k,l=1}^{3}\int_{\mathbb{R}^{3}}  \partial_{k}^{2}\partial_{l} b_{1}\partial_{1} b_{2}\partial_{l}\partial_{3} b_{1} + \partial_{k}^{2}\partial_{l} b_{1}\partial_{1} b_{3}\partial_{l}\partial_{1} b_{2} \nonumber \\
& - \partial_{k}^{2}\partial_{l} b_{1}\partial_{1} b_{3}\partial_{l}\partial_{2} b_{1} - \partial_{k}^{2}\partial_{l} b_{1} \partial_{3} b_{1}\partial_{l}\partial_{1} b_{2} dx \lesssim \int_{\mathbb{R}^{3}}  \lvert \nabla^{3} b_{h} \rvert \lvert \nabla b \rvert \lvert \nabla^{2} b_{h} \rvert dx, 
\end{align}
and
\begin{align}
\sum_{l\in \{1,3\}} \RomanII_{4,6,l} \overset{\eqref{est 22}}{=}& \sum_{k,l=1}^{3}\int_{\mathbb{R}^{3}}  \partial_{k}^{2}\partial_{l} b_{1}\partial_{1} b_{2}\partial_{l}\partial_{1} b_{3} - \partial_{k}^{2}\partial_{l} b_{1} \partial_{2} b_{1}\partial_{l}\partial_{1} b_{3}  dx \nonumber\\
=& - \sum_{k,l=1}^{3} \int_{\mathbb{R}^{3}}  \partial_{k}\partial_{l} b_{1}\partial_{k} (\partial_{1} b_{2}\partial_{l}\partial_{1} b_{3}) - \partial_{k}\partial_{l} b_{1} \partial_{k} (\partial_{2} b_{1} \partial_{l} \partial_{1} b_{3})  dx \nonumber\\
\lesssim& \int_{\mathbb{R}^{3}}  \lvert \nabla^{2} b_{h} \rvert (\lvert \nabla^{2} b_{h} \rvert \lvert \nabla^{2} b_{v} \rvert  + \lvert \nabla b_{h} \rvert \lvert \nabla^{3} b_{v} \rvert) dx. \label{est 136}
\end{align}
Applying \eqref{est 135}-\eqref{est 136} to \eqref{est 61}  gives \eqref{est 58}.   

Next, we prove \eqref{est 86} in the $2\frac{1}{2}$-D case. We can consider $b(t,x) = (b_{1}, b_{2}, b_{3})(t, x_{1}, x_{2})$ as the special 3-D flow that does not depend on $x_{3}$. Then we already have the identity \eqref{est 61}
\begin{align}
\int_{\mathbb{R}^{2}} \Delta \nabla \times (j\times b) \cdot \Delta b dx =& \sum_{l\in \{2,3,4,6,7,8\}} \RomanI_{1,3,l} + \sum_{l\in \{1,3,4,5,6,8 \}} \RomanI_{2,5,l} + \sum_{l \in \{1,2,3, 5,6,7 \}} \RomanI_{4,6,l} \nonumber \\
&+ \sum_{l\in \{2,3,4,6,7,8\}} \RomanII_{1,3,l} +  \sum_{l\in \{1,3,4,5,6,8 \}} \RomanII_{2,5,l} + \sum_{l \in \{1,2,3, 5,6,7 \}} \RomanII_{4,6,l}.  \label{est 65}  
\end{align}

First, we see from \eqref{est 6} that 
\begin{equation}\label{est 78} 
\sum_{l \in \{2, 3,4,6, 7,8\}} \RomanI_{1,3,l} = 0 
\end{equation} 
as they all involve $\partial_{3}$. 

Second, concerning $\sum_{l \in \{1,3,4,5,6,8 \}} \RomanI_{2,5,l}$, from \eqref{est 8} we see that $\RomanI_{2,5,4} = \RomanI_{2,5,8} = 0$ due to $\partial_{3}$ therein; on the other hand, 
 \begin{align}
\sum_{l \in \{1,3\}} \RomanI_{2,5,l}\overset{\eqref{est 8}}{=}& -2 \sum_{k,l=1}^{2} \int_{\mathbb{R}^{2}} \partial_{k} b_{2} \partial_{k} \partial_{1} b_{2}\partial_{l}^{2} \partial_{2} b_{3} - \partial_{k} b_{2}\partial_{k}\partial_{2} b_{1}\partial_{l}^{2}\partial_{2} b_{3} dx \nonumber\\
=& 2\sum_{k,l=1}^{2}\int_{\mathbb{R}^{2}} \partial_{l} (\partial_{k} b_{2}\partial_{k}\partial_{1} b_{2}) \partial_{l}\partial_{2} b_{3} - \partial_{l} (\partial_{k} b_{2}\partial_{k}\partial_{2} b_{1}) \partial_{l}\partial_{2} b_{3} dx \nonumber \\
\lesssim& \int_{\mathbb{R}^{2}} ( \lvert \nabla b_{h} \rvert \lvert \nabla^{3} b_{h} \rvert  + \lvert \nabla^{2} b_{h} \rvert^{2}) \lvert \nabla^{2} b_{v} \rvert dx. \label{est 137}
\end{align}
Additionally, 
\begin{align}
\sum_{l \in \{5,6\}} \RomanI_{2,5,l} \overset{\eqref{est 8}}{=}& 2 \sum_{k,l=1}^{2}\int_{\mathbb{R}^{2}} \partial_{k} b_{2}\partial_{k} \partial_{2} b_{3}\partial_{l}^{2}\partial_{1} b_{2} - \partial_{k} b_{2}\partial_{k}\partial_{2} b_{3}\partial_{l}^{2}\partial_{2} b_{1} dx \nonumber\\
\lesssim& \int_{\mathbb{R}^{2}}\lvert \nabla b_{h} \rvert \lvert \nabla^{3} b_{h} \rvert \lvert \nabla^{2} b_{v} \rvert dx. \label{est 79} 
\end{align}

Third, concerning $\sum_{l\in \{1,2,3,5,6,7\}} \RomanI_{4,6,l}$, from \eqref{est 10}  we see that $\RomanI_{4,6,2} = \RomanI_{4,6,7} = 0$ due to $\partial_{3}$ therein; on the other hand, 
\begin{align}
\sum_{l \in \{1,3\}} \RomanI_{4,6,l} \overset{\eqref{est 10}}{=}& - 2 \sum_{k,l=1}^{2} \int_{\mathbb{R}^{2}} \partial_{k} b_{1}\partial_{k}\partial_{1} b_{2}\partial_{l}^{2}\partial_{1} b_{3} - \partial_{k} b_{1}\partial_{k}\partial_{2} b_{1}\partial_{l}^{2} \partial_{1} b_{3} dx \nonumber \\
=& 2 \sum_{k,l=1}^{2} \int_{\mathbb{R}^{2}}\partial_{l} (\partial_{k} b_{1}\partial_{k}\partial_{1} b_{2}) \partial_{l}\partial_{1} b_{3} - \partial_{l} (\partial_{k} b_{1}\partial_{k}\partial_{2} b_{1}) \partial_{l}\partial_{1} b_{3} dx \nonumber \\
\lesssim& \int_{\mathbb{R}^{2}} (\lvert \nabla^{2} b_{h} \rvert^{2} + \lvert \nabla b_{h} \rvert \lvert \nabla^{3} b_{h} \rvert) \lvert \nabla^{2} b_{v} \rvert dx. \label{est 80}
\end{align}
Additionally,
\begin{align}
\sum_{l\in \{5,6\}} \RomanI_{4,6,l} \overset{\eqref{est 10}}{=}& 2 \sum_{k,l=1}^{2} \int_{\mathbb{R}^{2}} \partial_{k} b_{1}\partial_{k}\partial_{1} b_{3}\partial_{l}^{2}\partial_{1} b_{2} - \partial_{k} b_{1}\partial_{k}\partial_{1} b_{3} \partial_{l}^{2} \partial_{2} b_{1} dx \nonumber\\
\lesssim& \int_{\mathbb{R}^{2}} \lvert \nabla b_{h} \rvert \lvert \nabla^{3} b_{h} \rvert \lvert \nabla^{2} b_{v} \rvert dx.  \label{est 81} 
\end{align}

Fourth, from \eqref{est 14} we see that 
\begin{equation}\label{est 82} 
\sum_{l\in \{2,3,4,6,7,8\}} \RomanII_{1,3,l} =0 
\end{equation}
due to $\partial_{3}$ therein. 

Fifth, concerning $\sum_{l \in \{1,3,4,5,6,8 \}} \RomanII_{2,5,l}$, we see from \eqref{est 18} that $\RomanII_{2,5,4} = \RomanII_{2,5,8} = 0$ due to $\partial_{3}$ therein. We estimate 
\begin{align}
\sum_{l \in \{1,3\}} \RomanII_{2,5,l} \overset{\eqref{est 18}}{=}& \sum_{k,l=1}^{2} \int_{\mathbb{R}^{2}}\partial_{k}^{2}\partial_{l} b_{2}\partial_{1} b_{2}\partial_{l}\partial_{2} b_{3} - \partial_{k}^{2}\partial_{l} b_{2}\partial_{2} b_{1} \partial_{l}\partial_{2} b_{3} dx \nonumber\\
\lesssim& \int_{\mathbb{R}^{2}} \lvert \nabla^{3} b_{h} \rvert \lvert \nabla b_{h} \rvert \lvert \nabla^{2} b_{v} \rvert dx. \label{est 83} 
\end{align}
We only write $\sum_{l \in \{5,6\}} \RomanII_{2,5,l}$ in detail as 
\begin{equation}\label{est 67} 
\sum_{l \in \{5,6\}} \RomanII_{2,5,l} \overset{\eqref{est 18}}{=} - \sum_{k,l=1}^{2} \int_{\mathbb{R}^{2}}\partial_{k}^{2}\partial_{l} b_{2}\partial_{2} b_{3}\partial_{l}\partial_{1} b_{2} - \partial_{k}^{2}\partial_{l} b_{2} \partial_{2} b_{3}\partial_{l}\partial_{2} b_{1} dx = \sum_{l=1}^{8} \RomanV_{l} 
\end{equation} 
where 
\begin{subequations}\label{est 69} 
\begin{align}
\RomanV_{1} \triangleq& - \int_{\mathbb{R}^{2}} \partial_{1}^{2} \partial_{1} b_{2}\partial_{2} b_{3} \partial_{1}\partial_{1} b_{2} dx, \hspace{9mm} \RomanV_{2} \triangleq \int_{\mathbb{R}^{2}} \partial_{1}^{2}\partial_{1} b_{2}\partial_{2} b_{3}\partial_{1}\partial_{2} b_{1} dx, \\
\RomanV_{3} \triangleq& - \int_{\mathbb{R}^{2}} \partial_{1}^{2}\partial_{2}b_{2}\partial_{2} b_{3} \partial_{2}\partial_{1} b_{2} dx, \hspace{9mm} \RomanV_{4} \triangleq \int_{\mathbb{R}^{2}} \partial_{1}^{2}\partial_{2} b_{2}\partial_{2} b_{3}\partial_{2}\partial_{2} b_{1} dx, \\
\RomanV_{5} \triangleq& -\int_{\mathbb{R}^{2}} \partial_{2}^{2}\partial_{1} b_{2}\partial_{2} b_{3}\partial_{1}\partial_{1} b_{2} dx, \hspace{9mm} \RomanV_{6} \triangleq \int_{\mathbb{R}^{2}} \partial_{2}^{2}\partial_{1} b_{2}\partial_{2} b_{3}\partial_{1}\partial_{2} b_{1} dx, \\
\RomanV_{7} \triangleq& -\int_{\mathbb{R}^{2}} \partial_{2}^{2}\partial_{2} b_{2}\partial_{2} b_{3} \partial_{2}\partial_{1} b_{2} dx, \hspace{9mm} \RomanV_{8} \triangleq \int_{\mathbb{R}^{2}} \partial_{2}^{2}\partial_{2} b_{2}\partial_{2} b_{3}\partial_{2}\partial_{2} b_{1} dx, 
\end{align} 
\end{subequations} 
in which $\RomanV_{1}$ and $\RomanV_{2}, \RomanV_{3}$ and $\RomanV_{4}, \RomanV_{5}$ and $\RomanV_{6}$, as well as $\RomanV_{7}$ and $\RomanV_{8}$ correspond to the terms $(k,l) = (1,1), (1,2), (2,1),$ and $(2,2)$, respectively; we will come back to treat them subsequently. 

Sixth, concerning $\sum_{l \in \{1,2,3,5,6, 7 \}} \RomanII_{4,6,l}$, we see from \eqref{est 22} that $\RomanII_{4,6,2} = \RomanII_{4,6,7} = 0$ due to $\partial_{3}$ therein. We estimate 
\begin{align}
\sum_{l \in \{1,3 \}} \RomanII_{4,6,l} \overset{\eqref{est 22}}{=}& \sum_{k,l=1}^{2}\int_{\mathbb{R}^{2}} \partial_{k}^{2}\partial_{l} b_{1}\partial_{1} b_{2}\partial_{l}\partial_{1} b_{3} - \partial_{k}^{2}\partial_{l} b_{1} \partial_{2} b_{1}\partial_{l}\partial_{1} b_{3} dx \nonumber \\
\lesssim& \int_{\mathbb{R}^{2}} \lvert \nabla^{3} b_{h} \rvert \lvert \nabla b_{h} \rvert \lvert \nabla^{2} b_{v} \rvert dx. \label{est 84} 
\end{align}
On the other hand, we write out the remaining terms as 
\begin{equation}\label{est 66} 
\sum_{l \in \{5,6\}} \RomanII_{4,6,l} \overset{\eqref{est 22}}{=} -\sum_{k,l=1}^{2}\int_{\mathbb{R}^{2}} \partial_{k}^{2}\partial_{l} b_{1}\partial_{1} b_{3} \partial_{l} \partial_{1} b_{2} - \partial_{k}^{2} \partial_{l} b_{1} \partial_{1} b_{3} \partial_{l}\partial_{2} b_{1} dx  = \sum_{l=1}^{8} \RomanVI_{l} 
\end{equation} 
where 
\begin{subequations}\label{est 70} 
\begin{align}
\RomanVI_{1} \triangleq& - \int_{\mathbb{R}^{2}} \partial_{1}^{2} \partial_{1} b_{1}\partial_{1} b_{3} \partial_{1}\partial_{1} b_{2} dx, \hspace{9mm} \RomanVI_{2} \triangleq \int_{\mathbb{R}^{2}} \partial_{1}^{2}\partial_{1} b_{1}\partial_{1} b_{3}\partial_{1}\partial_{2} b_{1} dx, \\
\RomanVI_{3} \triangleq& - \int_{\mathbb{R}^{2}} \partial_{1}^{2}\partial_{2}b_{1}\partial_{1} b_{3} \partial_{2}\partial_{1} b_{2} dx, \hspace{9mm} \RomanVI_{4} \triangleq \int_{\mathbb{R}^{2}} \partial_{1}^{2}\partial_{2} b_{1}\partial_{1} b_{3}\partial_{2}\partial_{2} b_{1} dx, \\
\RomanVI_{5} \triangleq& -\int_{\mathbb{R}^{2}} \partial_{2}^{2}\partial_{1} b_{1}\partial_{1} b_{3}\partial_{1}\partial_{1} b_{2} dx, \hspace{9mm} \RomanVI_{6} \triangleq \int_{\mathbb{R}^{2}} \partial_{2}^{2}\partial_{1} b_{1}\partial_{1} b_{3}\partial_{1}\partial_{2} b_{1} dx, \\
\RomanVI_{7} \triangleq& -\int_{\mathbb{R}^{2}} \partial_{2}^{2}\partial_{2} b_{1}\partial_{1} b_{3} \partial_{2}\partial_{1} b_{2} dx, \hspace{9mm} \RomanVI_{8} \triangleq \int_{\mathbb{R}^{2}} \partial_{2}^{2}\partial_{2} b_{1}\partial_{1} b_{3}\partial_{2}\partial_{2} b_{1} dx, 
\end{align} 
\end{subequations} 
in which $\RomanVI_{1}$ and $\RomanVI_{2}, \RomanVI_{3}$ and $\RomanVI_{4}, \RomanVI_{5}$ and $\RomanVI_{6}$, and $\RomanVI_{7}$ and $\RomanVI_{8}$ correspond to the terms $(k,l) = (1,1), (1,2), (2,1),$ and $(2,2)$, respectively. We can estimate 
\begin{subequations}\label{est 72}
\begin{align}
\RomanV_{1} \overset{\eqref{est 69}}{=}& -\int_{\mathbb{R}^{2}} \partial_{1}^{2}\partial_{1} b_{2} \partial_{2} b_{3} \partial_{1}\partial_{1} b_{2} dx = - \int_{\mathbb{R}^{2}} \partial_{2} b_{3} \frac{1}{2} \partial_{1} (\partial_{1}^{2} b_{2})^{2} dx \nonumber \\
=& \int_{\mathbb{R}^{2}} \partial_{1}\partial_{2} b_{3} \frac{1}{2} (\partial_{1}^{2} b_{2})^{2} dx \lesssim \int_{\mathbb{R}^{2}} \lvert \nabla^{2} b_{v} \rvert \lvert \nabla^{2} b_{h} \rvert^{2} dx, \\
\RomanV_{3} \overset{\eqref{est 69}}{=}& - \int_{\mathbb{R}^{2}} \partial_{1}^{2}\partial_{2} b_{2}\partial_{2} b_{3}\partial_{2}\partial_{1} b_{2} dx = - \int_{\mathbb{R}^{2}} \partial_{2} b_{3} \frac{1}{2} \partial_{1} (\partial_{1}\partial_{2} b_{2})^{2} dx\nonumber\\
=& \int_{\mathbb{R}^{2}} \partial_{1} \partial_{2} b_{3} \frac{1}{2} (\partial_{2}\partial_{1} b_{2})^{2} dx \lesssim \int_{\mathbb{R}^{2}} \lvert \nabla^{2} b_{v} \rvert \lvert \nabla^{2} b_{h} \rvert^{2} dx, \\
\RomanVI_{6} \overset{\eqref{est 70}}{=}& \int_{\mathbb{R}^{2}} \partial_{2}^{2}\partial_{1} b_{1}\partial_{1}b_{3}\partial_{1}\partial_{2} b_{1} dx = \int_{\mathbb{R}^{2}} \partial_{1} b_{3} \frac{1}{2} \partial_{2} (\partial_{1}\partial_{2} b_{1})^{2} dx \nonumber \\
=& -\int_{\mathbb{R}^{2}} \partial_{2}\partial_{1} b_{3} \frac{1}{2} (\partial_{1}\partial_{2} b_{1})^{2} dx \lesssim \int_{\mathbb{R}^{2}} \lvert \nabla^{2} b_{v} \rvert \lvert \nabla^{2} b_{h} \rvert^{2} dx, \\
\RomanVI_{8} \overset{\eqref{est 70}}{=}& \int_{\mathbb{R}^{2}} \partial_{2}^{2}\partial_{2} b_{1}\partial_{1} b_{3}\partial_{2}\partial_{2} b_{1} dx = \int_{\mathbb{R}^{2}} \partial_{1} b_{3} \frac{1}{2} \partial_{2} (\partial_{2}^{2} b_{1})^{2} dx \nonumber \\
=& - \int_{\mathbb{R}^{2}} \partial_{2}\partial_{1} b_{3} \frac{1}{2} (\partial_{2}^{2} b_{1})^{2} dx \lesssim \int_{\mathbb{R}^{2}} \lvert \nabla^{2} b_{v} \rvert \lvert \nabla^{2} b_{h} \rvert^{2} dx. 
\end{align}
\end{subequations}
Next, we use divergence-free condition of $\partial_{1}b_{1} = - \partial_{2} b_{2}$ to estimate 
\begin{subequations}\label{est 73} 
\begin{align}
\RomanV_{6} \overset{\eqref{est 69}}{=}& \int_{\mathbb{R}^{2}} \partial_{2}^{2} \partial_{1} b_{2} \partial_{2} b_{3} \partial_{1}\partial_{2} b_{1} dx = -\int_{\mathbb{R}^{2}} \partial_{2} \partial_{1}^{2} b_{1} \partial_{2} b_{3}\partial_{1}\partial_{2} b_{1} dx \\
=& -\int_{\mathbb{R}^{2}} \partial_{2} b_{3} \frac{1}{2} \partial_{1} (\partial_{1} \partial_{2} b_{1})^{2} dx = \int_{\mathbb{R}^{2}} \partial_{1}\partial_{2} b_{3} \frac{1}{2} (\partial_{1}\partial_{2} b_{1})^{2} dx \lesssim \int_{\mathbb{R}^{2}} \lvert \nabla^{2} b_{v} \rvert \lvert \nabla^{2} b_{h} \rvert^{2} dx, \nonumber \\
\RomanV_{8} \overset{\eqref{est 69}}{=}& \int_{\mathbb{R}^{2}} \partial_{2}^{2} \partial_{2} b_{2} \partial_{2} b_{3} \partial_{2}^{2} b_{1} dx= -\int_{\mathbb{R}^{2}} \partial_{2}^{2} \partial_{1} b_{1}\partial_{2} b_{3}\partial_{2}^{2} b_{1} dx  \\
=& -\int_{\mathbb{R}^{2}} \partial_{2} b_{3} \frac{1}{2} \partial_{1} (\partial_{2}^{2} b_{1})^{2} dx =\int_{\mathbb{R}^{2}} \partial_{1}\partial_{2} b_{3} \frac{1}{2} (\partial_{2}^{2} b_{1})^{2} dx \lesssim \int_{\mathbb{R}^{2}} \lvert \nabla^{2} b_{v} \rvert \lvert \nabla^{2} b_{h} \rvert^{2} dx, \nonumber \\
\RomanVI_{1} \overset{\eqref{est 70}}{=}& - \int_{\mathbb{R}^{2}} \partial_{1}^{2} \partial_{1} b_{1}\partial_{1} b_{3}\partial_{1}\partial_{1} b_{2} dx = \int_{\mathbb{R}^{2}} \partial_{1}^{2}\partial_{2} b_{2}\partial_{1} b_{3} \partial_{1}^{2} b_{2} dx \\
=& \int_{\mathbb{R}^{2}} \partial_{1} b_{3} \frac{1}{2} \partial_{2} (\partial_{1}^{2} b_{2}) dx = -\int_{\mathbb{R}^{2}} \partial_{2}\partial_{1} b_{3} \frac{1}{2} (\partial_{1}^{2} b_{2})^{2} dx \lesssim \int_{\mathbb{R}^{2}} \lvert \nabla^{2} b_{v} \rvert \lvert \nabla^{2} b_{h} \rvert^{2} dx, \nonumber \\
\RomanVI_{3} \overset{\eqref{est 70}}{=}& -\int_{\mathbb{R}^{2}} \partial_{1}^{2}\partial_{2} b_{1} \partial_{1} b_{3} \partial_{2}\partial_{1} b_{2} dx = \int_{\mathbb{R}^{2}} \partial_{1}\partial_{2}^{2} b_{2} \partial_{1} b_{3} \partial_{2}\partial_{1} b_{2} dx \\
=& \int_{\mathbb{R}^{2}} \partial_{1} b_{3} \frac{1}{2} \partial_{2} (\partial_{1}\partial_{2} b_{2})^{2} dx = - \int_{\mathbb{R}^{2}} \partial_{2}\partial_{1} b_{3} \frac{1}{2} (\partial_{1}\partial_{2} b_{2})^{2} dx \lesssim \int_{\mathbb{R}^{2}} \lvert \nabla^{2} b_{v} \rvert \lvert \nabla^{2} b_{h} \rvert^{2} dx. \nonumber 
\end{align}
\end{subequations}
Next, we combine $\RomanV_{2}$ and $\RomanV_{5}$ and integrate by parts to shift ``$\partial_{1}$'' within $\RomanV_{5}$ to obtain 
\begin{align}
\RomanV_{2} + \RomanV_{5} \overset{\eqref{est 69}}{=}& \int_{\mathbb{R}^{2}} \partial_{1}^{2} \partial_{1} b_{2}\partial_{2} b_{3}\partial_{1}\partial_{2} b_{1} - \partial_{2}^{2}\partial_{1} b_{2}\partial_{2}b_{3}\partial_{1}\partial_{1} b_{2} dx \nonumber \\
=& \int_{\mathbb{R}^{2}} \partial_{1}^{3} b_{2}\partial_{2} b_{3}\partial_{1}\partial_{2} b_{1} + \partial_{2}^{2} b_{2}\partial_{1}\partial_{2} b_{3}\partial_{1}^{2} b_{2} + \partial_{2}^{2}b_{2} \partial_{2} b_{3} \partial_{1}^{3} b_{2} dx, \label{est 71}  
\end{align}
we use divergence-free condition of $\partial_{2}b_{2} = -\partial_{1} b_{1}$ in the third term of \eqref{est 71} to rewrite it as 
\begin{align*}
\int_{\mathbb{R}^{2}} \partial_{2}^{2}b_{2}\partial_{2}b_{3}\partial_{1}^{3} b_{2} dx = -\int_{\mathbb{R}^{2}} \partial_{1}\partial_{2} b_{1} \partial_{2} b_{3} \partial_{1}^{3} b_{2} dx 
\end{align*}
and realize that this cancels out the first term in \eqref{est 71} so that \eqref{est 71} simplifies to 
\begin{equation}\label{est 74} 
\RomanV_{2} + \RomanV_{5}  = \int_{\mathbb{R}^{2}} \partial_{2}^{2} b_{2}\partial_{1}\partial_{2} b_{3}\partial_{1}^{2} b_{2} dx \lesssim \int_{\mathbb{R}^{2}} \lvert \nabla^{2} b_{h} \rvert^{2} \lvert \nabla^{2} b_{v} \rvert dx. 
\end{equation} 
We will discover three more similar cancellations in \eqref{est 75}-\eqref{est 77}, for which we use divergence-free condition first and then integrate by parts for convenience. We work on $\RomanV_{4} + \RomanV_{7}$ as follows:  
\begin{align}
\RomanV_{4} + \RomanV_{7} \overset{\eqref{est 69}}{=}& \int_{\mathbb{R}^{2}} \partial_{1}^{2}\partial_{2} b_{2} \partial_{2} b_{3}\partial_{2}\partial_{2} b_{1} - \partial_{2}^{2}\partial_{2} b_{2} \partial_{2} b_{3} \partial_{2} \partial_{1} b_{2} dx \nonumber \\
=& \int_{\mathbb{R}^{2}} \partial_{1}^{2}\partial_{2} b_{2} \partial_{2} b_{3} \partial_{2}^{2} b_{1} + \partial_{2}^{2} \partial_{1} b_{1} \partial_{2} b_{3} \partial_{2}\partial_{1} b_{2} dx \nonumber \\
=& \int_{\mathbb{R}^{2}} \partial_{1}^{2} \partial_{2} b_{2}\partial_{2} b_{3} \partial_{2}^{2} b_{1} - \partial_{2}^{2} b_{1} \partial_{1}\partial_{2} b_{3} \partial_{2}\partial_{1} b_{2} - \partial_{2}^{2} b_{1} \partial_{2} b_{3} \partial_{2}\partial_{1}^{2} b_{2} dx \nonumber \\
=& -\int_{\mathbb{R}^{2}} \partial_{2}^{2} b_{1} \partial_{1} \partial_{2} b_{3} \partial_{2} \partial_{1} b_{2} dx \lesssim \int_{\mathbb{R}^{2}} \lvert \nabla^{2} b_{h} \rvert^{2} \lvert \nabla^{2} b_{v} \rvert dx. \label{est 75} 
\end{align}
Next, we work on $\RomanVI_{2} + \RomanVI_{5}$ as follows:
\begin{align} 
\RomanVI_{2} + \RomanVI_{5} \overset{\eqref{est 70}}{=}& \int_{\mathbb{R}^{2}} \partial_{1}^{2} \partial_{1} b_{1} \partial_{1} b_{3} \partial_{1} \partial_{2} b_{1} - \partial_{2}^{2} \partial_{1} b_{1} \partial_{1} b_{3} \partial_{1} \partial_{1} b_{2} dx \nonumber \\
=& \int_{\mathbb{R}^{2}} ( - \partial_{1}^{2} \partial_{2} b_{2} \partial_{1} b_{3} \partial_{1} \partial_{2} b_{1}) - \partial_{2}^{2}\partial_{1} b_{1} \partial_{1} b_{3}\partial_{1}^{2} b_{2} dx \nonumber \\
=& \int_{\mathbb{R}^{2}} \partial_{1}^{2} b_{2} \partial_{1} \partial_{2} b_{3} \partial_{1} \partial_{2} b_{1} + \partial_{1}^{2} b_{2} \partial_{1} b_{3}\partial_{2}\partial_{1} \partial_{2} b_{1} - \partial_{2}^{2}\partial_{1} b_{1} \partial_{1} b_{3} \partial_{1}^{2} b_{2} dx \nonumber \\
=& \int_{\mathbb{R}^{2}} \partial_{1}^{2} b_{2} \partial_{1} \partial_{2} b_{3} \partial_{1} \partial_{2} b_{1} dx \lesssim \int_{\mathbb{R}^{2}} \lvert \nabla^{2} b_{h} \rvert^{2} \lvert \nabla^{2} b_{v} \rvert dx. \label{est 76} 
\end{align}
Finally, we work on $\RomanVI_{4}+  \RomanVI_{7}$ as follows: 
\begin{align}
\RomanVI_{4} +\RomanVI_{7}  \overset{\eqref{est 70}}{=}& \int_{\mathbb{R}^{2}} \partial_{1}^{2}\partial_{2} b_{1}\partial_{1} b_{3}\partial_{2}\partial_{2} b_{1} - \partial_{2}^{2}\partial_{2} b_{1}\partial_{1} b_{3} \partial_{2}\partial_{1} b_{2} dx \nonumber \\
=& \int_{\mathbb{R}^{2}} (-\partial_{1}\partial_{2}^{2} b_{2} \partial_{1} b_{3}\partial_{2}^{2} b_{1}) - \partial_{2}^{3} b_{1}\partial_{1} b_{3}\partial_{2} \partial_{1} b_{2} dx \nonumber \\
=& \int_{\mathbb{R}^{2}} \partial_{1}\partial_{2} b_{2} \partial_{2}\partial_{1} b_{3} \partial_{2}^{2} b_{1} + \partial_{1}\partial_{2} b_{2} \partial_{1} b_{3}\partial_{2}^{3} b_{1} - \partial_{2}^{3} b_{1} \partial_{1} b_{3}\partial_{2}\partial_{1} b_{2} dx \nonumber \\
=& \int_{\mathbb{R}^{2}} \partial_{1}\partial_{2} b_{2}\partial_{2}\partial_{1} b_{3} \partial_{2}^{2} b_{1} dx \lesssim \int_{\mathbb{R}^{2}} \lvert \nabla^{2} b_{h} \rvert^{2} \lvert \nabla^{2} b_{v} \rvert dx. \label{est 77} 
\end{align}
 Applying \eqref{est 72}, \eqref{est 73}, \eqref{est 74}-\eqref{est 77} to \eqref{est 67}  and \eqref{est 66} to deduce 
\begin{align}
\sum_{l \in \{5,6\}} \RomanII_{2,5,l} + \sum_{l \in \{5,6\}} \RomanII_{4,6,l}  \overset{\eqref{est 67} \eqref{est 66}}{=}& \sum_{l=1}^{8} \RomanV_{l} + \RomanVI_{l} \nonumber\\
\overset{\eqref{est 72} \eqref{est 73} \eqref{est 74} \eqref{est 75} \eqref{est 76} \eqref{est 77} }{\lesssim}& \int_{\mathbb{R}^{2}} \lvert \nabla^{2} b_{h} \rvert^{2} \lvert \nabla^{2} b_{v} \rvert dx. \label{est 85}
\end{align}  
Applying \eqref{est 78}-\eqref{est 83}, \eqref{est 84}, and \eqref{est 85} to \eqref{est 65} gives 
\begin{equation*}
\int_{\mathbb{R}^{2}} \Delta \nabla \times (j\times b) \cdot \Delta b dx \lesssim \int_{\mathbb{R}^{2}} ( \lvert \nabla b_{h} \rvert \lvert \nabla^{3} b_{h} \rvert + \lvert \nabla^{2} b_{h} \rvert^{2}) \lvert \nabla^{2} b_{v} \rvert dx
\end{equation*} 
which is \eqref{est 86} as desired. 
\end{proof} 

\section{Proof of Theorem \ref{Theorem 2.1}}\label{Section 4}
We recall that the smooth solution to the 3-D Hall-MHD  \eqref{est 1} system satisfies 
\begin{equation*}
u, b \in L_{T}^{\infty} L_{x}^{2} \cap L_{T}^{2} H_{x}^{1}. 
\end{equation*} 
As we described in Subsection \ref{Subsection 1.3}, it suffices to prove the $H^{2}(\mathbb{R}^{3})$-bound.  The crux of the proof is Proposition \ref{Proposition 3.1} that allows our hypothesis \eqref{est 49b} on the horizontal components of the magnetic field to deduce an $H^{2}(\mathbb{R}^{3})$-bound. Nonetheless, because our condition on $u_{h}$ in \eqref{est 49a} is too weak to immediately deduce an $H^{2}(\mathbb{R}^{3})$-estimate, we start with an $H^{1}(\mathbb{R}^{3})$-estimate first. 

\begin{proposition}\label{Proposition 4.1} 
Under the hypothesis of Theorem \ref{Theorem 2.1}, suppose that $(u,b)$ is a smooth solution to the 3-D Hall-MHD system \eqref{est 1} over $[0,T]$. Then 
\begin{equation*}
u,b \in L_{T}^{\infty} H_{x}^{1} \cap L_{T}^{2} H_{x}^{2}. 
\end{equation*} 
\end{proposition} 

\begin{proof}[Proof of Proposition \ref{Proposition 4.1}]
We take $L^{2}(\mathbb{R}^{3})$-inner products on \eqref{est 1a}-\eqref{est 1b} with $(-\Delta u, -\Delta b)$ to deduce 
\begin{equation}\label{est 54} 
 \frac{1}{2} \frac{d}{dt} ( \lVert \nabla u \rVert_{L^{2}}^{2} + \lVert \nabla b \rVert_{L^{2}}^{2}) +  \lVert \Delta u \rVert_{L^{2}}^{2} + \lVert \Delta b \rVert_{L^{2}}^{2} = \sum_{l=1}^{5} \RomanIII_{l}, 
\end{equation} 
where 
\begin{subequations}\label{est 50} 
\begin{align}
& \RomanIII_{1} \triangleq - \sum_{k=1}^{3} \int_{\mathbb{R}^{3}} (\partial_{k} u \cdot \nabla) u \cdot \partial_{k} u dx, \hspace{3mm}  \RomanIII_{2} \triangleq - \sum_{k=1}^{3}\int_{\mathbb{R}^{3}}  (\partial_{k} u \cdot\nabla) b\cdot \partial_{k} b dx, \\
& \RomanIII_{3} \triangleq \sum_{k=1}^{3}\int_{\mathbb{R}^{3}} (\partial_{k} b \cdot\nabla) b \cdot \partial_{k} u dx, \hspace{6mm}  \RomanIII_{4} \triangleq \sum_{k=1}^{3} \int_{\mathbb{R}^{3}} (\partial_{k} b\cdot\nabla) u \cdot\partial_{k} b dx, \\
& \RomanIII_{5} \triangleq -\int_{\mathbb{R}^{3}} \nabla\times (j\times b) \cdot \Delta b dx. 
\end{align}
\end{subequations}
First, we work on $\RomanIII_{1}$ and rewrite it from \eqref{est 50} as 
\begin{align}
\RomanIII_{1} =& -\sum_{i=1}^{2} \sum_{j,k=1}^{3} \int_{\mathbb{R}^{3}}\partial_{k} u_{i}\partial_{i} u_{j} \partial_{k} u_{j} dx \nonumber\\
& - \sum_{j=1}^{2}\sum_{k=1}^{3}\int_{\mathbb{R}^{3}}\partial_{k} u_{3} \partial_{3} u_{j} \partial_{k} u_{j} dx - \sum_{k=1}^{3} \int_{\mathbb{R}^{3}} \partial_{k} u_{3} \partial_{3} u_{3}\partial_{k} u_{3} dx \nonumber \\
=& -\sum_{i=1}^{2} \sum_{j,k=1}^{3} \int_{\mathbb{R}^{3}}\partial_{k} u_{i}\partial_{i} u_{j} \partial_{k} u_{j} dx  \nonumber \\
& - \sum_{j=1}^{2}\sum_{k=1}^{3}\int_{\mathbb{R}^{3}}\partial_{k} u_{3} \partial_{3} u_{j} \partial_{k} u_{j} dx + \sum_{k=1}^{3} \sum_{l=1}^{2} \int_{\mathbb{R}^{3}} \partial_{k} u_{3} \partial_{l} u_{l}\partial_{k} u_{3} dx  \label{est 138} 
\end{align}
where we used the divergence-free condition so that $\partial_{3}u_{3} = -\sum_{l=1}^{2}\partial_{l}u_{l}$. Thus, we are ready to integrate by parts, apply H$\ddot{\mathrm{o}}$lder's, Gagliardo-Nirenberg, and Young's inequalities to estimate 
\begin{equation}\label{est 51} 
\RomanIII_{1} \lesssim \int_{\mathbb{R}^{3}} \lvert u_{h} \rvert \lvert \nabla u \rvert \lvert \nabla^{2} u \rvert dx \lesssim\lVert u_{h} \rVert_{L^{p_{1}}} \lVert \nabla u \rVert_{L^{\frac{2p_{1}}{p_{1} -2}}} \lVert \Delta u \rVert_{L^{2}} \leq \frac{1}{8} \lVert \Delta u \rVert_{L^{2}}^{2} + C \lVert u_{h} \rVert_{L^{p_{1}}}^{\frac{2p_{1}}{p_{1} -3}} \lVert \nabla u \rVert_{L^{2}}^{2} 
\end{equation} 
where we understand $\frac{2p_{1}}{p_{1} - 2} = 2, \frac{2p_{1}}{p_{1} -3} = 2$ in case $p_{1} = \infty$. Similarly, we can rewrite from \eqref{est 50} 
\begin{align}
\RomanIII_{2} =&  - \sum_{i=1}^{2} \sum_{j,k=1}^{3} \int_{\mathbb{R}^{3}} \partial_{k} u_{i} \partial_{i} b_{j} \partial_{k} b_{j} dx \nonumber\\
&- \sum_{j=1}^{2}\sum_{k=1}^{3} \int_{\mathbb{R}^{3}}\partial_{k} u_{3} \partial_{3} b_{j} \partial_{k} b_{j} dx + \sum_{k=1}^{3}\sum_{l=1}^{2} \int_{\mathbb{R}^{3}} \partial_{k} u_{3} \partial_{l} b_{l}\partial_{k} b_{ 3} dx.
\end{align}
We integrate by parts, apply H$\ddot{\mathrm{o}}$lder's, Gagliardo-Nirenberg, and Young's inequalities to estimate, slightly differently from \eqref{est 51} as  
\begin{align}
\RomanIII_{2} \lesssim& \lVert u_{h} \rVert_{L^{p_{1}}} \lVert \nabla b \rVert_{L^{\frac{2p_{1}}{p_{1} -2}}} \lVert \Delta b \rVert_{L^{2}} + \lVert b_{h} \rVert_{L^{\infty}} ( \lVert \nabla u \rVert_{L^{2}} + \lVert \nabla b \rVert_{L^{2}}) ( \lVert \Delta u \rVert_{L^{2}} + \lVert \Delta b \rVert_{L^{2}}) \nonumber \\
\lesssim& \lVert u_{h} \rVert_{L^{p_{1}}} \lVert \nabla b \rVert_{L^{2}}^{\frac{p_{1} -3}{p_{1}}} \lVert \nabla^{2} b \rVert_{L^{2}}^{\frac{3}{p_{1}}} \lVert \Delta b \rVert_{L^{2}} \nonumber  \\
& + \lVert b_{h} \rVert_{L^{2}}^{\frac{2(2p_{2} -3)}{7p_{2} -6}} \lVert \nabla^{2} b_{h} \rVert_{L^{p_{2}}}^{\frac{3p_{2}}{7p_{2} - 6}} ( \lVert \nabla u \rVert_{L^{2}} + \lVert \nabla b \rVert_{L^{2}}) ( \lVert \Delta u \rVert_{L^{2}} + \lVert \Delta b \rVert_{L^{2}})  \nonumber \\
\leq& \frac{1}{8} ( \lVert \Delta u \rVert_{L^{2}}^{2} + \lVert \Delta b \rVert_{L^{2}}^{2}) + C ( \lVert u_{h} \rVert_{L^{p_{1}}}^{\frac{2p_{1}}{p_{1} -3}} + \lVert \nabla^{2} b_{h} \rVert_{L^{p_{2}}}^{\frac{6p_{2}}{7p_{2} - 6}} )  ( \lVert \nabla u\rVert_{L^{2}}^{2} + \lVert \nabla b \rVert_{L^{2}}^{2}). \label{est 52} 
\end{align}
Finally, we can rewrite $\RomanIII_{3} + \RomanIII_{4}$ from \eqref{est 50} together as 
\begin{align}
&\RomanIII_{3} + \RomanIII_{4} =  \sum_{i=1}^{2} \sum_{j,k=1}^{3}\int_{\mathbb{R}^{3}} \partial_{k} b_{i}\partial_{i} b_{j}\partial_{k} u_{j} + \partial_{k} b_{i} \partial_{i} u_{j} \partial_{k} b_{j}  dx \\
& + \sum_{j=1}^{2}\sum_{k=1}^{3}  \int_{\mathbb{R}^{3}} \partial_{k} b_{3}\partial_{3} b_{j} \partial_{k} u_{j} + \partial_{k} b_{3}\partial_{3} u_{j} \partial_{k} b_{j} dx  - \sum_{k=1}^{3} \sum_{l=1}^{2}  \int_{\mathbb{R}^{3}} \partial_{k} b_{3}\partial_{l} b_{l}\partial_{k} u_{3} + \partial_{k} b_{3}\partial_{l} u_{l}\partial_{k} b_{3} dx \nonumber 
\end{align}
and estimate identically to \eqref{est 52} 
\begin{equation}\label{est 55} 
 \RomanIII_{3} + \RomanIII_{4}  \leq \frac{1}{8} ( \lVert \Delta u \rVert_{L^{2}}^{2} + \lVert \Delta b \rVert_{L^{2}}^{2}) + C ( \lVert u_{h} \rVert_{L^{p_{1}}}^{\frac{2p_{1}}{p_{1} -3}} + \lVert \nabla^{2} b_{h} \rVert_{L^{p_{2}}}^{\frac{6p_{2}}{7p_{2} - 6}} )  ( \lVert \nabla u\rVert_{L^{2}}^{2} + \lVert \nabla b \rVert_{L^{2}}^{2}).
\end{equation} 
At last, we rely on \eqref{est 53} to handle the Hall term $\RomanIII_{5}$ as follows: via H$\ddot{\mathrm{o}}$lder's inequality, the Sobolev embedding $\dot{H}^{1}(\mathbb{R}^{3}) \hookrightarrow L^{6}(\mathbb{R}^{3})$, Gagliardo-Nirenberg and Young's inequalities 
\begin{equation}\label{est 56} 
\RomanIII_{5} \overset{\eqref{est 50} \eqref{est 53}}{\lesssim}   \lVert \nabla b \rVert_{L^{6}} \lVert \nabla b_{h} \rVert_{L^{\frac{6p_{2}}{5p_{2} -6}}} \lVert \nabla^{2} b_{h} \rVert_{L^{p_{2}}} \leq \frac{1}{8} \lVert \Delta b \rVert_{L^{2}}^{2} + C \lVert \nabla b \rVert_{L^{2}}^{2} \lVert \nabla^{2} b_{h} \rVert_{L^{p_{2}}}^{\frac{2p_{2}}{2p_{2} - 3}}.
\end{equation} 
Applying \eqref{est 51}, \eqref{est 52}, \eqref{est 55}, and \eqref{est 56} to \eqref{est 54} gives us 
\begin{align}
& \frac{1}{2} \frac{d}{dt} ( \lVert \nabla u \rVert_{L^{2}}^{2} + \lVert \nabla b \rVert_{L^{2}}^{2}) + \lVert \Delta u \rVert_{L^{2}}^{2} + \lVert \Delta b \rVert_{L^{2}}^{2}  \\
\leq& \frac{1}{2} ( \lVert \Delta u \rVert_{L^{2}}^{2} + \lVert \Delta b \rVert_{L^{2}}^{2}) + C( \lVert u_{h} \rVert_{L^{p_{1}}}^{\frac{2p_{1}}{p_{1} -3}}  + \lVert \nabla^{2} b_{h} \rVert_{L^{p_{2}}}^{\frac{6p_{2}}{7p_{2} - 6}} + \lVert \nabla^{2} b_{h} \rVert_{L^{p_{2}}}^{\frac{2p_{2}}{2p_{2} - 3}}) (\lVert \nabla u \rVert_{L^{2}}^{2} + \lVert \nabla b\rVert_{L^{2}}^{2}). \nonumber 
\end{align}
Due to \eqref{est 49}, $( \lVert u_{h} \rVert_{L^{p_{1}}}^{\frac{2p_{1}}{p_{1} -3}}  + \lVert \nabla^{2} b_{h} \rVert_{L^{p_{2}}}^{\frac{6p_{2}}{7p_{2} - 6}} + \lVert \nabla^{2} b_{h} \rVert_{L^{p_{2}}}^{\frac{2p_{2}}{2p_{2} - 3}}) \in L_{T}^{1}$ and thus Gronwall's inequality completes the proof of Proposition \ref{Proposition 4.1}. 
\end{proof} 
 
\begin{proposition}\label{Proposition 4.2} 
Under the hypothesis of Theorem \ref{Theorem 2.1}, suppose that $(u,b)$ is a smooth solution to the 3-D Hall-MHD system \eqref{est 1} over $[0,T]$. Then 
\begin{equation*}
u,b \in L_{T}^{\infty} H_{x}^{2} \cap L_{T}^{2} H_{x}^{3}. 
\end{equation*} 
\end{proposition}

\begin{proof}[Proof of Proposition \ref{Proposition 4.2}]
We apply $\Delta$ on the Hall-MHD system \eqref{est 1} and take $L^{2}(\mathbb{R}^{3})$-inner products with $(\Delta  u, \Delta b)$ to obtain 
\begin{equation}\label{est 61 star} 
 \frac{1}{2} \frac{d}{dt} ( \lVert \Delta u \rVert_{L^{2}}^{2} + \lVert \Delta b \rVert_{L^{2}}^{2}) + \lVert \Delta \nabla u \rVert_{L^{2}}^{2} + \lVert \Delta \nabla b \rVert_{L^{2}}^{2} = \sum_{i=1}^{2} \RomanIV_{i} 
\end{equation} 
where 
\begin{subequations}\label{est 37} 
\begin{align}
 \RomanIV_{1} \triangleq&  -\int_{\mathbb{R}^{3}} \Delta [ ( u\cdot\nabla) u] \cdot \Delta u + \Delta [(u\cdot\nabla) b] \cdot \Delta b \nonumber\\
 & \hspace{10mm} - \Delta [(b\cdot\nabla) b] \cdot \Delta u - \Delta [(b\cdot\nabla) u] \cdot \Delta b dx, \\
\RomanIV_{2} \triangleq& -\int_{\mathbb{R}^{3}} \Delta \nabla \times (j\times b) \cdot \Delta b dx. 
\end{align}
\end{subequations} 
We can estimate via H$\ddot{\mathrm{o}}$lder's inequality, the Sobolev embedding $\dot{H}^{1}(\mathbb{R}^{3}) \hookrightarrow L^{6}(\mathbb{R}^{3})$, the Kato-Ponce commutator estimate \cite{KP88}, and Young's inequality 
\begin{align}
\RomanIV_{1} \overset{\eqref{est 37}}{=}& \int_{\mathbb{R}^{3}} [\Delta [(u\cdot\nabla) u] - (u\cdot\nabla) \Delta u ] \cdot \Delta u + [\Delta [(u\cdot\nabla) b] - (u\cdot\nabla) \Delta b] \cdot \Delta b \nonumber \\
&\hspace{3mm}  - [\Delta [(b\cdot\nabla) b] - (b\cdot\nabla) \Delta b] \cdot \Delta u - [\Delta [ (b\cdot\nabla) u] - (b\cdot\nabla) \Delta u] \cdot \Delta b dx \nonumber \\
\lesssim& ( \lVert \nabla u \rVert_{L^{3}} + \lVert \nabla b \rVert_{L^{3}}) ( \lVert \Delta u \rVert_{L^{2}} + \lVert \Delta b \rVert_{L^{2}}) ( \lVert \Delta \nabla u \rVert_{L^{2}} + \lVert \Delta \nabla b \rVert_{L^{2}}) \nonumber  \\
\lesssim& ( \lVert \nabla u \rVert_{L^{2}} + \lVert \nabla b \rVert_{L^{2}})^{\frac{1}{2}} ( \lVert \Delta u \rVert_{L^{2}} + \lVert \Delta b \rVert_{L^{2}})^{\frac{3}{2}} ( \lVert \Delta \nabla u \rVert_{L^{2}} + \lVert \Delta \nabla b \rVert_{L^{2}}) \nonumber  \\
\leq& \frac{1}{4} ( \lVert \Delta \nabla u \rVert_{L^{2}}^{2} + \lVert \Delta \nabla b \rVert_{L^{2}}^{2}) + C  ( \lVert \Delta u \rVert_{L^{2}}^{2} + \lVert \Delta b \rVert_{L^{2}}^{2})^{\frac{3}{2}}  \label{est 59} 
\end{align} 
where we used that $u,b \in L_{T}^{\infty} H_{x}^{1}$ from Proposition \ref{Proposition 4.1}. On the other hand, for the Hall term, we rely on Proposition \ref{Proposition 3.1} to deduce via H$\ddot{\mathrm{o}}$lder's, Gagliardo-Nirenberg, and Young's inequalities to estimate 
\begin{align}
\RomanIV_{2}\overset{\eqref{est 37} \eqref{est 58}}{\lesssim}& \int_{\mathbb{R}^{3}} \lvert \nabla^{2} b_{h} \rvert ( \lvert \nabla b \rvert \lvert \nabla^{3} b \rvert + \lvert \nabla^{2} b_{v} \rvert \lvert \nabla^{2} b \rvert) dx  \nonumber\\
\lesssim& \lVert \nabla^{2} b_{h} \rVert_{L^{p_{2}}} ( \lVert \nabla b \rVert_{L^{\frac{2p_{2}}{p_{2}-2}}} \lVert \nabla^{3} b \rVert_{L^{2}} + \lVert \nabla^{2} b \rVert_{L^{\frac{2p_{2}}{p_{2}-1}}}^{2}) \nonumber \\
\lesssim&  \lVert \nabla^{2}b_{h} \rVert_{L^{p_{2}}} \lVert \nabla^{2} b \rVert_{L^{2}}^{\frac{2p_{2}-3}{p_{2}}} \lVert \nabla^{3} b \rVert_{L^{2}}^{\frac{3}{p_{2}}}  \leq \frac{1}{4} \lVert \Delta \nabla b \rVert_{L^{2}}^{2} + C \lVert \nabla^{2} b_{h} \rVert_{L^{p_{2}}}^{\frac{2p_{2}}{2p_{2}-3}} \lVert \Delta b \rVert_{L^{2}}^{2}. \label{est 60} 
\end{align}
Applying \eqref{est 59}-\eqref{est 60} to \eqref{est 61 star} gives us 
\begin{align*}
&\frac{d}{dt} ( \lVert \Delta u \rVert_{L^{2}}^{2} + \lVert \Delta b \rVert_{L^{2}}^{2}) + \lVert \Delta \nabla u \rVert_{L^{2}}^{2} + \lVert \Delta \nabla b \rVert_{L^{2}}^{2} \nonumber\\
\leq&  C( \lVert \Delta u \rVert_{L^{2}} + \lVert \Delta b \rVert_{L^{2}} + \lVert \nabla^{2} b_{h} \rVert_{L^{p_{2}}}^{\frac{2p_{2}}{2p_{2} -3}})( \lVert \Delta u \rVert_{L^{2}}^{2} + \lVert \Delta b \rVert_{L^{2}}^{2}).
\end{align*} 
By Proposition \ref{Proposition 4.1} we know that $\lVert \Delta u \rVert_{L^{2}} + \lVert \Delta b \rVert_{L^{2}} \in L_{T}^{1}$ while $\lVert \nabla^{2} b_{h} \rVert_{L^{p_{2}}}^{\frac{2p_{2}}{2p_{2} -3}} \in L_{T}^{1}$ due to \eqref{est 49b} and thus Gronwall's inequality completes the proof. 
\end{proof} 

\section{Proof of Theorem \ref{Theorem 2.2}}\label{Section 5}
Taking an $L^{2}(\mathbb{R}^{2})$-inner products on \eqref{est 63} with $b$ leads us to 
\begin{equation}\label{est 89}
\sup_{t \in [0,T]} \lVert b(t) \rVert_{L_{x}^{2}} + \int_{0}^{T} \lVert \Lambda^{\frac{3}{2}} b_{h} \rVert_{L^{2}}^{2} + \lVert \Lambda^{\alpha} b_{v} \rVert_{L^{2}}^{2} ds \leq \lVert b^{\text{in}} \rVert_{L^{2}}^{2}. 
\end{equation} 
Now using \eqref{est 38}, as we discussed in \eqref{est 129}-\eqref{est 139}, we know that we can get an $H^{1}(\mathbb{R}^{2})$-bound for $b$ that solves the electron MHD system \eqref{est 63}. Instead, due to \eqref{est 86} from Proposition \ref{Proposition 3.1}, we are able to deduce the $H^{2}(\mathbb{R}^{2})$-bound immediately as follows. 

\begin{proposition}\label{Proposition 5.1} 
Under the hypothesis of Theorem \ref{Theorem 2.2}, suppose that $b$ is a smooth solution to the electron MHD system \eqref{est 63} over $[0,T]$. Then 
\begin{equation}\label{est 93}
b \in L_{T}^{\infty} H_{x}^{2}, \hspace{3mm}  b_{h} \in L_{T}^{2} H_{x}^{\frac{7}{2}}, \hspace{3mm} b_{v} \in L_{T}^{2} H_{x}^{2+ \alpha}. 
\end{equation} 
\end{proposition} 

\begin{proof}[Proof of Proposition \ref{Proposition 5.1} ]
Applying $\Delta$ to \eqref{est 63}, multiplying the resulting equation with $\Delta b$, and integrating over $\mathbb{R}^{2}$ give us 
\begin{align}\label{est 87}
\frac{1}{2} \frac{d}{dt} \lVert \Delta b \rVert_{L^{2}}^{2} + \lVert \Lambda^{\frac{7}{2}} b_{h} \rVert_{L^{2}}^{2} + \lVert \Lambda^{2+ \alpha} b_{v} \rVert_{L^{2}}^{2} 
=& - \int_{\mathbb{R}^{2}} \Delta \nabla \times (j\times b) \cdot \Delta b dx \nonumber\\
 \overset{\eqref{est 86}}{\lesssim}&  \int_{\mathbb{R}^{2}} ( \lvert \nabla b_{h} \rvert \lvert \nabla^{3} b_{h} \rvert + \lvert \nabla^{2} b_{h} \rvert^{2}) \lvert \nabla^{2} b_{v} \rvert dx. 
\end{align} 
We can continue to bound from \eqref{est 87} by H$\ddot{\mathrm{o}}$lder's inequality, the Sobolev embedding $\dot{H}^{\frac{1}{2}} (\mathbb{R}^{2}) \hookrightarrow L^{4} (\mathbb{R}^{2})$, Gagliardo-Nirenberg and Young's inequalities, 
\begin{align}
&\frac{1}{2} \frac{d}{dt} \lVert \Delta b \rVert_{L^{2}}^{2} + \lVert \Lambda^{\frac{7}{2}} b_{h} \rVert_{L^{2}}^{2} + \lVert \Lambda^{2+ \alpha} b_{v} \rVert_{L^{2}}^{2}   \nonumber\\
\lesssim&( \lVert \nabla b_{h} \rVert_{L^{4}} \lVert \nabla^{3} b_{h} \rVert_{L^{4}} + \lVert \nabla^{2} b_{h} \rVert_{L^{4}}^{2}) \lVert \Delta b \rVert_{L^{2}}  \nonumber \\
\lesssim& \lVert \Lambda^{\frac{3}{2}} b_{h} \rVert_{L^{2}} \lVert \Lambda^{\frac{7}{2}} b_{h} \rVert_{L^{2}} \lVert \Delta b \rVert_{L^{2}}  \leq \frac{1}{2} \lVert \Lambda^{\frac{7}{2}} b_{h} \rVert_{L^{2}}^{2} + C \lVert \Lambda^{\frac{3}{2}} b_{h} \rVert_{L^{2}}^{2} \lVert \Delta b \rVert_{L^{2}}^{2}. \label{est 88}
\end{align}
Subtracting $\frac{1}{2} \lVert \Lambda^{\frac{7}{2}} b_{h} \rVert_{L^{2}}^{2}$ from both sides of \eqref{est 88} and applying Gronwall's inequality complete the proof. 
\end{proof} 

With Proposition \ref{Proposition 5.1} in hand, we are ready to prove the $H^{3}(\mathbb{R}^{2})$-bound of the solution $b$ to the electron MHD system \eqref{est 63}. We apply $\partial_{m} \partial_{k} \partial_{l}$ for $m,k,l \in\{1,2\}$ on \eqref{est 63}, take $L^{2}(\mathbb{R}^{2})$-inner products with $\partial_{m}\partial_{k}\partial_{l} b$ and then sum over $m,k,l \in \{1,2\}$  to obtain 
\begin{align*}
\frac{1}{2} \frac{d}{dt} \lVert b \rVert_{\dot{H}^{3}}^{2} + \lVert \Lambda^{\frac{3}{2}} b_{h} \rVert_{\dot{H}^{3}}^{2} + \lVert \Lambda^{\alpha} b_{v} \rVert_{\dot{H}^{3}}^{2} = - \sum_{m,k,l=1}^{2} \int_{\mathbb{R}^{2}} \partial_{m}\partial_{k}\partial_{l}  (j\times b) \cdot \partial_{m}\partial_{k}\partial_{l} j dx
\end{align*} 
where using \eqref{est 3}, we see that 
\begin{equation}\label{est 96} 
\frac{1}{2} \frac{d}{dt} \lVert b \rVert_{\dot{H}^{3}}^{2} + \lVert \Lambda^{\frac{3}{2}} b_{h} \rVert_{\dot{H}^{3}}^{2} + \lVert \Lambda^{\alpha} b_{v} \rVert_{\dot{H}^{3}}^{2}  = \sum_{l=1}^{2} \RomanVII_{l}
\end{equation} 
where 
\begin{subequations}\label{est 90} 
\begin{align}
\RomanVII_{1} \triangleq -\sum_{m,k,l=1}^{2}\int_{\mathbb{R}^{2}} &[ \partial_{k}\partial_{l} j \times \partial_{m} b  \nonumber\\
&+ \partial_{m}\partial_{l} j \times \partial_{k} b + \partial_{m}\partial_{k} j \times \partial_{l} b + j \times \partial_{m}\partial_{k}\partial_{l} b] \cdot \partial_{m}\partial_{k}\partial_{l} j dx, \\
\RomanVII_{2} \triangleq -\sum_{m,k,l=1}^{2} \int_{\mathbb{R}^{2}} &[\partial_{l} j \times \partial_{m}\partial_{k} b + \partial_{k} j \times \partial_{m}\partial_{l} b + \partial_{m} j \times \partial_{k}\partial_{l} b] \cdot \partial_{m}\partial_{k}\partial_{l} j dx. 
\end{align}
\end{subequations} 
Before we start our estimates, we recall a standard inequality 
\begin{equation}\label{est 91} 
\lVert \Lambda^{1-\alpha} (fg) \rVert_{L^{2}}^{2} \lesssim \lVert (\Lambda^{1-\alpha} f) g \rVert_{L^{2}}^{2} + \lVert f (\Lambda^{1-\alpha} g ) \rVert_{L^{2}}^{2} 
\end{equation} 
as $\alpha > 0$, that can be proven by merely applying Plancherel theorem and triangle inequality. Now because $\alpha > \frac{1}{2}$ by hypothesis of Theorem \ref{Theorem 2.2}, we can find $\epsilon > 0$ such that 
\begin{equation}\label{est 140} 
\alpha > \frac{1}{2} + \epsilon 
\end{equation} 
and use H$\ddot{\mathrm{o}}$lder's and Gagliardo-Nirenberg inequalities to deduce
\begin{equation}\label{est 92} 
\RomanVII_{1} \overset{\eqref{est 90}\eqref{est 91}}{\lesssim} [ \lVert \Lambda^{1-\alpha} D^{3} b \rVert_{L^{\frac{1}{1- \alpha + \epsilon}}} \lVert Db \rVert_{L^{\frac{1}{\alpha - \frac{1}{2} - \epsilon}}} + \lVert D^{3} b \rVert_{L^{2}} \lVert \Lambda^{1-\alpha} Db \rVert_{L^{\infty}} ] \lVert b \rVert_{\dot{H}^{3+\alpha}}. 
\end{equation} 
Now $\frac{1}{2} < \alpha$ justifies the Sobolev embedding $H^{2+ \alpha}(\mathbb{R}^{2}) \hookrightarrow W^{2-\alpha,\infty}(\mathbb{R}^{2})$. We can also use the Sobolev embedding $H^{2}(\mathbb{R}^{2}) \hookrightarrow \dot{W}^{1, \frac{1}{\alpha - \frac{1}{2} - \epsilon}} (\mathbb{R}^{2})$, Gagliardo-Nirenberg and Young's inequalities to continue to bound from \eqref{est 92} by 
\begin{align}
\RomanVII_{1} \lesssim& ( \lVert b \rVert_{\dot{H}^{2}}^{\frac{2\epsilon}{1+ \alpha}} \lVert b \rVert_{\dot{H}^{3+ \alpha}}^{1- \frac{2\epsilon}{1+ \alpha}} \lVert b \rVert_{H^{2}} + \lVert b \rVert_{\dot{H}^{3}} \lVert b \rVert_{H^{2+ \alpha}}) \lVert b \rVert_{\dot{H}^{3+ \alpha}} \label{est 94} \\ 
\overset{\eqref{est 93}}{\lesssim}&  (\lVert b \rVert_{\dot{H}^{3+ \alpha}}^{1- \frac{2\epsilon}{1+ \alpha}}  + \lVert b \rVert_{\dot{H}^{3}} \lVert b \rVert_{H^{2+ \alpha}}) \lVert b \rVert_{\dot{H}^{3+ \alpha}} \leq \frac{1}{4} (\lVert \Lambda^{\frac{3}{2}} b_{h} \rVert_{\dot{H}^{3}}^{2} + \lVert \Lambda^{\alpha} b_{v} \rVert_{\dot{H}^{3}}^{2}) + C(1+ \lVert b \rVert_{\dot{H}^{3}}^{2} \lVert b \rVert_{H^{2+ \alpha}}^{2}). \nonumber 
\end{align} 
Next, we estimate by H$\ddot{\mathrm{o}}$lder's inequality, the Sobolev embedding $\dot{H}^{\alpha}(\mathbb{R}^{2}) \hookrightarrow L^{\frac{2}{1-\alpha}}(\mathbb{R}^{2})$, Gagliardo-Nirenberg and Young's inequalities  
\begin{align}
\RomanVII_{2} \overset{\eqref{est 90} \eqref{est 91}}{\lesssim}& \lVert \Lambda^{1-\alpha} D^{2} b \rVert_{L^{\frac{2}{1-\alpha}}} \lVert D^{2} b \rVert_{L^{\frac{2}{\alpha}}}\lVert b \rVert_{\dot{H}^{3+\alpha}}  \label{est 95} \\
\lesssim& \lVert b \rVert_{\dot{H}^{3}} \lVert b \rVert_{\dot{H}^{2}}^{(\alpha - \frac{1}{2}) \frac{2}{\alpha}} \lVert D^{2+ \alpha} b \rVert_{L^{2}}^{1- (\alpha - \frac{1}{2}) \frac{2}{\alpha}} \lVert b \rVert_{\dot{H}^{3+ \alpha}} \nonumber\\
\overset{\eqref{est 93}}{\leq}& \frac{1}{4}  ( \lVert \Lambda^{\frac{3}{2}} b_{h} \rVert_{\dot{H}^{3}}^{2} + \lVert \Lambda^{\alpha} b_{v} \rVert_{\dot{H}^{3}}^{2}) + C(1+ \lVert b \rVert_{\dot{H}^{3}}^{2} \lVert \Lambda^{2+ \alpha} b \rVert_{L^{2}}^{2( 1- (\alpha - \frac{1}{2})\frac{2}{\alpha})}). \nonumber 
\end{align}
Applying \eqref{est 94} and \eqref{est 95} to \eqref{est 96} gives us 
\begin{equation}\label{est 164} 
\sum_{l=1}^{2} \RomanVII_{l} \overset{\eqref{est 94} \eqref{est 95}}{\leq} \frac{1}{2} ( \lVert \Lambda^{\frac{3}{2}} b_{h} \rVert_{\dot{H}^{3}}^{2} + \lVert \Lambda^{\alpha} b_{v} \rVert_{\dot{H}^{3}}^{2}) + C(1+ \lVert b \rVert_{\dot{H}^{3}}^{2}) (1 + \lVert b \rVert_{\dot{H}^{2+ \alpha}}^{2}). 
\end{equation} 
Now Gronwall's inequality completes the proof of Theorem \ref{Theorem 2.2} due to Proposition \ref{Proposition 5.1}. 

\section{Proof of Theorem \ref{Theorem 2.3}}\label{Section 6}   
Taking an $L^{2}(\mathbb{R}^{2})$-inner products on \eqref{est 151} with $(u,b)$ leads us to 
\begin{align}
& \sup_{t \in [0,T]} ( \lVert u(t) \rVert_{L^{2}}^{2} + \lVert b(t) \rVert_{L^{2}}^{2})  \nonumber \\
&+ \int_{0}^{T} \lVert u_{h} \rVert_{\dot{H}^{\alpha}}^{2} + \lVert u_{v} \rVert_{\dot{H}^{1}}^{2} + \lVert b_{h} \rVert_{\dot{H}^{\frac{3}{2}}}^{2} + \lVert b_{v} \rVert_{\dot{H}^{\alpha}}^{2} ds \leq \lVert u^{\text{in}} \rVert_{L^{2}}^{2} + \lVert b^{\text{in}} \rVert_{L^{2}}^{2}. \label{est 119} 
\end{align}

\begin{proposition}\label{Proposition 6.1} 
Under the hypothesis of Theorem \ref{Theorem 2.3}, suppose that $(u,b)$ is a smooth solution to \eqref{est 151} over $[0,T]$. Then 
\begin{equation}\label{est 120}
u_{h} \in L_{T}^{\infty} H_{x}^{1} \cap L_{T}^{2} H_{x}^{1+ \alpha}.
\end{equation} 
\end{proposition} 

\begin{proof}[Proof of Proposition \ref{Proposition 6.1}]
We apply a curl operator on \eqref{est 151a} to deduce the following vorticity formulation:
\begin{equation}\label{est 171} 
\partial_{t} \omega_{3}  + (u\cdot\nabla) \omega_{3} + (-\Delta)^{\alpha} \omega_{3} = [\nabla \times (j\times b)]_{3}
\end{equation} 
 where we used \eqref{est 170a} and \eqref{est 150}. Now we define 
\begin{equation}\label{est 156}
z \triangleq \omega + b \text{ so that } z_{3} = \omega_{3} + b_{3}; 
\end{equation} 
consequently, by adding \eqref{est 171} and \eqref{est 151b}, we see that $z_{3}$ satisfies the equation of 
\begin{equation}\label{est 154}  
\partial_{t} z_{3} + (u\cdot\nabla) z_{3} - (b\cdot\nabla) u_{3} + (-\Delta)^{\alpha} z_{3} = 0.
\end{equation}  
We take $L^{2}(\mathbb{R}^{2})$-inner products on \eqref{est 154} with $z_{3}$ to compute 
\begin{align}
\frac{1}{2} \frac{d}{dt} \lVert z_{3} \rVert_{L^{2}}^{2} + \lVert z_{3} \rVert_{\dot{H}^{\alpha}}^{2} =& \int_{\mathbb{R}^{2}} (b_{h} \cdot \nabla_{h}) u_{3} z_{3} dx \nonumber \\
\leq& \lVert b_{h} \rVert_{L^{\infty}} \lVert \nabla u_{3} \rVert_{L^{2}} \lVert z_{3} \rVert_{L^{2}} \lesssim ( \lVert b_{h} \rVert_{H^{\frac{3}{2}}}^{2} + \lVert \nabla u_{3} \rVert_{L^{2}}^{2}) \lVert z_{3} \rVert_{L^{2}} \label{est 172}
\end{align}
by H$\ddot{\mathrm{o}}$lder's inequality, the Sobolev embedding of $H^{\frac{3}{2}}(\mathbb{R}^{2}) \hookrightarrow L^{\infty} (\mathbb{R}^{2})$, and Young's inequality. Due to \eqref{est 119}, $b_{h} \in L_{T}^{2} H_{x}^{\frac{3}{2}}$ and $u_{3} \in L_{T}^{2} \dot{H}_{x}^{1}$; thus, Gronwall's inequality applied on \eqref{est 172} implies 
\begin{equation}\label{est 173}
z_{3} \in L_{T}^{\infty} L_{x}^{2} \cap L_{T}^{2} \dot{H}_{x}^{\alpha}. 
\end{equation} 
Again, from \eqref{est 119} we know that $b_{3} \in L_{T}^{\infty} L_{x}^{2} \cap L_{T}^{2} \dot{H}_{x}^{\alpha}$. This, together with \eqref{est 173} and \eqref{est 156} implies that $\omega_{3} \in L_{T}^{\infty} L_{x}^{2} \cap L_{T}^{2} \dot{H}_{x}^{\alpha}$. As $\omega_{v} = \nabla \times u_{h}$, this implies $u_{h} \in L_{T}^{\infty} H_{x}^{1} \cap L_{T}^{2} H_{x}^{1+ \alpha}$ as desired. 
\end{proof} 

\begin{proposition}\label{Proposition 6.2} 
Under the hypothesis of Theorem \ref{Theorem 2.3}, suppose that $(u,b)$ is a smooth solution to \eqref{est 151} over $[0,T]$. Then 
\begin{subequations}\label{est 121} 
\begin{align}
&u \in L_{T}^{\infty} H_{x}^{1}, \hspace{3mm} u_{v} \in L_{T}^{2} H_{x}^{2}, \\
&b \in L_{T}^{\infty} H_{x}^{1},  \hspace{3mm} b_{h} \in L_{T}^{2} H_{x}^{\frac{5}{2}},  \hspace{3mm} b_{v} \in L_{T}^{2} H_{x}^{1+ \alpha}.   
\end{align} 
\end{subequations} 
\end{proposition} 

\begin{proof}[Proof of Proposition \ref{Proposition 6.2}]
We take $L^{2}(\mathbb{R}^{2})$-inner products of \eqref{est 151} with $(-\Delta u, -\Delta b)$ and compute 
\begin{equation}\label{est 143} 
\frac{1}{2} \frac{d}{dt} ( \lVert \nabla u \rVert_{L^{2}}^{2} + \lVert \nabla b \rVert_{L^{2}}^{2}) + \lVert \Lambda^{\alpha} \nabla u_{h} \rVert_{L^{2}}^{2} + \lVert \Delta u_{v} \rVert_{L^{2}}^{2} + \lVert \Lambda^{\frac{3}{2}} \nabla b_{h} \rVert_{L^{2}}^{2} + \lVert \Lambda^{\alpha} \nabla b_{v} \rVert_{L^{2}}^{2} = \RomanVIII_{1} + \RomanVIII_{2}  
\end{equation} 
where 
\begin{subequations}\label{est 175} 
\begin{align}
 \RomanVIII_{1} \triangleq &- \sum_{i,k=1}^{2}\sum_{j=1}^{3} \int_{\mathbb{R}^{2}} \partial_{k} u_{i}\partial_{i} u_{j} \partial_{k} u_{j} + \partial_{k} u_{i} \partial_{i} b_{j}\partial_{k} b_{j} \nonumber\\
& \hspace{10mm} - \int_{\mathbb{R}^{2}} \partial_{k} b_{i}\partial_{i} b_{j} \partial_{k} u_{j} + \partial_{k} b_{i}\partial_{i} u_{j} \partial_{k} b_{j} dx, \label{est 175a}\\
 \RomanVIII_{2} \triangleq& \int_{\mathbb{R}^{2}} \nabla \times (j\times b) \cdot \Delta b dx. \label{est 175b}
\end{align}
\end{subequations} 
First, we estimate from \eqref{est 175a}
\begin{align}
\RomanVIII_{1} \lesssim& \lVert \nabla u_{h} \rVert_{L^{\frac{2}{1-\alpha}}} \lVert \nabla u \rVert_{L^{2}} \lVert \nabla u \rVert_{L^{\frac{2}{\alpha}}} + \lVert \nabla u_{h} \rVert_{L^{\frac{2}{1-\alpha}}} \lVert \nabla b \rVert_{L^{2}} \lVert \nabla b \rVert_{L^{\frac{2}{\alpha}}} +\lVert \nabla b_{h} \rVert_{L^{\infty}} \lVert \nabla u \rVert_{L^{2}} \lVert \nabla b \rVert_{L^{2}} \nonumber  \\
\lesssim& \lVert u_{h} \rVert_{\dot{H}^{1+\alpha}} \lVert \nabla u \rVert_{L^{2}} \lVert u \rVert_{\dot{H}^{2-\alpha}} + \lVert u_{h} \rVert_{\dot{H}^{1+ \alpha}} \lVert \nabla b \rVert_{L^{2}} \lVert b\rVert_{\dot{H}^{2-\alpha}} + \lVert b_{h} \rVert_{H^{\frac{5}{2}}} \lVert \nabla u \rVert_{L^{2}} \lVert \nabla b\rVert_{L^{2}} \nonumber \\
\lesssim& \lVert u_{h} \rVert_{\dot{H}^{1+\alpha}} \lVert \nabla u \rVert_{L^{2}} (1+ \lVert u_{h} \rVert_{\dot{H}^{1+ \alpha}} + \lVert u_{v} \rVert_{\dot{H}^{2}}) + \lVert u_{h} \rVert_{\dot{H}^{1+ \alpha}} \lVert \nabla b \rVert_{L^{2}} (1+ \lVert b_{h} \rVert_{\dot{H}^{\frac{5}{2}}} + \lVert b_{v} \rVert_{\dot{H}^{1+ \alpha}}) \nonumber \\
&+ \lVert b_{h} \rVert_{H^{\frac{5}{2}}} (\lVert u_{h} \rVert_{H_{x}^{1+ \alpha}} + \lVert u_{v} \rVert_{\dot{H}_{x}^{1}}) \lVert \nabla b\rVert_{L^{2}} \nonumber\\
\leq& \frac{1}{4} ( \lVert u_{h} \rVert_{\dot{H}^{1+ \alpha}}^{2} + \lVert u_{v} \rVert_{\dot{H}^{2}}^{2} + \lVert b_{h} \rVert_{\dot{H}^{\frac{5}{2}}}^{2} + \lVert b_{v} \rVert_{\dot{H}^{1+ \alpha}}^{2}) \nonumber\\
& + C( 1+ \lVert \nabla u \rVert_{L^{2}}^{2} + \lVert \nabla b \rVert_{L^{2}}^{2}) (1+ \lVert u_{h} \rVert_{H^{1+ \alpha}}^{2} + \lVert u_{v} \rVert_{\dot{H}^{1}}^{2}) \label{est 142} 
\end{align}
by H$\ddot{\mathrm{o}}$lder's inequality, the Sobolev embeddings $\dot{H}^{\alpha} (\mathbb{R}^{2}) \hookrightarrow L^{\frac{2}{1-\alpha}}(\mathbb{R}^{2})$, $\dot{H}^{1-\alpha}(\mathbb{R}^{2}) \hookrightarrow L^{\frac{2}{\alpha}}(\mathbb{R}^{2})$, and $H^{\frac{3}{2}}(\mathbb{R}^{2}) \hookrightarrow L^{\infty} (\mathbb{R}^{2})$, the hypothesis that $\alpha > \frac{1}{2}$ so that $2- \alpha \leq 1+ \alpha$, and Young's inequality. Next, $\RomanVIII_{2}$ can be estimated identically to \eqref{est 139}: 
\begin{equation}\label{est 159} 
\RomanVIII_{2} \leq \frac{1}{4} \lVert \Lambda^{\frac{5}{2}} b_{h} \rVert_{L^{2}}^{2} + C \lVert \Lambda^{\frac{3}{2}} b_{h} \rVert_{L^{2}}^{2} \lVert \nabla b \rVert_{L^{2}}^{2}. 
\end{equation} 
Considering \eqref{est 142}-\eqref{est 159} in \eqref{est 143}, applying Gronwall's inequality to the resulting inequality, and then making use of $b_{h} \in L_{T}^{2} \dot{H}_{x}^{\frac{3}{2}}$ and $u_{v} \in L_{T}^{2} \dot{H}_{x}^{1}$ from \eqref{est 119} and $u_{h} \in L_{T}^{2} H_{x}^{1+ \alpha}$ from \eqref{est 120} verify \eqref{est 121}. 
\end{proof}  

\begin{proposition}\label{Proposition 6.3} 
Under the hypothesis of Theorem \ref{Theorem 2.3}, suppose that $(u,b)$ is a smooth solution to \eqref{est 151} over $[0,T]$. Then
\begin{subequations}\label{est 160} 
\begin{align}
& u \in L_{T}^{\infty} \dot{H}_{x}^{2}, \hspace{1mm}  u_{h} \in L_{T}^{2} H_{x}^{2+ \alpha}, \hspace{1mm} u_{v} \in L_{T}^{2} \dot{H}_{x}^{3},\\
& b \in L_{T}^{\infty} \dot{H}_{x}^{2}, \hspace{1mm}  b_{h} \in L_{T}^{2} H_{x}^{\frac{7}{2}}, \hspace{3mm} b_{v} \in L_{T}^{2} \dot{H}_{x}^{2+\alpha}.  
\end{align}
\end{subequations}
\end{proposition} 

\begin{proof}[Proof of Proposition \ref{Proposition 6.3}]
We apply $\Delta$ to \eqref{est 151} and take $L^{2}(\mathbb{R}^{2})$-inner products with $(\Delta u, \Delta b)$ to obtain 
\begin{equation}\label{est 161} 
\frac{1}{2} \frac{d}{dt} ( \lVert \Delta u \rVert_{L^{2}}^{2} + \lVert \Delta b \rVert_{L^{2}}^{2}) + \lVert \Lambda^{\alpha} \Delta u_{h} \rVert_{L^{2}}^{2} + \lVert \nabla \Delta u_{v} \rVert_{L^{2}}^{2} + \lVert \Lambda^{\frac{3}{2}} \Delta b_{h} \rVert_{L^{2}}^{2} + \lVert \Lambda^{\alpha} \Delta b_{v} \rVert_{L^{2}}^{2} = \RomanIX_{1} + \RomanIX_{2}  
\end{equation} 
where 
\begin{subequations}\label{est 165} 
\begin{align}
 \RomanIX_{1} \triangleq& -\int_{\mathbb{R}^{2}}  \Delta (( u\cdot\nabla) u) \cdot \Delta u + \Delta (( u \cdot \nabla) b)  \cdot \Delta b dx  \nonumber \\
&+ \int_{\mathbb{R}^{2}}  \Delta (( b\cdot\nabla) b) \cdot \Delta u + \Delta (( b \cdot\nabla) u)  \cdot \Delta b dx, \label{est 165a} \\
 \RomanIX_{2} \triangleq& -\int_{\mathbb{R}^{2}}  \Delta \nabla \times ( j \times b) \cdot \Delta b dx.\label{est 165b} 
\end{align}
\end{subequations} 
Using divergence-free property, we first bound $\RomanIX_{1}$ by 
\begin{align*}
\RomanIX_{1} \lesssim& \int_{\mathbb{R}^{2}} \lvert \nabla^{2} u_{h} \rvert \lvert \nabla u \rvert \lvert \nabla^{2} u \rvert + \lvert \nabla u_{h} \rvert \lvert \nabla^{2} u \rvert \lvert \nabla^{2} u \rvert + \lvert \nabla^{2} u_{h} \rvert \lvert \nabla b \rvert \lvert \nabla^{2} b \rvert + \lvert \nabla u_{h} \rvert \lvert \nabla^{2} b \rvert \lvert \nabla^{2} b \rvert \\
&+ \int_{\mathbb{R}^{2}} \lvert \nabla^{2} b_{h} \rvert \lvert \nabla b \rvert \lvert \nabla^{2} u \rvert + \lvert \nabla b_{h} \rvert \lvert \nabla^{2} b \rvert \lvert \nabla^{2} u \rvert + \lvert \nabla^{2} b_{h} \rvert \lvert \nabla u \rvert \lvert \nabla^{2} b \rvert + \lvert \nabla b_{h} \rvert \lvert \nabla^{2} u \rvert \lvert \nabla^{2} b \rvert dx 
\end{align*}
so that we can bound the first integral by 
\begin{align}
&  \int_{\mathbb{R}^{2}} \lvert \nabla^{2} u_{h} \rvert \lvert \nabla u \rvert \lvert \nabla^{2} u \rvert + \lvert \nabla u_{h} \rvert \lvert \nabla^{2} u \rvert \lvert \nabla^{2} u \rvert + \lvert \nabla^{2} u_{h} \rvert \lvert \nabla b \rvert \lvert \nabla^{2} b \rvert + \lvert \nabla u_{h} \rvert \lvert \nabla^{2} b \rvert \lvert \nabla^{2} b \rvert  \nonumber\\
\lesssim& \lVert \nabla^{2} u_{h} \rVert_{L^{\frac{2}{1-\alpha}}} ( \lVert \nabla u \rVert_{L^{\frac{2}{\alpha}}} \lVert \nabla^{2} u \rVert_{L^{2}} + \lVert \nabla b \rVert_{L^{\frac{2}{\alpha}}} \lVert \nabla^{2} b \rVert_{L^{2}}) \nonumber\\
&+ \lVert \nabla u_{h} \rVert_{L^{\frac{2}{1-\alpha}}} ( \lVert \nabla^{2} u \rVert_{L^{\frac{2}{\alpha}}} \lVert \nabla^{2} u \rVert_{L^{2}} + \lVert \nabla^{2} b \rVert_{L^{\frac{2}{\alpha}}} \lVert \nabla^{2} b \rVert_{L^{2}}) \nonumber\\
\leq& \frac{1}{8} [\lVert u_{h} \rVert_{\dot{H}^{2+ \alpha}}^{2}  + \lVert u_{v} \rVert_{\dot{H}^{3}}^{2} + \lVert b_{h} \rVert_{\dot{H}^{\frac{7}{2}}}^{2} + \lVert b_{v} \rVert_{\dot{H}^{2+ \alpha}}^{2} ]  \nonumber \\
&+ C(\lVert u_{h} \rVert_{H^{1+\alpha}}^{2} + \lVert u_{v} \rVert_{H^{2}}^{2} + \lVert b_{h} \rVert_{H^{\frac{5}{2}}}^{2} + \lVert b_{v} \rVert_{H^{1+ \alpha}}^{2}) ( 1+ \lVert \Delta u \rVert_{L^{2}}^{2} + \lVert \Delta b \rVert_{L^{2}}^{2}) \label{est 162} 
\end{align}
due to H$\ddot{\mathrm{o}}$lder's inequality, the Sobolev embeddings $\dot{H}^{\alpha} (\mathbb{R}^{2}) \hookrightarrow L^{\frac{2}{1-\alpha}}(\mathbb{R}^{2})$, $\dot{H}^{1-\alpha}(\mathbb{R}^{2}) \hookrightarrow L^{\frac{2}{\alpha}}(\mathbb{R}^{2})$, and Young's inequality, while the second integral by 
\begin{align}
& \int_{\mathbb{R}^{2}} \lvert \nabla^{2} b_{h} \rvert \lvert \nabla b \rvert \lvert \nabla^{2} u \rvert + \lvert \nabla b_{h} \rvert \lvert \nabla^{2} b \rvert \lvert \nabla^{2} u \rvert + \lvert \nabla^{2} b_{h} \rvert \lvert \nabla u \rvert \lvert \nabla^{2} b \rvert + \lvert \nabla b_{h} \rvert \lvert \nabla^{2} u \rvert \lvert \nabla^{2} b \rvert dx  \nonumber \\
\lesssim& \lVert \nabla^{2} b_{h} \rVert_{L^{4}} ( \lVert \nabla b \rVert_{L^{4}} \lVert \nabla^{2} u \rVert_{L^{2}} + \lVert \nabla u \rVert_{L^{4}} \lVert \nabla^{2} b \rVert_{L^{2}}) + \lVert \nabla b_{h} \rVert_{L^{\infty}} \lVert \nabla^{2} b \rVert_{L^{2}} \lVert \nabla^{2} u \rVert_{L^{2}} \nonumber\\
\lesssim& (1+ \lVert b_{h} \rVert_{H^{\frac{5}{2}}}^{2}) (1+ \lVert \Delta u \rVert_{L^{2}}^{2} + \lVert \Delta b \rVert_{L^{2}}^{2}) \label{est 174} 
\end{align} 
due to H$\ddot{\mathrm{o}}$lder's inequality, the Sobolev embeddings of $\dot{H}^{\frac{1}{2}}(\mathbb{R}^{2}) \hookrightarrow L^{4}(\mathbb{R}^{2})$ and $H^{\frac{3}{2}}(\mathbb{R}^{2})\hookrightarrow L^{\infty} (\mathbb{R}^{2})$, and Young's inequality. We can estimate $\RomanIX_{2}$ identically to \eqref{est 87}-\eqref{est 88} to deduce 
\begin{equation}\label{est 163} 
\RomanIX_{2} \leq \frac{1}{4} \lVert \Lambda^{\frac{7}{2}} b_{h} \rVert_{L^{2}}^{2} + C \lVert \Lambda^{\frac{3}{2}} b_{h} \rVert_{L^{2}}^{2} \lVert \Delta b \rVert_{L^{2}}^{2}. 
\end{equation} 
Considering \eqref{est 162}, \eqref{est 174}, and \eqref{est 163} to \eqref{est 161}, applying Gronwall's inequality to the resulting inequality, and then making use of $u_{h} \in L_{T}^{2} H_{x}^{1+ \alpha}$ from \eqref{est 120} and $u_{v} \in L_{T}^{2} H_{x}^{2}, b_{h} \in L_{T}^{2} H_{x}^{\frac{5}{2}}$, and $b_{v} \in L_{T}^{2} H_{x}^{1+ \alpha}$ from \eqref{est 121} verify \eqref{est 160}. 
\end{proof} 

We are now ready to deduce the $H^{3}(\mathbb{R}^{2})$-bound and conclude the proof of Theorem \ref{Theorem 2.3}. We apply $\partial_{m}\partial_{k}\partial_{l}$ for $m,k,l \in \{1,2\}$ on \eqref{est 151}, take $L^{2}(\mathbb{R}^{2})$-inner products with $(\partial_{m}\partial_{k}\partial_{l} u, \partial_{m}\partial_{k}\partial_{l} b)$, and then sum over $m,k,l \in \{1,2\}$ to obtain similarly to \eqref{est 96} and \eqref{est 90}, 
\begin{align}
\frac{1}{2} \frac{d}{dt} ( \lVert u \rVert_{\dot{H}^{3}}^{2} + \lVert b \rVert_{\dot{H}^{3}}^{2}) + \lVert u_{h} \rVert_{\dot{H}^{3+ \alpha}}^{2} + \lVert u_{v} \rVert_{\dot{H}^{4}}^{2} + \lVert b_{h} \rVert_{\dot{H}^{\frac{9}{2}}}^{2} + \lVert b_{v} \rVert_{\dot{H}^{3+ \alpha}}^{2} = \RomanXI  + \sum_{l=1}^{2} \RomanVII_{l}\label{est 167} 
\end{align}
where 
\begin{align*}
\RomanXI \triangleq& -\sum_{m,k,l=1}^{2} \int_{\mathbb{R}^{2}} \partial_{m} \partial_{k} \partial_{l} (( u\cdot\nabla) u) \cdot \partial_{m}\partial_{k}\partial_{l} u + \partial_{m} \partial_{k} \partial_{l} (( u\cdot\nabla) b) \cdot \partial_{m} \partial_{k} \partial_{l} b dx \nonumber \\
& - \int_{\mathbb{R}^{2}} \partial_{m} \partial_{k} \partial_{l} (( b\cdot\nabla) b) \cdot \partial_{m} \partial_{k} \partial_{l} u + \partial_{m} \partial_{k} \partial_{l} (( b\cdot\nabla) u) \cdot \partial_{m} \partial_{k} \partial_{l} b dx
\end{align*}
and $\{\RomanVII_{l}\}_{l=1}^{2}$ was given in \eqref{est 90}. Similarly to the estimate \eqref{est 165a} and \eqref{est 162} but in a more straight-forward manner, we can estimate 
\begin{equation}\label{est 166}
\RomanXI \lesssim ( \lVert \nabla u \rVert_{L^{\infty}} + \lVert \nabla b \rVert_{L^{\infty}}) ( \lVert u \rVert_{\dot{H}^{3}}^{2} + \lVert b \rVert_{\dot{H}^{3}}^{2}) \lesssim ( \lVert u \rVert_{H^{2+ \alpha}} + \lVert b \rVert_{H^{2+ \alpha}}) ( \lVert u \rVert_{\dot{H}^{3}}^{2} +\lVert  b \rVert_{\dot{H}^{3}}^{2}) 
\end{equation} 
by the Sobolev embedding $H^{2+ \alpha} (\mathbb{R}^{2}) \hookrightarrow W^{1,\infty} (\mathbb{R}^{2})$. Considering \eqref{est 166} and the estimate \eqref{est 164} for $\sum_{l=1}^{2} \RomanVII_{l}$ in \eqref{est 167}, applying Gronwall's inequality to the resulting inequality, and making use of $u \in L_{T}^{2} H_{x}^{2+ \alpha}$ and $b \in L_{T}^{2} H_{x}^{2+ \alpha}$ from \eqref{est 160} complete the proof of Theorem \ref{Theorem 2.3}. 

\section{Appendix} 
\subsection{Proof of the local well-posedness of \eqref{est 63} in $H^{m} (\mathbb{R}^{2})$ for $m \in\mathbb{N}$ such that $m > 2$}
We recall the Littlewood-Paley decomposition (e.g., from \cite[Section 2.2]{BCD11}). Let $\mathcal{S}$ be the Schwartz space, $\chi, \phi$ be smooth functions such that 
\begin{align*}
\text{supp } \phi \subset \left\{\xi \in \mathbb{R}^{2} : \frac{3}{4} \leq \lvert \xi\rvert \leq \frac{8}{3} \right\}, \hspace{2mm}\text{supp } \chi \subset \left\{\xi \in \mathbb{R}^{2}: \lvert \xi\rvert \leq \frac{4}{3}\right\}, \hspace{2mm} \chi(\xi) + \sum_{j\geq 0} \phi(2^{-j} \xi) = 1,
\end{align*}
and denote the Littlewood-Paley operators by 
\begin{align*}
\Delta_{j} f \triangleq \mathcal{F}^{-1} (\phi(2^{-j} \lvert \xi\rvert)\hat{f}), \hspace{3mm} \Delta_{-1} f \triangleq \mathcal{F}^{-1} ( \chi \hat{f}),  \hspace{3mm} S_{j}f \triangleq  \sum_{j' \leq j-1} \Delta_{j'} u.
\end{align*}
We define $\mathcal{S}_{r}'$ to be the subspace of $\mathcal{S}'$ such that $\lim_{j\to -\infty} \lVert S_{j} f\rVert_{L^{\infty}} = 0$ for all $f \in \mathcal{S}_{r}'$.  
\begin{define}
For $p, q \in [1,\infty], s \in \mathbb{R}$, we define the Besov spaces $B_{p,q}^{s}(\mathbb{R}^{2}) \triangleq \{f \in \mathcal{S}_{r}' : \lVert f\rVert_{B_{p,q}^{s}} < \infty\}$ where 
\begin{equation*}
\lVert f\rVert_{B_{p,q}^{s}} \triangleq \left\lVert 2^{js}\lVert \Delta_{j} f\rVert_{L^{p}} \right\rVert_{l^{q}(j \geq -1)}.
\end{equation*}
It is well-known that $B_{2,2}^{s}(\mathbb{R}^{2}) = H^{s}(\mathbb{R}^{2})$. 
\end{define}
For any $l \geq -1$ we apply $\Delta_{l}$ to \eqref{est 63} to obtain 
\begin{equation*}
\partial_{t} \Delta_{l} b + \epsilon \Delta_{l} \nabla \times (j\times b) + \Delta_{l} (-\Delta)^{\frac{3}{2}} b_{h} + \Delta_{l} (-\Delta)^{\alpha} b_{v} = 0. 
\end{equation*} 
For $l \geq 5$, we use Bernstein's inequality to compute for universal constants $C_{0,1}, C_{0,2} \geq 0$,  
\begin{equation}\label{est 145} 
\frac{1}{2} \frac{d}{dt}  \lVert \Delta_{l} b \rVert_{L^{2}}^{2} + C_{0,1} 2^{3l} \lVert \Delta_{l} b_{h} \rVert_{L^{2}}^{2} + C_{0,2} 2^{2\alpha l} \lVert \Delta_{l} b_{v} \rVert_{L^{2}}^{2} \leq - \int_{\mathbb{R}^{2}} \Delta_{l} \nabla \times (j\times b) \cdot \Delta_{l} b dx. 
\end{equation} 
It is shown on \cite[p. 631]{CWW15a} that we can estimate for any $l \geq -1$, 
\begin{equation}\label{est 97}
- \int_{\mathbb{R}^{2}} \Delta_{l} \nabla \times (j\times b) \cdot \Delta_{l} b dx \lesssim 2^{l} \lVert \nabla b \rVert_{L^{\infty}} \lVert \Delta_{l} b \rVert_{L^{2}} \left( \lVert \Delta_{l} b \rVert_{L^{2}} + \sum_{k\geq l-1} \lVert \Delta_{k} b \rVert_{L^{2}} \right).
\end{equation} 
For $l = -1, \hdots, 4$, we have 
\begin{align}
\frac{1}{2} \frac{d}{dt} \lVert \Delta_{l} b \rVert_{L^{2}}^{2} + \lVert (-\Delta)^{\frac{3}{4}} b_{v} \Delta_{l} b_{h} \rVert_{L^{2}}^{2} + \lVert  (-\Delta)^{\frac{\alpha}{2}} \Delta_{l} b_{v} \rVert_{L^{2}}^{2} \leq - \int_{\mathbb{R}^{2}} \Delta_{l} \nabla \times (j\times b) \cdot \Delta_{l} b dx  \label{est 144} 
\end{align}
so that together with \eqref{est 97} we deduce for all $l \geq -1$, 
\begin{align}
& \frac{1}{2} \frac{d}{dt} \lVert \Delta_{l} b \rVert_{L^{2}}^{2} + C_{0,1} 2^{3l} \lVert \Delta_{l} b_{h} \rVert_{L^{2}}^{2} + C_{0,2} 2^{2\alpha l} \lVert \Delta_{l} b_{v} \rVert_{L^{2}}^{2}  \nonumber \\
\lesssim& (1+ 2^{l} \lVert \nabla b \rVert_{L^{\infty}} ) \lVert \Delta_{l} b \rVert_{L^{2}}^{2} + 2^{l} \lVert \nabla b \rVert_{L^{\infty}} \lVert \Delta_{l} b \rVert_{L^{2}} \sum_{k\geq l-1} \lVert \Delta_{k} b \rVert_{L^{2}}.\label{est 146} 
\end{align}
We multiply \eqref{est 146} by $2^{2ml}$, sum over $l \geq -1$, integrate over $[0,t]$ and follow identical estimates in \cite[p. 631]{CWW15a} to obtain 
\begin{align}
&\lVert b(t) \rVert_{H^{m}}^{2} + \int_{0}^{t} C_{0,1} \lVert  b_{h} \rVert_{H^{m+ \frac{3}{2}}}^{2} + C_{0,2} \lVert b_{v}  \rVert_{H^{m+\alpha}}^{2} ds \\
\leq& \frac{ \min \{ C_{0,1}, C_{0,2} \}}{2} \int_{0}^{t} \lVert b \rVert_{H^{m+\alpha}}^{2} ds + C ( \lVert b^{\text{in}} \rVert_{H^{m}}^{2} + \int_{0}^{t} \lVert b \rVert_{H^{m}}^{2} +  \lVert \nabla b \rVert_{L^{\infty}}^{\frac{2\alpha}{2\alpha -1}} \lVert b \rVert_{H^{m}}^{2} ds ). \nonumber
\end{align}
We can write 
\begin{equation*}
\frac{ \min\{C_{0,1}, C_{0,2} \}}{2} \lVert b \rVert_{H^{m+\alpha}}^{2}  \leq \frac{C_{0,1}}{2} \lVert b_{h} \rVert_{H^{m+ \frac{3}{2}}}^{2} + \frac{C_{0,2}}{2} \lVert b_{v} \rVert_{H^{m+\alpha}}^{2} 
\end{equation*} 
and apply $H^{m} (\mathbb{R}^{2}) \hookrightarrow W^{1,\infty} (\mathbb{R}^{2})$ because $m > 2$ to deduce 
\begin{align*}
\lVert b(t) \rVert_{H^{m}}^{2} + \int_{0}^{t} \frac{C_{0,1}}{2} \lVert b_{h} \rVert_{H^{m+ \frac{3}{2}}}^{2} + \frac{C_{0,2}}{2} \lVert b_{v} \rVert_{H^{m+ \alpha}}^{2} ds \lesssim \lVert b^{\text{in}} \rVert_{H^{m}}^{2} + \int_{0}^{t} (1+ \lVert b \rVert_{H^{m}}^{\frac{2\alpha}{2\alpha -1} + 2}) ds, 
\end{align*}
from which local well-posedness follows identically to \cite{CWW15a}.

\subsection{Proof of \eqref{est 100}}
We prove the inequality \eqref{est 100}. Again, we restart from the identity \eqref{est 61}:
\begin{align}
\int_{\mathbb{R}^{3}} \Delta \nabla \times (j\times b) \cdot \Delta b dx =& \sum_{l\in \{2,3,4,6,7,8\}} \RomanI_{1,3,l} + \sum_{l\in \{1,3,4,5,6,8 \}} \RomanI_{2,5,l} + \sum_{l \in \{1,2,3, 5,6,7 \}} \RomanI_{4,6,l} \nonumber \\
&+ \sum_{l\in \{2,3,4,6,7,8\}} \RomanII_{1,3,l} +  \sum_{l\in \{1,3,4,5,6,8 \}} \RomanII_{2,5,l} + \sum_{l \in \{1,2,3, 5,6,7 \}} \RomanII_{4,6,l}  \label{est 133}  
\end{align} 
The following computations will consist of many more terms and must be done more carefully. First, 
\begin{align}
\sum_{l \in \{ 4,8\}} \RomanI_{1,3,l} \overset{\eqref{est 6} }{=}& -2 \sum_{k,l=1}^{3} \int_{\mathbb{R}^{3}} \partial_{k} b_{3} \partial_{k} \partial_{3} b_{1} \partial_{l}^{2} \partial_{3} b_{2} - \partial_{k} b_{3} \partial_{k} \partial_{3} b_{2} \partial_{l}^{2} \partial_{3} b_{1} dx  \nonumber \\
\lesssim& \int_{\mathbb{R}^{3}} \lvert \nabla b_{v} \rvert \lvert \nabla^{2}b_{h} \rvert \lvert \nabla^{3} b_{h} \rvert dx. \label{est 118} 
\end{align}
Second, we work on $\RomanI_{1,3,l}$ for $l \in \{3,7\}$ which is non-trivial as it consists of $b_{3}$ twice; this difficult term does not appear in the $2\frac{1}{2}$-D case as they vanish due to $\partial_{3}$ therein. For these terms we first integrate by parts to obtain 
\begin{align}
\sum_{l \in \{3,7 \}} \RomanI_{1,3,l} \overset{\eqref{est 6}}{=}& 2 \sum_{k,l=1}^{3} \int_{\mathbb{R}^{3}} \partial_{k} b_{3} \partial_{k} \partial_{3} b_{1} \partial_{l}^{2} \partial_{2} b_{3} - \partial_{k} b_{3}\partial_{k}\partial_{3} b_{2} \partial_{l}^{2} \partial_{1} b_{3} dx  \nonumber \\
=& - 2 \sum_{k,l=1}^{3} \int_{\mathbb{R}^{3}} \partial_{k} b_{1} \partial_{3} (\partial_{k} b_{3} \partial_{l}^{2} \partial_{2} b_{3}) - \partial_{k} b_{2} \partial_{3} (\partial_{k} b_{3} \partial_{l}^{2} \partial_{1} b_{3}) dx \nonumber \\
=& -2 \sum_{k,l=1}^{3} \int_{\mathbb{R}^{3}} \partial_{k} b_{1} ( \partial_{k} \partial_{3} b_{3} \partial_{l}^{2} \partial_{2} b_{3} + \partial_{k} b_{3} \partial_{l}^{2} \partial_{2} \partial_{3} b_{3}) \nonumber \\
& \hspace{15mm} - \partial_{k} b_{2} (\partial_{k} \partial_{3} b_{3} \partial_{l}^{2} \partial_{1} b_{3} + \partial_{k} b_{3} \partial_{l}^{2} \partial_{1} \partial_{3} b_{3}) dx. \label{est 101}
\end{align}
We now use divergence-free condition which allows us to write $\partial_{3} = -\partial_{1} b_{1} - \partial_{2} b_{2}$ to continue to compute 
\begin{align}
\sum_{l \in \{3,7 \}} \RomanI_{1,3,l}  \overset{\eqref{est 101}}{=}& - 2 \sum_{k,l=1}^{3} \int_{\mathbb{R}^{3}} \partial_{k} b_{1} ( \partial_{k} [ -\partial_{1} b_{1} - \partial_{2} b_{2} ] \partial_{l}^{2} \partial_{2} b_{3} + \partial_{k} b_{3} \partial_{l}^{2} \partial_{2} [-\partial_{1} b_{1} - \partial_{2} b_{2} ]) \nonumber \\
& \hspace{5mm} - \partial_{k} b_{2} ( \partial_{k} [-\partial_{1} b_{1} - \partial_{2} b_{2} ] \partial_{l}^{2} \partial_{1} b_{3} + \partial_{k} b_{3} \partial_{l}^{2} \partial_{1} [ -\partial_{1} b_{1} - \partial_{2} b_{2} ]) dx  \label{est 102} 
\end{align}
and integrate by parts once more on all four terms to conclude 
\begin{align}
\sum_{l \in \{3,7 \}} \RomanI_{1,3,l}  \overset{\eqref{est 102}}{=}&2 \sum_{k,l=1}^{3} \int_{\mathbb{R}^{3}} \partial_{2} [ \partial_{k} b_{1} \partial_{k} [-\partial_{1} b_{1} - \partial_{2} b_{2} ]] \partial_{l}^{2} b_{3} + \partial_{2} [\partial_{k} b_{1}\partial_{k} b_{3} ] \partial_{l}^{2} [-\partial_{1} b_{1} - \partial_{2} b_{2} ] \nonumber \\
& \hspace{9mm} - \partial_{1} [\partial_{k} b_{2} \partial_{k} [-\partial_{1} b_{1} - \partial_{2} b_{2} ]] \partial_{l}^{2} b_{3} - \partial_{1} [\partial_{k} b_{2}\partial_{k} b_{3} ] \partial_{l}^{2} [-\partial_{1} b_{1} - \partial_{2} b_{2} ] dx \nonumber \\
\lesssim& \int_{\mathbb{R}^{3}} \lvert \nabla^{2} b_{h} \rvert^{2} \lvert \nabla^{2} b_{v} \rvert + \lvert \nabla^{2} b_{h} \rvert \lvert \nabla^{3} b_{h} \rvert \lvert \nabla b_{v} \rvert  + \lvert \nabla b_{h} \rvert \lvert \nabla^{3} b_{h} \rvert \lvert \nabla^{2} b_{v} \rvert dx. \label{est 103}
\end{align}
Third, because $\sum_{l\in \{2,6\}} \RomanI_{1,3,l}$ consists of $b_{3}$ twice, similarly to $\sum_{l \in \{3,7 \}} \RomanI_{1,3,l}$, we integrate by parts and use divergence-free property of $\partial_{3} = -\partial_{1} b_{1} - \partial_{2} b_{2}$  to deduce 
\begin{align}
\sum_{l\in \{2,6\}} \RomanI_{1,3,l} \overset{\eqref{est 6}}{=}& 2 \sum_{k,l=1}^{3} \int_{\mathbb{R}^{3}} \partial_{k} b_{3} \partial_{k} \partial_{1} b_{3} \partial_{l}^{2} \partial_{3} b_{2} - \partial_{k} b_{3} \partial_{k}\partial_{2} b_{3} \partial_{l}^{2} \partial_{3} b_{1} dx \nonumber  \\
=& -2 \sum_{k,l=1}^{3} \int_{\mathbb{R}^{3}} \partial_{3} (\partial_{k} b_{3} \partial_{k} \partial_{1} b_{3}) \partial_{l}^{2} b_{2} - \partial_{3} (\partial_{k} b_{3} \partial_{k} \partial_{2} b_{3}) \partial_{l}^{2} b_{1} dx \nonumber \\
=& -2\sum_{k,l=1}^{3} \int_{\mathbb{R}^{3}} (\partial_{k}\partial_{3} b_{3} \partial_{k} \partial_{1} b_{3} + \partial_{k} b_{3}\partial_{k}\partial_{1} \partial_{3} b_{3}) \partial_{l}^{2} b_{2} \nonumber \\
& \hspace{10mm} - (\partial_{k} \partial_{3} b_{3} \partial_{k} \partial_{2} b_{3} + \partial_{k} b_{3} \partial_{k} \partial_{2} \partial_{3} b_{3}) \partial_{l}^{2} b_{1} dx \nonumber \\
=& -2 \sum_{k,l=1}^{3} \int_{\mathbb{R}^{3}} [\partial_{k} ( -\partial_{1} b_{1} - \partial_{2} b_{2}) \partial_{k} \partial_{1} b_{3} + \partial_{k} b_{3} \partial_{k} \partial_{1} (-\partial_{1} b_{1} - \partial_{2} b_{2} ) ] \partial_{l}^{2} b_{2} \nonumber \\
& \hspace{10mm} - [\partial_{k} (-\partial_{1} b_{1} - \partial_{2} b_{2}) \partial_{k} \partial_{2} b_{3} + \partial_{k} b_{3} \partial_{k} \partial_{2} (-\partial_{1} b_{1} - \partial_{2} b_{2} ) ] \partial_{l}^{2} b_{1} dx \nonumber \\
\lesssim& \int_{\mathbb{R}^{3}} ( \lvert \nabla^{2} b_{h} \rvert \lvert \nabla^{2} b_{v} \rvert + \lvert \nabla b_{v} \rvert \lvert \nabla^{3} b_{h} \rvert) \lvert \nabla^{2} b_{h} \rvert dx. \label{est 104}
\end{align}
Fourth, we integrate by parts and deduce 
\begin{align}
\sum_{l \in \{1,3\}} \RomanI_{2,5,l} \overset{\eqref{est 8} }{=}& -2 \sum_{k,l=1}^{3} \int_{\mathbb{R}^{3}} \partial_{k} b_{2} \partial_{k} \partial_{1} b_{2} \partial_{l}^{2} \partial_{2} b_{3} - \partial_{k} b_{2} \partial_{k} \partial_{2} b_{1} \partial_{l}^{2} \partial_{2} b_{3} dx \nonumber \\
=&2 \sum_{k,l=1}^{3} \int_{\mathbb{R}^{3}} \partial_{2} (\partial_{k} b_{2} \partial_{k} \partial_{1} b_{2}) \partial_{l}^{2} b_{3} - \partial_{2} (\partial_{k} b_{2} \partial_{k} \partial_{2} b_{1}) \partial_{l}^{2} b_{3} dx \nonumber \\
\lesssim& \int_{\mathbb{R}^{3}} ( \lvert \nabla^{2} b_{h} \rvert^{2} + \lvert \nabla b_{h} \rvert \lvert \nabla^{3} b_{h} \rvert ) \lvert \nabla^{2} b_{v} \rvert dx.\label{est 105}
\end{align}
Fifth, we immediately bound 
\begin{align}
\sum_{l \in \{4,8 \}} \RomanI_{2,5,l} \overset{\eqref{est 8}}{=}& -2 \sum_{k,l=1}^{3} \int_{\mathbb{R}^{3}} \partial_{k} b_{2} \partial_{k} \partial_{2} b_{1} \partial_{l}^{2} \partial_{3} b_{2} - \partial_{k} b_{2} \partial_{k} \partial_{3} b_{2} \partial_{l}^{2} \partial_{2} b_{1} dx  \nonumber \\
\lesssim& \int_{\mathbb{R}^{3}} \lvert \nabla b_{h} \rvert \lvert \nabla^{2} b_{h} \rvert \lvert \nabla^{3} b_{h} \rvert dx. \label{est 106} 
\end{align}
Sixth, we immediately bound 
\begin{align}
\sum_{l \in \{5,6 \}} \RomanI_{2,5,l} \overset{\eqref{est 8}}{=}& 2 \sum_{k,l=1}^{3} \int_{\mathbb{R}^{3}} \partial_{k} b_{2} \partial_{k} \partial_{2} b_{3} \partial_{l}^{2} \partial_{1} b_{2} - \partial_{k} b_{2} \partial_{k} \partial_{2} b_{3} \partial_{l}^{2} \partial_{2} b_{1} dx \nonumber \\
\lesssim& \int_{\mathbb{R}^{3}} \lvert \nabla b_{h} \rvert \lvert \nabla^{3} b_{h} \rvert \lvert \nabla^{2} b_{v} \rvert dx.\label{est 107} 
\end{align}
Seventh, we integrate by parts and bound 
\begin{align}
\sum_{l \in \{1,3 \}} \RomanI_{4,6,l}  \overset{\eqref{est 10}}{=}& - 2 \sum_{k,l=1}^{3} \int_{\mathbb{R}^{3}} \partial_{k} b_{1} \partial_{k} \partial_{1} b_{2} \partial_{l}^{2} \partial_{1} b_{3} - \partial_{k} b_{1} \partial_{k} \partial_{2} b_{1} \partial_{l}^{2} \partial_{1} b_{3} dx \nonumber \\
=& 2 \sum_{k,l=1}^{3} \int_{\mathbb{R}^{3}} \partial_{1} (\partial_{k} b_{1} \partial_{k} \partial_{1} b_{2}) \partial_{l}^{2} b_{3} - \partial_{1} (\partial_{k} b_{1} \partial_{k} \partial_{2} b_{1}) \partial_{l}^{2} b_{3} dx \nonumber  \\
\lesssim& \int_{\mathbb{R}^{3}} ( \lvert \nabla^{2} b_{h} \rvert^{2} + \lvert \nabla b_{h} \rvert \lvert \nabla^{3} b_{h} \rvert) \lvert \nabla^{2} b_{v} \rvert dx.  \label{est 108} 
\end{align}
Eighth, we immediately bound 
\begin{align}
\sum_{l \in \{2,7 \}} \RomanI_{4,6,l} \overset{\eqref{est 10}}{=}& 2 \sum_{k,l=1}^{3} \int_{\mathbb{R}^{3}} \partial_{k} b_{1} \partial_{k} \partial_{1} b_{2} \partial_{l}^{2} \partial_{3} b_{1} - \partial_{k} b_{1} \partial_{k} \partial_{3} b_{1} \partial_{l}^{2} \partial_{1} b_{2} dx \nonumber \\
\lesssim& \int_{\mathbb{R}^{3}} \lvert \nabla b_{h} \rvert \lvert \nabla^{2} b_{h} \rvert \lvert \nabla^{3} b_{h} \rvert dx. \label{est 109} 
\end{align}
Ninth, we immediately bound 
\begin{align}
\sum_{l \in \{5,6 \}} \RomanI_{4,6,l} \overset{\eqref{est 10}}{=}& 2 \sum_{k,l=1}^{3} \int_{\mathbb{R}^{3}} \partial_{k} b_{1} \partial_{k} \partial_{1} b_{3} \partial_{l}^{2} \partial_{1} b_{2} - \partial_{k} b_{1} \partial_{k} \partial_{1} b_{3} \partial_{l}^{2} \partial_{2} b_{1} dx \nonumber \\
\lesssim& \int_{\mathbb{R}^{3}} \lvert \nabla b_{h} \rvert \lvert \nabla^{2} b_{v} \rvert \lvert \nabla^{3} b_{h} \rvert dx. \label{est 110}
\end{align}
Tenth, because $\sum_{ l \in \{2,6 \}} \RomanII_{1,3,l}$ consists of $b_{3}$ twice, similarly to $\sum_{l \in \{3,7 \}} \RomanI_{1,3,l}$, we integrate by parts to deduce and use divergence-free property of $\partial_{3} = -\partial_{1} b_{1} - \partial_{2} b_{2}$  to deduce 
\begin{align}
\sum_{ l \in \{2,6 \}} \RomanII_{1,3,l}  \overset{\eqref{est 14} }{=}& -\sum_{k,l=1}^{3} \int_{\mathbb{R}^{3}} \partial_{k}^{2} \partial_{l} b_{3} \partial_{1} b_{3} \partial_{l} \partial_{3} b_{2} - \partial_{k}^{2}\partial_{l} b_{3} \partial_{2} b_{3} \partial_{l} \partial_{3} b_{1} dx \nonumber \\
=& \sum_{k,l=1}^{3} \int_{\mathbb{R}^{3}} ( \partial_{k}^{2} \partial_{l} \partial_{3} b_{3} \partial_{1} b_{3} + \partial_{k}^{2} \partial_{l} b_{3} \partial_{1} \partial_{3} b_{3}) \partial_{l} b_{2} \nonumber \\
& \hspace{5mm} - (\partial_{k}^{2} \partial_{l} \partial_{3} b_{3} \partial_{2} b_{3} + \partial_{k}^{2} \partial_{l} b_{3} \partial_{2} \partial_{3} b_{3}) \partial_{l} b_{1} dx \nonumber  \\
=& \sum_{k,l=1}^{3} \int_{\mathbb{R}^{3}} \partial_{k}^{2} \partial_{l} (-\partial_{1} b_{1} - \partial_{2} b_{2}) \partial_{1} b_{3} \partial_{l} b_{2} + \partial_{k}^{2} \partial_{l} b_{3} \partial_{1} (-\partial_{1} b_{1} - \partial_{2} b_{2}) \partial_{l} b_{2} \nonumber \\
& \hspace{5mm} - \partial_{k}^{2} \partial_{l} (-\partial_{1} b_{1} - \partial_{2} b_{2}) \partial_{2} b_{3} \partial_{l} b_{1} - \partial_{k}^{2} \partial_{l} b_{3} \partial_{2} (-\partial_{1} b_{1} - \partial_{2} b_{2}) \partial_{l} b_{1} dx \label{est 111} 
\end{align}
and integrate by parts once more to conclude 
\begin{align}
\sum_{ l \in \{2,6 \}} &\RomanII_{1,3,l}  \overset{\eqref{est 111}}{=} - \sum_{k,l=1}^{3} \int_{\mathbb{R}^{3}} \partial_{k}^{2} (-\partial_{1} b_{1} - \partial_{2} b_{2}) \partial_{l} (\partial_{1} b_{3} \partial_{l} b_{2}) + \partial_{k}^{2} b_{3} \partial_{l} (\partial_{1} [-\partial_{1} b_{1} - \partial_{2} b_{2} ] \partial_{l} b_{2}) \nonumber \\
& \hspace{5mm}  - \partial_{k}^{2} (-\partial_{1} b_{1} - \partial_{2} b_{2} ) \partial_{l} (\partial_{2} b_{3} \partial_{l} b_{1}) - \partial_{k}^{2} b_{3} \partial_{l} (\partial_{2} [-\partial_{1} b_{1} - \partial_{2} b_{2} ] \partial_{l} b_{1}) dx \nonumber \\
& \hspace{2mm} \lesssim \int_{\mathbb{R}^{3}} \lvert \nabla^{3} b_{h} \rvert ( \lvert \nabla^{2} b_{v} \rvert \lvert \nabla b_{h} \rvert + \lvert \nabla b_{v} \rvert \lvert \nabla^{2} b_{h} \rvert ) + \lvert \nabla^{2} b_{v} \rvert ( \lvert \nabla^{3} b_{h} \rvert \lvert \nabla b_{h} \rvert + \lvert \nabla^{2} b_{h} \rvert^{2} ) dx. \label{est 112}
\end{align}
Eleventh, we integrate by parts and bound 
\begin{align}
\sum_{l \in \{4,8 \}} \RomanII_{1,3,l} \overset{\eqref{est 14}}{=}& \sum_{k,l=1}^{3} \int_{\mathbb{R}^{3}} \partial_{k}^{2} \partial_{l} b_{3} \partial_{3} b_{1} \partial_{l} \partial_{3} b_{2} - \partial_{k}^{2} \partial_{l} b_{3} \partial_{3} b_{2} \partial_{l} \partial_{3} b_{1} dx \nonumber \\
=& - \sum_{k,l=1}^{3} \int_{\mathbb{R}^{3}} \partial_{k}^{2} b_{3} \partial_{l} (\partial_{3} b_{1} \partial_{l} \partial_{3} b_{2}) - \partial_{k}^{2} b_{3} \partial_{l} (\partial_{3} b_{2} \partial_{l} \partial_{3} b_{1}) dx \nonumber \\
\lesssim& \int_{\mathbb{R}^{3}} \lvert \nabla^{2} b_{v} \rvert ( \lvert \nabla^{2} b_{h} \rvert^{2} + \lvert \nabla b_{h} \rvert \lvert \nabla^{3} b_{h} \rvert) dx. \label{est 113}
\end{align}
Twelfth, we work on another non-trivial group $\sum_{ l \in \{3,7 \}} \RomanII_{1,3,l}$: 
\begin{equation}\label{est 134} 
\sum_{l \in \{3,7 \}} \RomanII_{1,3,l} \overset{\eqref{est 14}}{=} -\sum_{k,l=1}^{3} \int_{\mathbb{R}^{3}} \partial_{k}^{2} \partial_{l} b_{3} \partial_{3} b_{1} \partial_{l} \partial_{2} b_{3} - \partial_{k}^{2} \partial_{l} b_{3} \partial_{3} b_{2}\partial_{l} \partial_{1} b_{3} dx.
\end{equation} 
The difficulty here is that not only does it consist of $b_{3}$ twice, if we integrate to shift $\partial_{3}$ from $\partial_{3} b_{1}$ and $\partial_{3} b_{2}$ therein, then we end up respectively with $b_{1}$ and $b_{2}$ with no derivatives which will make subsequent estimates too difficult. Instead, we first 
integrate by parts to shift $\partial_{k}$ therein to deduce
\begin{align}
\sum_{l \in \{3,7 \}} \RomanII_{1,3,l} \overset{\eqref{est 134}}{=}& \sum_{k,l=1}^{3} \int_{\mathbb{R}^{3}} \partial_{k} \partial_{l} b_{3} \partial_{k} \partial_{3} b_{1} \partial_{l} \partial_{2} b_{3} + \partial_{k} \partial_{l} b_{3} \partial_{3} b_{1} \partial_{k} \partial_{l} \partial_{2} b_{3} \nonumber \\
& \hspace{5mm} - \partial_{k}\partial_{l} b_{3} \partial_{k} \partial_{3} b_{2} \partial_{l} \partial_{1} b_{3} - \partial_{k} \partial_{l} b_{3} \partial_{3} b_{2} \partial_{k} \partial_{l} \partial_{1} b_{3} dx. \label{est 114} 
\end{align}
Now we integrate by parts to shift $\partial_{3}$ of $\partial_{k} \partial_{3} b_{1}$ in the first and $\partial_{k} \partial_{3} b_{2}$ in the third terms while make a square in the second and fourth terms to write 
\begin{align}
\sum_{l \in \{3,7 \}} \RomanII_{1,3,l}  \overset{\eqref{est 114}}{=}& - \sum_{k,l=1}^{3} \int_{\mathbb{R}^{3}} \partial_{k} b_{1} \partial_{3} (\partial_{k} \partial_{l} b_{3} \partial_{l} \partial_{2} b_{3}) - \frac{1}{2} \partial_{2} (\partial_{k} \partial_{l} b_{3})^{2} \partial_{3} b_{1} \nonumber \\
& \hspace{5mm} - \partial_{k} b_{2}\partial_{3} (\partial_{k} \partial_{l} b_{3}\partial_{l} \partial_{1} b_{3}) + \frac{1}{2} \partial_{1} (\partial_{k} \partial_{l} b_{3})^{2} \partial_{3} b_{2} dx. \label{est 115} 
\end{align}
We integrate by parts twice in the second and fourth terms, and use the divergence-free property so that $\partial_{3} b_{3} = -\partial_{1} b_{1} - \partial_{2} b_{2}$ to conclude 
\begin{align}
\sum_{l \in \{3,7 \}} \RomanII_{1,3,l} \overset{\eqref{est 115}}{=}& - \sum_{k,l=1}^{3} \int_{\mathbb{R}^{3}} \partial_{k} b_{1} ( \partial_{k} \partial_{l} \partial_{3} b_{3} \partial_{l} \partial_{2} b_{3} + \partial_{k} \partial_{l} b_{3} \partial_{l} \partial_{2} \partial_{3} b_{3}) - \frac{1}{2} \partial_{3} (\partial_{k} \partial_{l} b_{3})^{2} \partial_{2} b_{1} \nonumber \\
& \hspace{5mm} - \partial_{k} b_{2} ( \partial_{k} \partial_{l} \partial_{3} b_{3} \partial_{l} \partial_{1} b_{3} + \partial_{k} \partial_{l} b_{3} \partial_{l} \partial_{1} \partial_{3} b_{3}) + \frac{1}{2} \partial_{3} (\partial_{k} \partial_{l} b_{3})^{2} \partial_{1} b_{2} dx \nonumber \\
=& -\sum_{k,l=1}^{3} \int_{\mathbb{R}^{3}} \partial_{k} b_{1} \partial_{k} \partial_{l} (-\partial_{1} b_{1} - \partial_{2} b_{2}) \partial_{l} \partial_{2} b_{3} + \partial_{k} b_{1} \partial_{k} \partial_{l} b_{3} \partial_{l} \partial_{2} (-\partial_{1} b_{1} - \partial_{2} b_{2}) \nonumber \\
& \hspace{5mm} -\partial_{k} \partial_{l} b_{3} \partial_{k} \partial_{l} (-\partial_{1} b_{1} - \partial_{2} b_{2}) \partial_{2} b_{1} - \partial_{k} b_{2} \partial_{k} \partial_{l} ( -\partial_{1} b_{1} -\partial_{2} b_{2}) \partial_{l} \partial_{1} b_{3} \nonumber \\
& \hspace{5mm} -\partial_{k} b_{2}\partial_{k}\partial_{l} b_{3} \partial_{l} \partial_{1} (-\partial_{1} b_{1} - \partial_{2} b_{2}) + \partial_{k} \partial_{l} b_{3} \partial_{k} \partial_{l} (-\partial_{1} b_{1} - \partial_{2} b_{2}) \partial_{1} b_{2} dx \nonumber \\
\lesssim& \int_{\mathbb{R}^{3}} \lvert \nabla b_{h} \rvert \lvert \nabla^{3} b_{h} \rvert \lvert \nabla^{2} b_{v} \rvert dx. \label{est 116}
\end{align}
The rest of the terms can be bounded immediately as follows:
\begin{subequations}\label{est 117} 
\begin{align}
\sum_{l \in \{4,8 \}} \RomanII_{2,5,l} \overset{\eqref{est 18} }{=}& \sum_{k,l=1}^{3} \int_{\mathbb{R}^{3}} \partial_{k}^{2} \partial_{l} b_{2} \partial_{2} b_{1} \partial_{l} \partial_{3} b_{2} - \partial_{k}^{2} \partial_{l} b_{2} \partial_{3} b_{2} \partial_{l} \partial_{2} b_{1} dx \nonumber \\
\lesssim& \int_{\mathbb{R}^{3}} \lvert \nabla^{3} b_{h} \rvert \lvert \nabla b_{h} \rvert \lvert \nabla^{2} b_{h} \rvert dx,  \\
\sum_{l\in \{5,6 \}} \RomanII_{2,5,l} \overset{\eqref{est 18}}{=}& - \sum_{k,l=1}^{3} \int_{\mathbb{R}^{3}} \partial_{k}^{2} \partial_{l} b_{2} \partial_{2} b_{3} \partial_{l} \partial_{1} b_{2} - \partial_{k}^{2} \partial_{l} b_{2} \partial_{2} b_{3} \partial_{l} \partial_{2} b_{1} dx \nonumber \\
\lesssim& \int_{\mathbb{R}^{3}} \lvert \nabla^{3} b_{h} \rvert \lvert \nabla b_{v} \rvert \lvert \nabla^{2} b_{h} \rvert dx, \\
\sum_{l \in \{1,3\}} \RomanII_{2,5,l} \overset{\eqref{est 18} }{=}& \sum_{k,l=1}^{3} \int_{\mathbb{R}^{3}} \partial_{k}^{2} \partial_{l} b_{2} \partial_{1} b_{2} \partial_{l} \partial_{2} b_{3} - \partial_{k}^{2}\partial_{l} b_{2} \partial_{2} b_{1} \partial_{l} \partial_{2} b_{3} dx \nonumber \\
\lesssim& \int_{\mathbb{R}^{3}} \lvert \nabla^{3} b_{h} \rvert \lvert \nabla b_{h} \rvert \lvert \nabla^{2} b_{v} \rvert dx, \\
\sum_{l \in \{2,7 \}} \RomanII_{4,6,l} \overset{ \eqref{est 22} }{=}& - \sum_{k,l=1}^{3} \int_{\mathbb{R}^{3}} \partial_{k}^{2} \partial_{l} b_{1} \partial_{1} b_{2} \partial_{l} \partial_{3} b_{1} - \partial_{k}^{2} \partial_{l} b_{1} \partial_{3} b_{1} \partial_{l} \partial_{1} b_{2} dx \nonumber \\
\lesssim& \int_{\mathbb{R}^{3}} \lvert \nabla^{3} b_{h} \rvert \lvert \nabla b_{h} \rvert \lvert \nabla^{2} b_{h} \rvert dx, \\
\sum_{l \in \{5,6 \}} \RomanII_{4,6,l} \overset{\eqref{est 22}}{=}& - \sum_{k,l=1}^{3} \int_{\mathbb{R}^{3}} \partial_{k}^{2} \partial_{l} b_{1} \partial_{1} b_{3}\partial_{l} \partial_{1} b_{2} - \partial_{k}^{2} \partial_{l} b_{1} \partial_{1} b_{3} \partial_{l} \partial_{2} b_{1} dx \nonumber \\
\lesssim& \int_{\mathbb{R}^{3}} \lvert \nabla^{3} b_{h} \rvert \lvert \nabla b_{v} \rvert \lvert \nabla^{2} b_{h} \rvert dx, \\
\sum_{l\in \{1,3\}} \RomanII_{4,6,l} \overset{\eqref{est 22}}{=}& \sum_{k,l=1}^{3} \int_{\mathbb{R}^{3}} \partial_{k}^{2}\partial_{l} b_{1} \partial_{1} b_{2} \partial_{l} \partial_{1} b_{3} - \partial_{k}^{2}\partial_{l} b_{1} \partial_{2} b_{1} \partial_{l} \partial_{1} b_{3} dx \nonumber \\
\lesssim& \int_{\mathbb{R}^{3}} \lvert \nabla^{3} b_{h} \rvert \lvert \nabla b_{h} \rvert \lvert \nabla^{2} b_{v} \rvert dx. 
\end{align}
\end{subequations} 
We apply \eqref{est 118}, \eqref{est 103}-\eqref{est 110}, \eqref{est 112}, \eqref{est 113}, \eqref{est 116}, and \eqref{est 117} in \eqref{est 133} and conclude \eqref{est 100}. 

\section*{Acknowledgments}
The second author gratefully acknowledges a grant from the Simons Foundation (962572, KY) and thanks Prof. Adam Larios, Prof. Jiahong Wu, and Prof. Theodore Drivas for valuable comments.

\end{document}